\documentclass[reqno,11pt]{amsart}
\usepackage[utf8x]{inputenc}
\usepackage[unicode]{hyperref}
\usepackage{mathrsfs,bbm}
\usepackage{graphicx}
\usepackage{a4wide}
\usepackage[footnotesize]{caption}

\usepackage[english]{babel}
\usepackage{verbatim}
\usepackage{supertabular}
\usepackage{stmaryrd}
\usepackage{longtable}
\usepackage{varioref}
\usepackage{calc}
\usepackage{ifthen}
\usepackage{indentfirst}
\usepackage{fancyhdr}
\usepackage{enumerate}
\usepackage{float}
\usepackage{color}
\usepackage{multicol}
\usepackage[BCOR4mm,DIV10]{typearea}
\usepackage[refpage]{nomencl/nomencl}
\usepackage{amssymb}
\usepackage{amsxtra}
\usepackage{amsmath}
\usepackage{amsfonts}
\usepackage{etoolbox}

\numberwithin{equation}{section}
\allowdisplaybreaks

\usepackage{wrapfig}
\usepackage{comment}
\usepackage[usenames,dvipsnames]{xcolor}
\usepackage{todonotes}

\usepackage{multicol}
\usepackage{chngcntr}

\counterwithin{equation}{section}

\newcommand{\omitted}[1]{}

\newcommand{\heikodetail}[1]{}
\newcommand{\hide}[1]{}

\makeatletter
\newcommand{\tpitchfork}{%
	\vbox{
		\baselineskip\z@skip
		\lineskip-.52ex
		\lineskiplimit\maxdimen
		\m@th
		\ialign{##\crcr\hidewidth\smash{$-$}\hidewidth\crcr$\pitchfork$\crcr}
	}%
}

\DeclareMathOperator{\dist}{dist}

\DeclareMathOperator{\graph}{graph}
\DeclareMathOperator{\lip}{Lip}

 \newcommand{\va}{\ang}
\newcommand\ang{\mathop{\mbox{$<\!\!\!)$}}\nolimits}

\newcommand{\R}{\mathbb{R}}
\newcommand{\N}{\mathbb{N}}

\renewcommand{\S}{\mathbb{S}}

\newcommand\Id{{{\rm Id}}}

\newcommand{\RM}{\mathscr{M}}
\newcommand{\HM}{\mathscr{H}}

\newcommand{\eps}{\varepsilon}
\renewcommand{\rho}{\varrho}
\newcommand{\Sigmaa}{\Sigma^\ast}

\newcommand{\Aam}{\mathscr{A}^m(\alpha,M)}  

\definecolor{mygreen}{RGB}{30,50,0}

\newtheorem{thm}{Theorem}[section]
\newtheorem*{thm*}{Theorem}
\newtheorem{prop}[thm]{Proposition}
\newtheorem{lem}[thm]{Lemma}
\newtheorem{cor}[thm]{Corollary}
\newtheorem{defin}[thm]{Definition}

\newtheorem*{introthm*}{Regularity Theorem}

\theoremstyle{definition}

\newtheorem{rem}[thm]{Remark}
\newtheorem*{rem*}{Remark}

\newcommand{\Fo}{\,\,\,\text{for }\,\,}
\newcommand{\Foa}{\,\,\,\text{for all }\,\,}

\newcommand{\With}{\,\,\,\text{with }\,\,}

\newcommand{\If}{\,\,\,\text{if }\,\,}

\newcommand{\AND}{\,\,\,\text{and }\,\,}

\def\li{\left(}
\def\ri{\right)}

\usepackage{tikz}
\usetikzlibrary{quotes,angles}
\usetikzlibrary{external}

\author{Bastian K\"afer}
\address[B.~K\"afer]{
	\newline{}
	Institut f{\"u}r Mathematik
	\newline{}
	RWTH Aachen University
	\newline{}
	Templergraben~55
	\newline{}
	D-52062 Aachen, Germany}
\email{kaefer@instmath.rwth-aachen.de}

\author{Heiko von~der~Mosel}
\address[H.~von~der~Mosel]{
	\newline{}
	Institut f{\"u}r Mathematik
	\newline{}
	RWTH Aachen University
	\newline{}
	Templergraben~55
	\newline{}
	D-52062 Aachen, Germany}
\email{heiko@instmath.rwth-aachen.de}


\date{\today}

\title[M{\"o}bius-invariant self-avoidance energies
]
{M{\"o}bius-invariant self-avoidance energies for non-smooth sets in arbitrary dimensions
}

\begin{document}
	\frenchspacing

	\begin{abstract}
In the present paper we investigate generalizations of O'Hara's
M\"obius
energy on curves \cite{ohara_1991a}, to M\"obius-invariant energies
on non-smooth subsets of
$\R^n$ of arbitrary dimension and co-dimension. In particular, we show under
mild assumptions on the local flatness of an admissible possibly unbounded
set $\Sigma\subset
\R^n$ that locally finite energy implies that $\Sigma$ is, in fact,
an embedded
Lipschitz submanifold of $\R^n$ -- sometimes even smoother (depending on the
a priorily given additional regularity of the admissible set). 
We also prove, on the other hand,
that a local graph structure of low fractional Sobolev regularity on
a set $\Sigma$ is already
    sufficient to guarantee finite energy of $\Sigma$.
       This type of Sobolev regularity is exactly what one would expect
 in view of Blatt's characterization \cite{blatt_2012a}
  of the correct energy space for the M\"obius energy on
  closed curves. Our results hold in particular for Kusner and Sullivan's
  cosine energy $E_\textnormal{KS}$
  \cite{kusner-sullivan_1997} since one of the energies
  considered here is equivalent to $E_\textnormal{KS}$.		
	\end{abstract}
	
	\maketitle


\section{Introduction} \label{sec:intro}
\subsection{Motivation and outline}\label{sec:1.1}
One of the most prominent examples of a repulsive energy on curves
is the M\"obius\footnote{This name originates in the energy's invariance under
M\"obius transformations of the ambient space 
as first shown in \cite[Theorem 2.1]{freedman-etal_1994}.}
energy introduced by J. O'Hara \cite{ohara_1991a}, which
can be written as
\begin{equation}\label{eq:moebius-curves}
E_{\textnormal{M\"ob}}(\gamma)=
\int\nolimits_\gamma\int\nolimits_\gamma \Big(\,\frac{1}{|x-y|^2}-
\frac{1}{d_\gamma(x,y)^2}\,\Big)\,dxdy,
\end{equation}
where $d_\gamma(x,y)$ denotes the intrinsic distance of two 
points $x,y\in\gamma$ along the curve $\gamma$. Ever since the seminal work
of M. Freedman, Zh.-Xu He, and Zh. Wang \cite{freedman-etal_1994}
it is clear that the M\"obius energy can be used as a fundamental
tool in Geometric Knot Theory, and since then a lot of geometric and
analytic work has been done.  Variational and gradient
formulas were derived and
analyzed  
\cite{he_2000,reiter_2012,blatt_2012a,ishizeki-nagasawa_2014,ishizeki-nagasawa_2015,ishizeki-nagasawa_2016,reiter-schumacher_2020}, the regularity
of minimizers and critical points was established in
\cite{freedman-etal_1994,he_2000, reiter_2012,blatt-etal_2016,blatt-vorderobermeier_2019},
and the $L^2$-gradient flow was studied in 
\cite{he_2000,blatt_2012b,blatt_2020a}. Various
discrete versions of the M\"obius energy were examined 
\cite{kusner-sullivan_1997,rawdon-worthington_2010,scholtes_2014b,blatt-etal_2019a,blatt-etal_2019b}, and one knows that the round circle is the
absolutely minimizing closed curve 
\cite{freedman-etal_1994,abrams-etal_2003}, whereas the
stereographic projection of the standard Hopf link uniquely minimizes
the corresponding version of the M\"obius energy on non-split links in $\R^3$
\cite{agol-etal_2016}.

Apart from the  very recent contribution by O'Hara on the self-repulsiveness
of Riesz potentials on smooth immersions \cite{ohara_2020}, 
comparatively little is known for 
higher-dimensional versions of the 
M\"obius energy such as the ones discussed in \cite{auckly-sadun_1997},
\cite{kusner-sullivan_1997}, and \cite{ohara-solanes_2018}.
It is the aim of this paper to 
initiate a systematic study of a family of M\"obius-invariant
energies on non-smooth subsets $\Sigma\subset\R^n$ of arbitrary dimension
and co-dimension. Such an investigation has been carried out
for other (not M\"obius-invariant)
self-avoiding energies such as integral Menger curvature
or tangent-point energies in arbitrary dimensions 
\cite{strzelecki-vdm_2011a,strzelecki-vdm_2013b,
kolasinski_2012,blatt_2013b,kolasinski-szumanska_2013,
kolasinski-etal_2013a,blatt-kolasinski_2012}

Following the ideas of R. B. Kusner and J. M. Sullivan
in \cite{kusner-sullivan_1997}
we first describe in an informal way how
to use only first order information encoded
in the class of admissible sets to define M\"obius-invariant
energies.  As observed by P. Doyle and O. Schramm \cite{kusner-sullivan_1997},
\cite[Chapter 3.4]{ohara_2003} the M\"obius
energy for closed curves $\gamma:
\S^1\to\R^n$
can be rewritten -- up to an additive constant --
in terms of the \emph{conformal angle}
$\vartheta_\gamma(x,y)$ as
\begin{equation}\label{eq:moebius-angle-curves}
E_{\textnormal{M\"ob}}(\gamma)=
\int_\gamma\int_\gamma\frac{1-\cos\vartheta_\gamma(x,y)}{|x-y|^2}\,dxdy.
\end{equation}
The conformal angle $\vartheta_\gamma(x,y)$ is defined as the angle between
the circle $\S^1(x,x,y)$ through the points $x,y\in\gamma$ and tangent
to $\gamma $ at $x$, and the circle $\S^1(y,y,x)$ also containing $x,y$ but 
now tangent to $\gamma$ at $y$. 

With that idea in mind  Kusner and  Sullivan \cite{kusner-sullivan_1997}
created M\"obius-invariant energies defined on embedded,
oriented
$m$-dimensional $C^1$-submanifolds $\mathscr{M}^m\subset\R^n$ of the
form
\begin{equation}\label{eq:moebius-angle-submanifolds}
E_L(\mathscr{M})=\int_\mathscr{M}\int_\mathscr{M}
\frac{L\big(\vartheta_\mathscr{M}(x,y)\big)}{|x-y|^{2m}}
\,d\textnormal{vol}_\mathscr{M}(x)d\textnormal{vol}_\mathscr{M}(y),
\end{equation}
where now the conformal angle $\vartheta_\mathscr{M}(x,y)$ is the angle
between the two unique 
$m$-dimensional spheres $\S^m(x,x,y)$ and
$\S^m(y,y,x)$, tangent to $\mathscr{M}$ at $x$ and $y$, respectively, and
both containing $x$ and $y$.  In principle, $L$ in the numerator of
the Lagrangian could be any non-negative function vanishing sufficiently
fast at zero to  balance the singularity of the denominator. Kusner and 
Sullivan, however, investigate more closely, in particular, numerically,
the specific energy $E_\textnormal{KS}:=E_{L_\textnormal{KS}}$ with the numerator 
\begin{equation}\label{eq:kusner-sullivan-integrand}
L_\textnormal{KS}(\vartheta):=(1-\cos\vartheta)^m.
\end{equation}
Notice that any choice of $L$ in \eqref{eq:moebius-angle-submanifolds}
requires first order information about the submanifold, and Kusner and Sullivan
provide in \cite[Section 11]{kusner-sullivan_1997} a convenient method to calculate the angle $\vartheta_\mathscr{M}(x,y)$
in terms of the tangent $m$-planes $T_x\mathscr{M}$ and $T_y\mathscr{M}$ by a 
simple reflection. 
We adopt this idea in Definition \ref{def:energies} below, but aiming at \emph{non-smooth} sets $\Sigma\subset\R^n$ we need to replace classic
tangent planes at points $p\in\Sigma$ by suitably approximating
$m$-planes
$H(p)$ that serve as ``mock tangent planes'',
similarly
as in previous collaborations of the second author
on various
geometric curvature energies 
\cite{strzelecki-vdm_2011a,strzelecki-vdm_2013b,kolasinski-etal_2013a}.
These mock tangent planes  enter our 
definition of admissible sets; see Definition \ref{def:admissible_sets}.
We then introduce in Definition \ref{def:energies}
a  family of M\"obius-invariant energies
$E^\tau\equiv E_{L_\tau}$ parametrized by a scalar $\tau\in\R$,
on non-smooth admissible sets $\Sigma\subset\R^n$
by replacing the numerator $L$ in \eqref{eq:moebius-angle-submanifolds} by
functions
$L_\tau=L_\tau(x,y,H(x),H(y))$ that roughly correspond to
$\tau$-dependent powers of the conformal angle $\vartheta_\Sigma(x,y)$. 
Instead of principal angles as in \cite{kusner-sullivan_1997}, however, 
we prefer to work with the \emph{angle metric}
\begin{equation}\label{eq:angle-metric}
\ang 
\big(F,G\big):=\|\Pi_F-\Pi_G\|\quad\textnormal{for two $m$-planes $F$, $G$},
\end{equation}
where $\Pi_F$ denotes the orthogonal projection onto the subspace $F$, and
$\|\cdot\|$  stands for the operator norm, so that our new
energies turn out to resemble the ``sin-energies'' that are discussed
only briefly
in \cite[Section 4]{kusner-sullivan_1997}. 
Nevertheless, we show in Appendix \ref{app:angles}
that the choice $\tau=1$ for our 
numerator $L_\tau$ generates
a Lagrangian bounded from above and below by constant multiples 
of $L_\textnormal{KS}$
in \eqref{eq:kusner-sullivan-integrand}. So all of the results described
in Section \ref{sec:1.2} below also hold for the specific 
M\"obius energy  $E_\textnormal{KS}$
studied
by Kusner and Sullivan.

\subsection{Main results}\label{sec:1.2}
Let $\mathscr{G}(n,m)$ denote the Grassmannian consisting of all $m$-di\-men\-sion\-al linear subspaces of $\R^n$ equipped with the angle metric defined in
\eqref{eq:angle-metric}. For $x\in\R^n$ and $F\in\mathscr{G}(n,m)$ the 
orthogonal projection onto the 
\emph{affine $m$-plane} $x+F$ is defined by
\begin{equation}\label{eq:affine-projection}
\Pi_{x+F}(z):=x+\Pi_F(z-x)\quad\Fo z\in\R^n.
\end{equation}
In addition,  let
\begin{equation}\label{eq:cone}
C_x(\beta,F):=\big\{z\in\R^n\colon \big|\Pi_{F^\perp}(z-x)\big|
\le\beta\,\big|\Pi_F(z-x)\big|\big\}
\end{equation}
be the cone around the affine $m$-plane $x+F$, centered at $x$ with 
opening angle $2\arctan \beta$. Throughout the paper, $B_r(x)$ 
denotes the open ball with radius $r>0$ centered at $x\in\R^n$.

We start with the definition of admissible sets.
\begin{defin}[Admissible sets] \label{def:admissible_sets}
Let $m,n \in \N$, $1\leq m \leq n$, $\alpha > 0$, $M>0$, and define the 
\emph{admissibility class $\Aam$} to be the set of all subsets 
$\Sigma \subset \R^n$ satisfying the following two properties.

\noindent
{\rm (i)}\,
$\Sigma$ is closed, 
and there exists a  function;
$H\colon\Sigma\to  \mathscr{G}(n,m)$. 

\noindent
{\rm (ii)}\,
There exists a dense subset $\Sigmaa \subset \Sigma$, with the following property:
For all compact sets $K\subset \Sigma$, there exist a radius $R_K\in (0,1]$ and a 
constant $c_K>0$, such that for all $p\in\Sigmaa \cap K$, 
there is a dense subset $D_p\subset (p+H(p)) \cap B_{R_K}(p)$, 
such that for all $x \in D_p$, there exists a point $\eta_x \in \Sigma \cap 
C_p(\alpha,H(p))$ 
with
             $  \Pi_{p+H(p)}(\eta_x)=x $,
	                and
\[\HM^m\left(E_{\alpha,M}(p) \cap B_r(\eta_x)\right) \geq c_K r^m 
\Foa r \in\left(0, {R_K}/{10^5}\right],\]
			                where
$ E_{\alpha,M}(p):=\left\{ \mu \in \Sigma : \ \ang(H(\mu),H(p)) < M\alpha \right\}.$ 
\end{defin}
As an immediate consequence of that definition we notice the monotonicity relation
\begin{equation}\label{eq:monotonicity}
\mathscr{A}^m(\alpha_1,M_1)\subset\mathscr{A}^m(\alpha_2,M_2)
\Foa 0<\alpha_1\le\alpha_2,\, 0<M_1\le M_2. 
\end{equation}
Intuitively speaking, an admissible set $\Sigma$ possesses a dense subset $\Sigma^*
$ of ``good points'' $p$ such that the conical portion $\Sigma\cap
C_p(\alpha,H(p))$ projects densely onto the affine plane $p+H(p)$ locally
near $p$,
and in each of the fibres under this projection there is
at least one $\Sigma$-point $q$ near which there is
sufficient mass of other $\Sigma$-points (possibly stretching
beyond the cone) whose
mock tangent
planes are close to $H(q)$. We should emphasize that we neither
require 
$\Sigma$  to be contained in the cone near $p$, nor a topological
linking condition as postulated, e.g., in the admissibility
class described in \cite[Section 2.3]{strzelecki-vdm_2013b}, nor do we assume
any relation 
between the mock tangent planes
$H(q)$ of various points $q\in\Sigma\cap
C_p(\alpha,H(p))$, or between $H(q)$ and $H(p)$. On the other hand, 
we do require
a \emph{uniform} local radius $R_K$
once we have fixed a compact subset $K\subset\Sigma$, which 
excludes some of the admissible example sets of previous work such as  in \cite[Example 2.14 \& Figure 1]{strzelecki-vdm_2013b}
with smaller and smaller structures accumulating locally.

\medskip

It is easy to see that embedded $m$-dimensional submanifolds 
$\Sigma:=\mathscr{M}^m
\subset\R^n$ of class $C^2$ without boundary
are admissible, since one can choose $H(p):=T_p\Sigma$ for all $p\in
\Sigma^*:=\mathscr{M}$, so that $\mathscr{M}\subset\mathscr{A}^m(\alpha,M)$ for all
$\alpha>0$ and $M>0$. Also finite unions $\Sigma:=\bigcup_{i=1}^N\mathscr{M}_i
$ of such $C^2$-submanifolds are admissible for any $\alpha,M>0$ (see Figure
\ref{fig:examples} a), since one
can define $H(p):=T_p\mathscr{M}_{i(p)}$, where $
i(p):=\min\{j\in\{1,\ldots,N\}:p\in\mathscr{M}_j\}. $
For these examples one can allow $c_K=\omega_m/2$ 
in Definition \ref{def:admissible_sets}, 
where $\omega_m$ denotes
the volume of the $m$-dimensional unit ball $B_1(0)\subset\R^m$.
Higher-dimensional sets that are foliated by lower-dimensional
$C^2$-submanifolds with
a uniform curvature bound can also be admissible, such  as the
 two-dimensional
planar
annulus  generated by uncountably many circles of varying radius depicted
in
Figure \ref{fig:examples} b.
Lipschitz submanifolds, countable collections of Lipschitz graphs, and
the images of compact $C^1$-manifolds under $C^1$-immersions, as well
as finite unions of those, also turn out to 
be admissible; the detailed proofs of these last statements are carried out 
in  Section \ref{sec:admissible_sets}.

It is possible to relax condition (ii) in Definition 
\ref{def:admissible_sets} requiring only sufficiently large projection
of $\Sigma\cap C_p(\alpha,H(p))$ onto an affine \emph{halfspace} $p+H_*(p)$,
at the cost of adding a condition relating different measures of flatness, as done
for the definition of $m$-fine sets in 
\cite{kolasinski_2011,kolasinski-etal_2013a}. 
This generates a new admissibility class $\mathscr{A}^m_*(\alpha,M)$ that 
also 
contains non-smooth examples with accumulation zones or certain unions
of manifolds with and without boundaries; see 
the bottom of Figure \ref{fig:examples}. All results mentioned below
also hold for this modified admissibility class $\mathscr{A}^m_*(\alpha,M)$;
for its precise definition and the necessary
modifications  in the proofs we refer to Remark \ref{rem:modified-admissible}
at the end of Section 
\ref{sec:lip-c1-mfds}.

\begin{figure}[!t]
	\begin{center}
		\begin{tikzpicture}[scale=0.15]
		
		\begin{scope}[shift={(-18,0)}]
		\node[inner sep=0pt] (russell) at (-11,10)
		{\includegraphics[width=.35\textwidth]{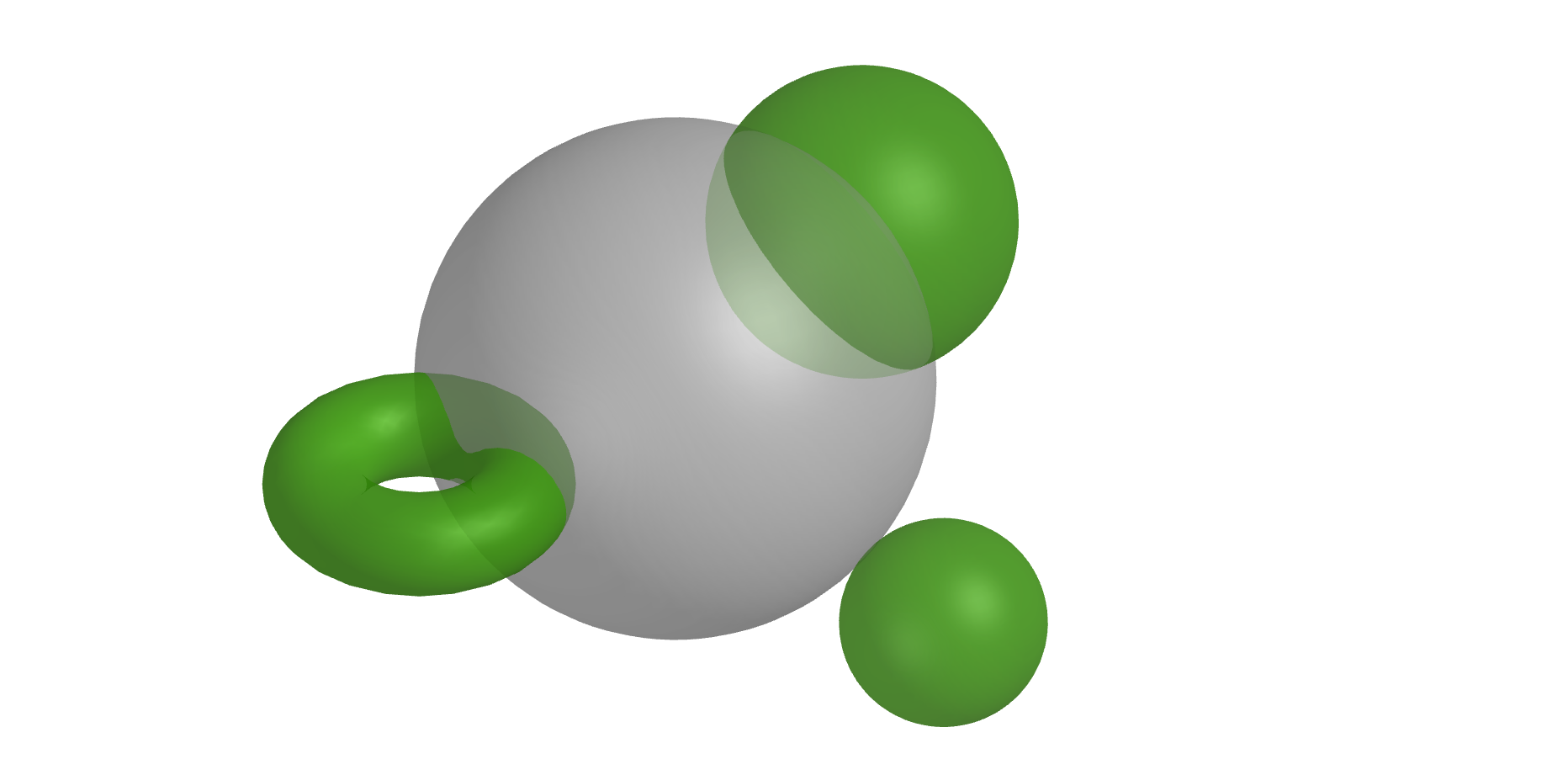}};
		\draw[color=black] (-10,2) node {\textbf{a}};
		\end{scope}
		
		
		\begin{scope}[shift={(5,0)}]
		\fill[black, opacity=0.2] (10,10) circle (6);
		\fill[white] (10,10) circle (3);
		\draw[black, thick] (10,10) circle (6);
		\draw[black, thick] (10,10) circle (3);
		\draw[color=black] (10,2) node {\textbf{b}};
		\end{scope}
		
		
		\begin{scope}[shift={(-10,0)}, xscale=1.8]
		
		\draw[black, thick] (0,-4) -- (-20,-4);
		\draw[black, thick] (0,-14) -- (-20,-14);
		\draw[black,thick, dotted] (0,-4) -- (4,-4); \draw[black,thick, dotted] (-24,-4) -- (-20,-4);
		\draw[black,thick, dotted] (0,-14) -- (4,-14); \draw[black,thick, dotted] (-24,-14) -- (-20,-14);
		\draw[black, thick] (-10,-4) -- (-10,-14);
		\draw[black, thick] (-18,-9) -- (-11,-9); \draw[black,thick, dotted] (-11,-9) -- (-9,-9); \draw[black, thick] (-9,-9) -- (-2,-9);
		\draw[black, thick] (-14,-6.5) -- (-11,-6.5); \draw[black,thick, dotted] (-11,-6.5) -- (-9,-6.5); \draw[black, thick] (-9,-6.5) -- (-6,-6.5);
		\draw[black, thick] (-14,-11.5) -- (-11,-11.5); \draw[black,thick, dotted] (-11,-11.5) -- (-9,-11.5); \draw[black, thick] (-9,-11.5) -- (-6,-11.5);
		\draw[black, thick] (-12,-5.25) -- (-11,-5.25); \draw[black,thick, dotted] (-11,-5.25) -- (-9,-5.25); \draw[black, thick] (-9,-5.25) -- (-8,-5.25);
		\draw[black, thick] (-12,-7.75) -- (-11,-7.75); \draw[black,thick, dotted] (-11,-7.75) -- (-9,-7.75); \draw[black, thick] (-9,-7.75) -- (-8,-7.75);
		\draw[black, thick] (-12,-10.25) -- (-11,-10.25); \draw[black,thick, dotted] (-11,-10.25) -- (-9,-10.25); \draw[black, thick] (-9,-10.25) -- (-8,-10.25);
		\draw[black, thick] (-12,-12.75) -- (-11,-12.75); \draw[black,thick, dotted] (-11,-12.75) -- (-9,-12.75); \draw[black, thick] (-9,-12.75) -- (-8,-12.75);
		\draw[black, thick] (-11,-4) -- (-11,-14); \draw[black, thick] (-12,-4) -- (-12,-14); \draw[black, thick] (-14,-4) -- (-14,-14); \draw[black, thick] (-18,-4) -- (-18,-14);
		\draw[black, thick] (-9,-4) -- (-9,-14); \draw[black, thick] (-8,-4) -- (-8,-14); \draw[black, thick] (-6,-4) -- (-6,-14); \draw[black, thick] (-2,-4) -- (-2,-14);
		
		\draw [red, fill] (-10,-7.05) circle (6pt);
		\draw [red, dashed] (-20,-13.26) -- (0,-0.76); \draw [red, dashed, opacity=0.8] (-20,-0.76) -- (0,-13.26);
		\draw [->,red,dashed] (-20.4,-1.5) arc (164:196:20) ;
		\draw[color=red]  (-22.5,-7.05) node [ fill=white, fill opacity=0.7, text opacity =1] {$2 \arctan \alpha$};
		\draw[color=black] (-10,-18) node {\textbf{c}};
		\end{scope}
		
		
		\begin{scope}[shift={(5,0)}]
		\draw[black, thick] (10,-4) circle (2);
		\draw[black, thick] (10,-6) -- (10,-12);
		\draw[black, thick] (10,-8) -- (6,-7);
		\draw[black, thick] (10,-8) -- (14,-7);
		\draw[black, thick] (10,-12) -- (6,-14);
		\draw[black, thick] (10,-12) -- (14,-14);
		\draw[black, thick] (14.986, -6.835) circle (1);
		\draw[black, thick] (5.083, -6.835) circle (1);
		\draw[black, thick] (5.051, -14.316) circle (1);
		\draw[black, thick] (14.948, -14.316) circle (1);
		\draw[color=black] (10,-18) node {\textbf{d}};
		\end{scope}
		
		\end{tikzpicture}
	\end{center}

\caption{\label{fig:examples}
{\bf Admissible sets.}  Finite unions of embedded $C^2$-submanifolds ({\bf a}) and
the annulus ({\bf b}) as the uncountable union of circles with positively bounded
radii are contained in $\mathscr{A}^m(\alpha,M)$. {\bf c.} An unbounded
set with fine structures accumulating within a compact
subset is in the modified class $\mathscr{A}^m_*(\alpha,M)$ for a certain $\alpha$.
{\bf d.} Finite unions of smooth manifolds with and without boundary may also 
be
in $\mathscr{A}^m_*(\alpha,M)$ for any positive $\alpha$ and $M$.
}
\end{figure}
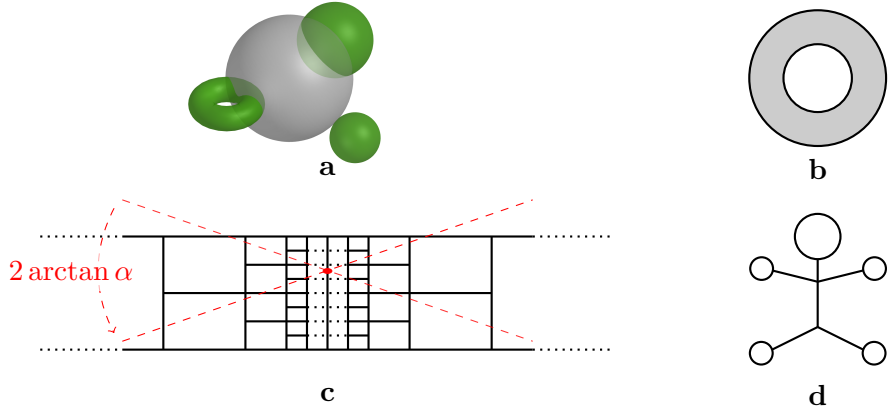

\smallskip

The
conformal angle between the tangential spheres $\S^m(x,x,y)$ and $\S^m(y,y,x)$
is  defined as the angle
between the spheres' tangent planes $T_z\S^m(x,x,y)$ and $T_z\S^m(y,y,x)$ at an \emph{arbitrary} point
$z$ contained in the intersection of the two spheres, so, e.g., at $z:=x$, where one has
 $T_x\S^m(x,x,y)=H(x)$. This leaves us to  compute the tangent plane $T_x\S^m(y,y,x)$ for which  it suffices to
 reflect $T_y\S^m(y,y,x)=H(y)$ at the subspace $(x-y)^\perp$ by virtue of the 
  mapping
 \begin{align}\label{eq:reflection}
\mathcal{R}_{xy}\colon \R^n\to\R^n,\ 
z\mapsto z-\frac 2{\vert x-y\vert^2}\langle z,x-y\rangle \cdot (x-y).
\end{align}
Therefore, we can express the conformal angle $\vartheta_\Sigma$
 of an admissible set $\Sigma$ as 
 $$
 \vartheta_\Sigma(x,y)=\ang\big(H(x),\mathcal{R}_{xy}(H(y))\big)=
 \ang\big(\mathcal{R}_{xy}(H(x)),H(y)\big).
 $$ 
 This 
 leads to the following definition of energies $E^\tau$.
\begin{defin}[M{\"o}bius energies]\label{def:energies}
Let $\alpha>0$ and $M>0$. 
For $\tau\in\R$ and $\Sigma\in\mathscr{A}^m(\alpha,M)$ we define
the energy
\begin{equation}\label{eq:energies}
E^\tau(\Sigma,H)\equiv E^\tau(\Sigma):=\int_\Sigma\int_\Sigma
\frac{\ang\big(\mathcal{R}_{xy}(H(x)),H(y)\big)^{(1+\tau)m}}{|x-y|^{2m}}\,d\HM^m(x)d\HM^m(y).
\end{equation}
\end{defin}
In Lemma \ref{lem:princ_angles} it is shown that the angle in the numerator coincides 
with the sine of
the largest principal angle, which is invariant under M\"obius transformations,
so that  according to
\cite[Section 2]{kusner-sullivan_1997}  all these energies $E^\tau$ are
M\"obius invariant. For $\tau=1$ we prove in Corollary \ref{cor:energy_comparison}
that $E^\tau=E^1$ is equivalent to the Kusner-Sullivan energy $E_\textnormal{KS}$
with the numerator \eqref{eq:kusner-sullivan-integrand}.

Since unbounded sets are not excluded in the admissibility class we can in general
not expect finite energy of the whole set. So, we say  \emph{$\Sigma$
has locally finite energy $E^\tau$} if and only if 
\begin{equation}\label{eq:locally-bdd-energy}
E^\tau(\Sigma\cap B_N(0))<\infty\quad\Foa N\in\N.
\end{equation}
For admissible sets with locally finite energy we prove
the following self-avoidance result, under a smallness condition 
on the product of the parameters $\alpha $ and $M$, which balances locally
the degree of flatness with the mass of regions of mildly varying mock
tangent planes.

\begin{thm}[Self-avoidance]\label{thm:self-avoidance}
For fixed dimensions $2\le m\le n$ there is a universal constant 
$\delta=\delta(m)$ such that for any $\alpha,M>0$ with
\begin{equation}\label{eq:alphaM}
\alpha (M+1) < \delta/50
\end{equation}
every admissible set $\Sigma\in
\mathscr{A}^m(\alpha,M)$
with locally finite M\"obius energy $E^\tau$, $\tau\in  (-1,\infty) $, is
an embedded Lipschitz submanifold of $\R^n$.
\end{thm}
In fact, we prove in Theorem \ref{thm:lipschitz_graph}
that every compact subset $K\subset\Sigma$ possesses a 
local graph representation by Lipschitz functions,
which according to 
\cite{naumann-simader_2007} is even slightly stronger than $\Sigma$
being a Lipschitz submanifold.

In Section \ref{subsec:immersion} we prove that the image
$\Sigma:=f(\mathscr{M})$ of a compact abstract $m$-dimensional
 $C^1$-manifold $\mathscr{M}$
 under an immersion $f:\mathscr{M}\to\R^n$
 is admissible for \emph{any} $\alpha>0$ and $M>0$. 
Consequently, one can guarantee that assumption \eqref{eq:alphaM} of
Theorem \ref{thm:self-avoidance}  holds true so that finite
M\"obius energy implies that $\Sigma$ is an embedded
$C^{0,1}$-submanifold of $\R^n$. But the initially granted regularity
of $\mathscr{M}$ and $f$ 
leads to the corresponding smoothness of the embedded
submanifold. 
\begin{cor}\label{cor:c1mfd}
Let $k\in\N$ and suppose 
$\Sigma=f(\mathscr{M})$ satisfies $E^\tau(\Sigma)<\infty$ for
$\tau\in (-1,\infty)$, where $\mathscr{M}$ is an $m$-dimensional compact 
$C^{k}$-manifold, 
and $f:\mathscr{M}\to\R^n$ is a $C^{k}$-immersion. Then $\Sigma$
is an embedded $m$-dimensional $C^{k}$-submanifold of $\R^n$.
\end{cor}
So,  finite M\"obius energy yields  embedded submanifolds, which
then inherit some additional presupposed regularity of the admissible set. 
This effect of a transferred initial regularity can also be observed
in the Lipschitz category; see Corollary \ref{cor:lip_collection}.

Very recently O'Hara proved the self-repulsiveness of the
Kusner-Sullivan energy  $E_\textnormal{KS}$ in the $C^2$-topology
on the class
of  embedded $C^2$-submanifolds with uniform curvature bounds; see
\cite[Theorem 3.3]{ohara_2020}. 
The equivalence of $E_\textnormal{KS}$ and $E^\tau$ for $\tau=1$ proven in
  Corollary \ref{cor:energy_comparison}
 in Appendix \ref{app:angles}, together with 
Corollary \ref{cor:c1mfd} might help to generalize
O'Hara's result to self-repulsiveness on suitably normalized
$C^1$-submanifolds in the $C^1$-topology.

\medskip

While Theorem \ref{thm:self-avoidance} 
and Corollary \ref{cor:c1mfd} describe self-avoidance effects
of the M\"obius energies $E^\tau$, one may ask, on the other hand,
under which 
regularity assumptions on (topological) embedded submanifolds one obtains 
finite energy. It is easy to show that 
$C^2$-regularity implies  finite
energy $E^\tau$ for any $\tau>0$; 
see Lemma \ref{lem:C2_sufficiency} and Corollary \ref{cor:embedded-c2}. 
But we prove, in addition, that a relatively mild fractional Sobolev
regularity of the local graph representations is already
sufficient to produce
finite energy. To state the precise result we 
recall the notion of these fractional
spaces (see, e.g., \cite{valdinoci-etal_2012}): For an open set $\Omega\subset\R^m$, and parameters
$k\in\N\cup\{0\}$, $s\in (0,1)$, and $\rho\in [1,\infty)$, the 
\emph{Sobolev-Slobodecki\v{\i}-space} $W^{k+s,\rho}(\Omega)$ is the set
of all Sobolev functions $f\in W^{k,p}(\Omega)$ such that 
$$
[\partial^\beta f]^\rho_{s,\rho}:=\int_\Omega\int_\Omega \frac{|\partial^\beta f(x)-\partial^\beta f(y)|^\rho}{|x-y|^{m+s\rho}}
\,dxdy<\infty \quad\textnormal{for all multi-indices $\beta$ with
$|\beta|=k$.}
$$
\begin{thm}[Sufficient fractional Sobolev regularity]
\label{thm:sufficient-sobolev}
If $\mathscr{M}^m\subset\R^n$ is an embedded compact submanifold
with local graph representations of class
$C^{0,1}\cap W^{\frac{2+\tau}{1+\tau},(1+\tau)m}$ for some
$\tau\in (0,\infty)$, then $E^\tau(\mathscr{M})<\infty.$
\end{thm}
By  Corollary \ref{cor:energy_comparison},  
 the Kusner-Sullivan energy $E_\textnormal{KS}$ is bounded from above
 by a constant multiple of $E^\tau$ for $\tau\in (0,1)$ and 
even equivalent to $E^1$, so that $C^{0,1}\cap W^{3/2,2m}$-regularity
suffices to guarantee
finite $E_\textnormal{KS}$. This fractional Sobolev regularity corresponds to the exact
regularity
that \emph{characterizes} finite M\"obius energy in dimension $m=1$ by
the work of S. Blatt
\cite[Theorem 1.1]{blatt_2012a}. For arbitrary dimensions
$m\ge 2$, however, Theorem \ref{thm:sufficient-sobolev}
establishes only \emph{one} direction of such a characterization.
It is open at this
point -- even for $\tau=1$ -- if the energies $E^\tau$ exhibit sufficiently
strong regularizing
effects\footnote{That the energies $E^\tau$ do regularize at least to a 
certain extent is reflected in the fact 
that certain types of singularities, like
a wedge-shaped crease, lead to infinite $E^\tau$-energy for any
$\tau>-1$; see Remark \ref{rem:wedge} in Appendix \ref{app:angles}.}
to guarantee that finite $E^\tau$-energy yields embedded
$W^{\frac{2+\tau}{1+\tau},(1+\tau)m}$-submanifolds, even if we add the
extra assumption that the admissible set is already an embedded 
$C^1$-submanifold.  Such a characterization, however, holds true in arbitrary 
dimensions for the scale-invariant tangent point energy introduced in
\cite{strzelecki-vdm_2013b} but analyzed there only in the regime \emph{above}
scale-invariance. In an upcoming note we prove the self-avoidance property
on a wider class of admissible non-smooth sets and use Blatt's technique
developed in \cite{blatt_2012a}  to prove the characterization of the exact
energy space for the scale-invariant tangent-point energy. It seems,
however, that this energy is slightly more singular than $E^1$.

\subsection{Strategy of proofs}\label{sec:1.3}
The crucial step to prove the self-avoidance property stated in Theorem 
\ref{thm:self-avoidance} is to bound for  any given 
$\delta\in (0,1)$ the \emph{beta number} 
\begin{equation}\label{eq:beta-number}
\beta_\Sigma(x,r):=
\inf_{F\in\mathscr{G}(n,m)}
\sup_{y\in \Sigma \cap  B_r(x)}\dist\big(y,(x+F)\cap B_r(x)\big)/r \quad\Fo x\in\Sigma
\end{equation}
from above by $\delta$  on small scales $r$.
This is carried out in Theorem \ref{thm:dist_sigma_to_plane} with an indirect argument as follows. Assuming
the contrary one finds points $p\in\Sigma$ and $q\in \Sigma\cap B_r(p)$ such 
that $\dist(q,p+H(p))>\delta r.$ This geometric situation contributes 
substantially to the energy in a way that 
depends on the angle $\ang\big(H(p),H(q)\big)$ between the mock tangent
planes at $p$ and $q$. If that angle happens to be small then, loosely
speaking, a sufficient amount of mass 
of $\Sigma$ near $q$   interacts a lot with strands of $\Sigma$
through suitable points contained in the cone $C_p(\alpha,H(p))$ near $p$,
because according to Part (ii) in
Definition \ref{def:admissible_sets} the 
deviation of the mock tangent planes from $H(p)$ is also small on these 
$\Sigma$-strands
near $p$. These ``almost parallel'' sheets of $\Sigma$ thus
generate a
certain quantum of energy basically through many almost identical but
mutually shifted
tangent-point spheres so that the respective conformal angles are large; see
Lemma \ref{lem:dist_sigma_to_plane_1}. 
If $\ang(H(p),H(q))$ is large, on the other hand,  then for each mock tangent plane at points near $q$ there is at least
one basis vector that deviates substantially 
from any basis of $H(p)$. That basis
vector can be used to define a controlled
macroscopic shift orthogonal to its projection
onto $H(p)$ to find sufficiently many pairs of tangent-point spheres
with a fairly large conformal angle; see Lemma 
\ref{lem:dist_sigma_to_plane_2}. 
It is interesting to note that the explicit estimates in those lemmas reveal 
that two close-by 
almost parallel sheets of $\Sigma$ seem to contribute a lot
more energy than transversal sheets close to self-intersection. 
Such a phenomenon was first observed for the suitably
desingularized M\"obius
energy on immersed planar curves ($m=1$) 
with self-intersections by R. Dunning \cite{dunning_2011}; see also the
work of D. Kube \cite{kube_2018} who derived a 
limit energy depending only on  the angle between two  self-intersecting arcs, 
and which is  
uniquely minimized by the intersection angle
$\pi/2$. Similarly,
 O'Hara observed in \cite[Section 3.1.2]{ohara_2020}
  different
energy contributions to the regularized Riesz energies comparing tangential
 with orthogonal self-intersections
of smooth surfaces ($m=2$).

Combining the bounds of the beta numbers with a uniform estimate on 
$\dist(\xi,\Sigma\cap B_r(p))$ for points $\xi\in (p+H(p))\cap B_r(p)$
which can actually
be derived for
\emph{all}  sets in $\Aam$, one establishes Reifenberg flatness of $\Sigma$ (Corollary
\ref{cor:dist_plane_to_sigma}). This implies
by virtue of Reifenberg's famous topological disk lemma 
\cite{reifenberg_1960,simon_1996a,hong-wang_2010} that $\Sigma$ is a 
topological manifold locally
by-H\"older homeomorphic to the open unit ball in $\R^m$ as stated in
Corollary \ref{cor:top_mfd}. But we do \emph{not} rely on Reifenberg's deep
result, we can take an easier and more direct route instead to
prove the better Lipschitz regularity of local graph representations
in Theorem \ref{thm:self-avoidance}.
For that it suffices to show
that 
the orthogonal projections
onto approximating planes restricted to sufficiently small balls  are bijective; see 
Lemmas \ref{lem:injective_projection} and \ref{lem:surjective_projection}.
This approach partly inspired by the proof of \cite[Proposition 9.1]{david-etal_2001} 
also leads to the improved $C^k$-regularity 
of the graph representations in Corollary \ref{cor:c1mfd}, as 
well as to improved Lipschitz constants as stated in Corollary 
\ref{cor:lip_collection}.

In Section \ref{sec:sufficient_regularity} we estimate the integrand,
first assuming
$C^2$-smoothness (Lemma \ref{lem:C2_sufficiency}), and then assuming only
fractional Sobolev regularity (Lemma \ref{lem:sobolev_sufficiency}), which
proves Theorem \ref{thm:sufficient-sobolev}.

  In Appendix \ref{app:angles} we
express the angle metric \eqref{eq:angle-metric} in terms of principal 
angles and demonstrate how our M\"obius-invariant energies relate to the ones
considered by Kusner and Sullivan, in particular to $E_\textnormal{KS}$; see
Corollary \ref{cor:energy_comparison}. 
In Appendix \ref{app:graphs} we prove various results on general Lipschitz graphs. Of particular and
independent interest is Lemma \ref{lem:intersect_lip_graphs}
stating in a quantitative way 
that the intersection of two $m$-dimensional Lipschitz graphs
in $\R^n$ is contained in a lower-dimensional Lipschitz
graph as long as the $m$-planar
domains of the graph functions
intersect in an angle that is sufficiently large compared to
the Lipschitz constants. This, in some way, generalizes the well-known
fact that the intersection of
two  transversal $C^1$-submanifolds  is a lower-dimensional $C^1$-submanifold, 
which is usually proven using the implicit function theorem; see, e.g.,
\cite[p. 30]{guillemin-pollack_1974}.


\section{Examples of admissible sets} \label{sec:admissible_sets}
We already mentioned in the introduction that the class of admissible sets
contains immersed compact $C^1$-manifolds as well as countable collections
of Lipschitz graphs, which we prove now. 

\subsection{Immersed compact $C^1$-manifolds} \label{subsec:immersion}
\begin{prop} \label{prop:imm_ex}
	Let $\RM$ be an $m$-dimensional, compact $C^1$-manifold and
	$f\colon \RM\to\R^n$  a $C^1$-immersion. Then, $\Sigma:=f(\RM)\in \Aam$ for all
 $\alpha>0$ and all $M>0$.
\end{prop}
The proof of this proposition  is based on the well-known fact that the 
images of sufficiently  small portions
of the manifold $\mathscr{M}$ under $f$ can be expressed as graph patches with
arbitrarily small $C^1$-norm; a detailed  proof of the following lemma
is carried out in
\cite[Section 4.3]{kaefer_2020}.
\begin{lem}[Local graph representation of immersed coordinate patches] \label{lem:imm_local_graph}
	Let $\RM$ be an $m$-dimensional $C^1$-manifold for $1\leq m\leq n$ and 
$f\in C^1\li\RM,\R^n\ri$ a $C^1$-immersion. 
Then for every $\beta>0$ and $x\in \RM$ there exist a radius $r_x=r_x(\beta,f,\RM)>0$ and a 
function $u_x \in C^1(T_xf,T_xf^\perp)$ 
satisfying $u_x(0)=0$, $Du_x(0)=0$, and
		\begin{align}\label{imm_C0_bound}
			\Vert Du_x\Vert_{C^0}< \beta,
		\end{align}	
	such that
		\begin{align}\label{imm_local_graph}
			f\li \mathscr {U}_{x,r_x}\ri= \big( f(x) + \graph u_x \big) \cap B_{r_x}(f(x)),
		\end{align}
		where we set $T_xf:=Df(x)(T_x\mathscr{M})$, and
		 \begin{align}\label{imm_U}
		 	\mathscr {U}_{x,r}\subset \RM
		 \end{align}
		 is defined as the largest connected open subset of the preimage 
$f^{-1}\li B_r(f(x))\ri$ containing 
the point $x \in \RM$.
\end{lem}

In order to define $H\colon \Sigma \to \mathscr{G}(n,m)$ when proving Proposition \ref{prop:imm_ex}
we use a covering 
with these local graph patches. Notice that $H$ depends on the choice of the covering and 
of the finite subcovering below, and also on the ordering of the finite index set. Any such choice will lead to an 
admissible set $\Sigma=f(\RM)$. Note that for compact $\RM$ the set $\Sigma$ is compact, hence closed.

\begin{proof}[Proof of Proposition \ref{prop:imm_ex}]
By virtue of the monotonicity \eqref{eq:monotonicity} 
we may assume $\alpha\leq 1$. 
Now, fix
\begin{align} \label{imm_beta}
\beta:=
M\alpha\cdot (M+1)^{-1}(\alpha+1)^{-1}/5
\end{align}
and consider the open covering 
$ \RM\subset \bigcup_{x\in\RM}\mathscr {U}_{x,r_x/4}$
for the sets $\mathscr{U}_{x,r_x/4}$ as defined in \eqref{imm_U} with positive radii $r_x=r_x(\beta,f,\RM)$. 
The manifold $\RM$ is  compact, so that there is a finite subcover
\begin{align}\label{imm_subcover}
\RM\subset \bigcup \nolimits_{i=1}^{N}\mathscr{U}_{x_i, {r_i}/4} \subset \RM
\end{align}
for distinct points $x_1,\dots,x_N\in \RM$ and radii $r_i:=r_{x_i}(\beta,f,\RM)>0$ for $i=1,\dots,N$. Set
\begin{align} \label{imm_R}
R:=R(\alpha,M,f,\RM):= \min \left\{r_1,\dots,r_N\right\}/4,
\end{align}
where we note that $R$ depends on $\alpha$ and $M$ via \eqref{imm_beta}. Observe that for any $y \in \RM$ there 
is at least one $k=k_y \in\{1,\dots,N\}$ such that $y\in \mathscr{U}_{x_k,r_k/4}$ implying $f(y) 
\in B_{{r_k}/4}\li f(x_k)\ri$; hence
$B_r(f(y))\subset B_{r_k}(f(x_k)) \Foa r \in (0,3R].$
Using \eqref{imm_local_graph} of Lemma \ref{lem:imm_local_graph} we therefore find
\begin{align}\label{imm_graph_uni}
f\li \mathscr{U}_{x_k,r_k}\ri \cap B_r(f(y)) = \li f(x_k) +\graph u_{x_k}\ri \cap B_r(f(y)) \Foa r \in (0,3R].
\end{align}
Define $H\colon \Sigma \to \mathscr{G}(n,m)$ as $H(p):=T_p(\,f(x_{i(p)})
+\graph u_{x_{i(p)}})$ for $p \in \Sigma$, where $i(p)$
 is the smallest index $i \in \{1,\dots,N\}$ such that $p\in f(\mathscr U_{x_i,r_i/4})$ which is well-defined
by virtue of  \eqref{imm_subcover}.
Denote $p_k:=f(x_k)$,  $F_k:=T_{x_k}f$, and $\Gamma_k:=p_k+\graph u_k$,
 and notice that 
$F_k=T_{p_k}(\Gamma_k)$ for $k=1,\dots,N$. Set $\Sigma^\ast:=\Sigma$ and fix any $p\in \Sigma$ and abbreviate 
$i:=i(p)\in \{1,\dots,N\}$, $u_i:=u_{x_{i(p)}}\colon F_i\to F_i^\perp$, so that by definition $H(p)=T_{p-p_i}\graph u_i$.
Now use \eqref{imm_graph_uni} for  $y \in \mathscr{U}_{x_i, {r_i}/4}$ with $f(y)=p$ and for the radius $r=3R$ to obtain
\begin{align}\label{imm_graph_p}
f\li\mathscr{U}_{x_i,r_i}\ri \cap B_{3R}(p)= \Gamma_i \cap B_{3R}(p).
\end{align}
In particular, $p=p_i+\xi+u_i(\xi)$ for some $\xi \in F_i$. Any other graph point $q=p_i+\mu+u_i(\mu)$ 
with $\mu\in F_i$, is contained in the cone
$C_p\li\beta,F_i\ri$
since \eqref{imm_C0_bound} implies
\begin{align}\label{eq:lip-cone-est}
\vert \Pi_{F_i^\perp}(q-p) \vert &=
\left \vert u_i(\mu)-u_i(\xi)\right \vert 
\leq \beta \vert \mu-\xi\vert = \beta \left \vert\Pi_{F_i}(q-p)\right\vert.
\end{align}
The Cone Lemma \ref{lem:cone} applied to $F:=F_i$, $G:=H(p)$, 
$\chi:=\beta$, $\sigma:=\beta$, and $\kappa:=\alpha$ 
implies that any such $q\in \Gamma_i$ is also contained in $C_p\li\alpha,H(p)\ri$ since 
Lemma \ref{lem:angle_lip} guarantees
\begin{align} \label{imm_angle}
\va\li H(p),F_i\ri =\va \li T_{p-p_i} \graph u_i ,T_{0}\graph u_i\ri \leq \left \Vert 
Du_i(\xi)-Du_i(0)\right \Vert_{C^0}\leq \beta
\end{align}
with $\beta < \frac \alpha 5 \leq \frac 1 5$ by \eqref{imm_beta} and, therefore, $\frac {\beta +
 (1+\beta)\beta}{1-(1+\beta)\beta}< \frac {11}{19}\alpha$. We deduce from \eqref{imm_graph_p}
\begin{align} \label{imm_graph_cone}
\Gamma_i \cap B_{3R}(p)=f\li \mathscr{U}_{x_i,r_i}\ri \cap B_{3R}(p)\subset \Sigma 
\cap C_p\li\alpha,H(p)\ri \cap B_{3R}(p).
\end{align}
Recall that $p=p_i+\xi+u_i(\xi)$ for $\xi\in F_i$, so that we can
use the Shifting Lemma \ref{lem:shift-graphs} to find the  shifted 
function $\tilde u \colon F_i\to F_i^\perp$  
 satisfying
 $\tilde u(0)=0$,  $\lip \tilde u=\lip u_i\leq \beta$, and
$ \Gamma_i = p+\graph \tilde u.$
Therefore, \eqref{imm_graph_cone} yields
\begin{align}\label{imm_graph_cone_tilde}
\Gamma_i \cap B_{3R}(p)=
( p+\graph \tilde u ) \cap B_{3R}(p)\subset \Sigma \cap C_p(\alpha,H(p)) \cap B_{3R}(p).
\end{align}
Applying the Tilting Lemma \ref{lem:tilting} to $u:=\tilde u$, $F:=F_i$, $G:=H(p)$ with $\va\li H(p),F_i\ri
 \leq \beta$ by \eqref{imm_angle}, so $\chi:=\beta$ and $\sigma=\chi (1+\lip u) \leq \beta (1+\beta ) <1$, 
we find
\begin{align} \label{imm_projection}
	B_{\frac{\li 1-\sigma\ri\rho}{\sqrt{1+\li\lip \tilde u\ri^2}}}(0)\cap H(p)\subset \Pi_{H(p)}\li 
\graph \tilde u \cap {B_\rho(0)}\ri \Foa \rho>0.
\end{align}
For the set $D_p:=B_{\frac {\li 1-\beta \li 1+\beta\ri \ri R}{\sqrt{1+\beta^2}}}(p) \cap \li p + H(p)\ri $ we obtain 
in
 particular, by \eqref{imm_graph_cone_tilde}
\begin{align}
D_p \subset \Pi_{p+H(p)}\big( \Gamma_i\cap {B_R(p)}       \big) \subset
\Pi_{p+H(p)}\big( \Sigma \cap C_p\li \alpha, H(p)\ri\big),
\end{align}
so that we can choose a uniform radius 
\begin{align*}
	R_K\equiv R_\Sigma \li \alpha,M,f,\RM\ri := \textstyle\min\Big\{\frac {1-\beta\li 1+\beta\ri}{\sqrt{1+\beta^2}}R,1\Big\} 
\quad\textnormal{for all compact $K\subset\Sigma$},
\end{align*}
where $R$ is defined as in \eqref{imm_R} and $\beta$ as in \eqref{imm_beta}.
In particular, for all $x \in D_p$ there exists a point 
\begin{equation}\label{imm_eta_x_inclusion}
\eta_x \in \Gamma_i\cap {B_R(p)} \subset \Sigma \cap C_p\li \alpha ,H(p)\ri \With
 	\Pi_{p+H(p)}(\eta_x)=x.
\end{equation}
Let $\mathscr{L}\subset \{1,\dots, N\}$ be the set of indices $l$ such that
$\va \li T_{0}\graph u_l, T_{0}\graph u_i\ri \geq  {M\alpha}/2.$
Notice, of course, $i\notin \mathscr L$. Then, by  Lemma \ref{lem:angle_lip} and \eqref{imm_beta}, one has
\begin{align}\label{imm_bad_angle}
\va \li T_{q-p_l} \graph u_l, T_{q-p_i}\graph u_i\ri &\geq \va\li T_{0}\graph u_l,T_{0}\graph u_i\ri \notag \\
&\hspace{-4cm}- \va \li T_{0}\graph u_l,T_{q-p_l}\graph u_l\ri - \va\li T_{0}\graph u_i,T_{q-p_i}\graph u_i\ri \notag\\
&\hspace{-6.0cm}\geq  {M\alpha}/2 - \left \Vert Du_l(0)-Du_l(\mu)\right \Vert_{C^0} - 
\left \Vert Du_i(0)-Du_i(\xi)\right \Vert_{C^0}
\geq  {M\alpha}/2-2\beta >0,
\end{align}
for all $l\in \mathscr L$ and all points $q=p_l+\mu+u_l(\mu)=p_i+\xi+u_i(\xi)$ contained in the
intersection $S_l:=\Gamma_l\cap\Gamma_i
$,
which means that the two $C^1$-graphs intersect transversally, so that 
$\dim_\HM(S_l)\leq m-1$ for all $l \in \mathscr L$ In particular, for any
 $x \in D_p$ with $\eta_x$ as in \eqref{imm_eta_x_inclusion}, one finds for all $r>0$
\begin{align}\label{imm_measure}
	\textstyle\HM^m\big(  \Gamma_i \cap B_r(\eta_x)  \setminus 
	\bigcup\nolimits_{l \in \mathscr L} 
\Gamma_l\big)=\HM^m\big( \Gamma_i \cap B_r(\eta_x)\big).
\end{align}
 Notice, finally, that all points $q\in \big(\Gamma_i \setminus
 \bigcup_{l\in\mathscr L}\Gamma_l\big) \cap B_{3R}$ are contained in the set 
$E_{\alpha,M}(p)$ introduced in Definition \ref{def:admissible_sets}. Indeed, by 
 \eqref{imm_graph_cone_tilde} any
such point $q$ is contained in $\Sigma$, and the
$m$-plane
$H(q)$ is contained in 
$\left\{ T_{q-p_k}\graph u_k \colon  
k\in\{1,\dots,N\}\setminus \mathscr L \right\}$, 
so that there is a $k=k(q)\in \{1,\dots,N\} \setminus \mathscr L$ with 
\begin{align*}
\va (H(q),H(p)) &\leq \va (H(q),T_{0}\graph u_k) + 
\va  (T_{0}\graph u_k,T_{0}\graph u_i) 
\\
&\hspace{1cm}
+ \va \li T_{0}\graph u_i,H(p)\ri\\
&\hspace{-2.5cm}<\left \Vert Du_k(\mu)-Du_k(0)\right \Vert_{C^0} + 
 {M\alpha}/2 + \left \Vert Du_i(0)-Du_i(\xi)\right 
\Vert_{C^0}
\leq 2\beta +  {M\alpha}/2 <M\alpha,
\end{align*}
where we wrote $q=p_k+\mu+u_k(\mu)$ for some $\mu\in F_k$, and as before, $p=p_i+\xi+u_i(\xi)$ for $\xi\in F_i$, and 
used Lemma \ref{lem:angle_lip} and \eqref{imm_beta}. Consequently,
\begin{align} \label{imm_inclusion_Eam}
	\textstyle\big( \Gamma_i \setminus \bigcup\nolimits_{l\in\mathscr L}
	\Gamma_l \big)
\cap B_{3R}(p)\subset E_{\alpha,M}(p)\cap B_{3R}(p).
\end{align}
Now, for $x \in D_p=\li p + H(p)\ri \cap B_{R_K}(p)$ with $\eta_x$ as in \eqref{imm_eta_x_inclusion} we estimate
\[\vert \eta_x-p\vert^2 = \vert \Pi_{H(p)}(\eta_x-p) \vert^2 + \vert
 \Pi_{H(p)^\perp}(\eta_x-p) \vert^2 \leq ( 1+\alpha^2) \vert x-p\vert^2, \]
so that 
$\vert \eta_x -p\vert <\sqrt{1+\alpha^2} R_K\le 2 R_K < 2R$, 
since we assumed $\alpha\le 1$. This implies,
by definition of $R_K$ in \eqref{imm_R},
\begin{align}\label{imm_inclusion}
	B_r(\eta_x)\subset B_{3R}(p) \Foa r \in \left(0,R_K\right].
\end{align}
Combining  \eqref{imm_inclusion_Eam}, \eqref{imm_measure}, and \eqref{imm_inclusion} we arrive at
\begin{align*}
&\HM^m\big( E_{\alpha,M}(p)\cap B_r(\eta_x)\big) \geq \HM^m\big( \big[ \Gamma_i \setminus \textstyle
\bigcup\nolimits_{l\in\mathscr L}\Gamma_l \big] \cap B_r(\eta_x)\big)
=\HM^m\big( \Gamma_i \cap B_r(\eta_x)\big)\\ 
&
\geq \HM^m\big( \Pi_{F_i}\li \Gamma_i \cap B_r(\eta_x)\ri \big)
 \geq \HM^m\big( 
B_{ r /{\sqrt{1+\beta^2}}}(0)\cap F_i\big)
=\omega_m\big( {r}/{\sqrt{1+\beta^2}}\big)^m
\end{align*}
for all $ r \in \left(0, R_K\right].$
Notice that we used \eqref{imm_eta_x_inclusion}, i.e.,  $\eta_x =p_i+x+u_i(x)$, so that any other point 
$q =p_i+\mu+u_i(\mu)\in \Gamma_i$ with
$\vert \mu - x\vert < r/ {\sqrt{1+\beta^2}}$ satisfies  
$q\in B_r(\eta_x)$, which implies
$B_{  {r}/{\sqrt{1+\beta^2}} }(x)\cap \li p_i+F_i\ri \subset \Pi_{F_i}\big( 
B_r(\eta_x)\cap \Gamma_i \big).$
\end{proof}

\subsection{Countable unions of Lipschitz graphs} \label{subsec:lipschitz}
The following considerations can also be localized to study countable unions of
pieces of Lipschitz graphs using similar arguments as in Section
\ref{subsec:immersion}. For simplicity we restrict here to collections of entire Lipschitz graphs.
\begin{prop} \label{lem:lip_admis}
	Suppose $
	\Sigma= \overline{ \bigcup_{i\in\N} \li p_i + \graph u_i\ri },$
	where $u_i \in C^{0,1}\li F_i,F_i^\perp\ri$, $F_i \in \mathscr{G}(n,m)$, $u_i(0)=0$, and $\lip u_i\leq \beta$ 
for all $i \in \N$. If
	\begin{align} \label{beta_lip_cond}
	0 \leq \beta \leq  M\alpha / (16(M+1))
	\end{align}
	for given $\alpha \in (0,1)$ and $M>0$, then one finds
$\Sigma \in \Aam.$
\end{prop}
Notice that the monotonicity property \eqref{eq:monotonicity} implies
that $\Sigma \in \Aam$ for \emph{all} $\alpha,M>0$ as 
long as $\beta$ satisfies \eqref{beta_lip_cond} with $\alpha$ replaced
by some constant $\tilde{\alpha}<\min\{\alpha,1\}$.

\begin{proof}
	Fix an arbitrary $m$-plane $F_0 \in \mathscr{G}(n,m)$ as a 
	"dummy plane" and
	set $\Gamma_i:=p_i+\graph u_i$ for $i\in\N.$ Notice that for 
every point $p$ contained in the union $\bigcup_{i=1}^\infty \Gamma_i$ there exists a unique 
smallest index $i(p)\in \N$, such that $
	p\in \Gamma_{i(p)},$
	that is, for every $1\leq j<i(p)$ with $i(p)>1$ one has
	$p\not \in \Gamma_j.$
	Now define the map $H\colon \Sigma \to \mathscr{G}(n,m)$ by setting $ H(p):=T_p \Gamma_{i(p)}$
	 if
$p=p_{i(p)}+\xi+u_{i(p)}(\xi)$ for some $\xi\in F_{i(p)}$ 
such that $Du_{i(p)}(\xi)$ exists. In all other cases set
$H(p):=F_0$, which happens either if $Du_{i(p)}(\xi)$ does \emph{not} exist or if $p$ is not contained in \emph{any} 
of the graphs $\Gamma_i$.
Let $\Sigmaa := \{ p \in \Sigma \ \colon \ 
p=p_{i(p)} + \xi + u_{i(p)}(\xi) 
\text{ and } Du_{i(p)}(\xi) 
\text{ exists} \}.$ 
Notice that $\Sigma$ is closed by definition, and that $\Sigmaa\subset \Sigma$ is dense, 
since for any $\eps>0$ and any $q\in\Sigma$, there is a point $q_\eps \in \bigcup_{i=1}^\infty \Gamma_i$, 
such that $\vert q-q_\eps\vert< \eps/ 2$. For $q_\eps$ there exists $i_\eps :=i(q_\eps)$ such that 
$q_\eps =p_{i_\eps}+x_\eps + u_{i_\eps}(x_\eps)$ for some $x_\eps \in F_{i_\eps}$. By Rademacher's Theorem applied to
the Lipschitz function $u_{i_{\eps}}$ there 
exists $\xi_\eps \in F_{i_\eps}$ such that $Du_{i_\eps}(\xi_\eps)$ exists and 
$\vert x_\eps - \xi_\eps\vert <  \eps/( {2\sqrt{1+\beta^2}})$ so that the corresponding graph point 
$\tilde q_\eps:=p_{i_\eps}+\xi_\eps+u_{i_\eps}(\xi_\eps)$ satisfies
	$ 
	\vert \tilde q_\eps-q\vert   <  \eps / 2 + 
\left \vert \xi_\eps-x_\eps + u_{i_\eps}(\xi_\eps)-u_{i_\eps}(x_\eps)\right\vert 
	<  \eps / 2 + \sqrt {1+\beta^2} \vert \xi_\eps - x_\eps \vert <\eps.
	$
	We may assume that $i_\eps = i(\tilde q_\eps)$, since if for every point 
$\xi \in F_{i_\eps}$ with $\vert x_\eps -\xi \vert < \eps/ (2\sqrt{1+\beta^2})$
such that 
$Du_{i_\eps}(\xi)$ exists, the corresponding graph point $\sigma:=p_{i_\eps}+\xi + u_{i_\eps}(\xi)$ 
had smallest index $i(\sigma)<i_\eps$, we could select a sequence $\xi_k \to x_\eps$ with graph points 
$\sigma_k:=p_{i_\eps}+\xi_k+u_{i_\eps}(\xi_k)$ satisfying $i(\sigma_k)=j$ for some fixed $1\leq j<i_\eps$. 
But this would imply $\Gamma_j \ni \sigma_k\to q_\eps$, 
so that $i(q_\eps)=j <i(q_\eps)$, 
which is a contradiction. Hence we have proved that $\Sigmaa\subset \Sigma$ is dense.
	
	To check the remaining conditions of Definition \ref{def:admissible_sets}, fix $p\in\Sigmaa$ and set 
$i:=i(p)$. Then $p=p_i+\xi+u_i(\xi)$ for some $\xi \in F_i$, and we can
use the Shifting Lemma \ref{lem:shift-graphs} to find  the   shifted function  
$\tilde u\in C^{0,1}(F_i,F_i^\perp)$ satisfying
	\begin{align} \label{lip_shifted}
	\Gamma_i = p + \graph \tilde u \quad\With \tilde u(0)=0 \AND \lip \tilde u = \lip u_i\leq \beta.
	\end{align}
	In particular, similarly as  in \eqref{eq:lip-cone-est}
	\begin{align} \label{lip_cone}
	\graph \tilde u\cap B_\varrho(0) \subset C_0\li\beta,F_i\ri  \quad\Foa\varrho>0.
	\end{align}
	Now, $\va \li H(p),F_i \ri = \va \li T_{p-p_i} \graph u_i,F_i\ri \leq \Vert Du_i(\xi)\Vert \leq \beta$, 
by Lemma \ref{lem:angle_lip}. The Cone Lemma 
	\ref{lem:cone} applied to $F:=F_i,$ $G:=H(p)$, $\chi=\sigma:= \beta$, and 
	$\kappa:=\alpha$, where condition \eqref{cone_lem_condition} is satisfied due to \eqref{beta_lip_cond}, 
implies $\graph \tilde u \cap B_\varrho(0) \subset C_0\li \alpha,H(p)\ri$, and therefore,
	\begin{align}\label{eq:graphs-equal}
	\Gamma_i \cap B_\varrho(p) \overset{\eqref{lip_shifted}}{=}
	 \li p+\graph \tilde u \ri \cap B_\varrho(p)\subset C_p\li \alpha,H(p)\ri.
	\end{align} 
	The Tilting Lemma \ref{lem:tilting} applied to $F:= F_i$, $G:=H(p)$, $\va(F,G)\leq \beta \equiv \chi$, 
$u:=\tilde u$ with $\lip u \leq \beta$ satisfying $\sigma:=\chi\li1+\lip u \ri \leq \beta\li1+\beta\ri <1$, implies 
	\[B_{\frac {1-(\beta+1)\beta}{\sqrt{1+\beta^2}}\varrho}(0) \cap H(p)\subset \Pi_{H(p)}\big( \graph \tilde u 
\cap {B_\varrho(0)}\big) \Foa \varrho>0.\]
	Define for arbitrary $\varrho >0$ the flat $m$-dimensional disks
	$D_p^\varrho := B_{\frac {1-(\beta+1)\beta}{\sqrt{1+\beta^2}}\varrho}(p) \cap \big(p+H(p)\big)$
	so that
	\begin{align} \label{eq:nice-proj}
	D_p^\varrho\,\subset\, \Pi_{p+H(p)} \big( (p + 
\graph \tilde u) \cap {B_\varrho(p)}\big)
\,\subset\, \Pi_{p+H(p)}\big( \Sigma \cap C_p(\alpha,H(p))\big).
	\end{align}
	Now let $\mathscr{L}\subset \N$ be the set of indices $l \in \N$, such that 
	\begin{align} \label{lip_bad_angle}
	\va \li F_l,F_i\ri \geq  {M\alpha}/2.
	\end{align}
	For any fixed $l\in \mathscr {L}$ consider the graph $\Gamma_l$.
If there exists a point $q \in \Gamma_l \cap \Gamma_i$ use the Shifting Lemma \ref{lem:shift-graphs} to find the shifted functions
$	\tilde u_l\in C^{0,1}(F_l,F_l^\perp)$ and $\tilde{u}_i
	\in C^{0,1}(F_i,F_i^\perp)$ with  $\tilde{u}_l(0)=0$, 
	$ \lip \tilde u_l
	\leq \beta,$ and $\tilde{u}_i(0)=0$, $\lip\tilde{u}_i\le\beta$,
	satisfying  as in \eqref{lip_shifted} 
	\begin{align}\label{eq:equal-graphs}
	\Gamma_i = q+\graph \tilde u_i \quad\AND\quad \Gamma_l 
	= q+\graph \tilde u_l.
	\end{align}
	We  can now use \eqref{beta_lip_cond} and \eqref{lip_bad_angle} to find
	$\va \li F_l,F_i\ri \geq  {M\alpha}/2 > 8\beta$, which allows us to
	 apply the Lemma of Intersecting Lipschitz Graphs 
	\ref{lem:intersect_lip_graphs}  for $\sigma:=\beta$ and $\chi:=M\alpha/2 $ to conclude by means of \eqref{eq:equal-graphs} and the translational invariance
	of the Hausdorff measure
	\begin{align}\label{lip_dimension}
	\dim_{\HM} \li \Gamma_i \cap \Gamma_l \ri& 
	= 
	\dim_{\HM}\li \graph \tilde u_i\cap \graph \tilde u_l\ri \leq m-1,
	\end{align}
	which is true, of course, also if $	\Gamma_l \cap \Gamma_i=
	\emptyset,$ so that \eqref{lip_dimension} holds for all $l\in\mathscr{L}$.
 If $k \in \N\setminus \mathscr{L}$, on the other hand, we have for any point
 $q \in \Gamma_k$, 
due to \eqref{beta_lip_cond} and Lemma \ref{lem:angle_lip}, 
	$\va \li T_{q-p_k}\graph u_k,H(p)\ri \leq \va \li F_k,F_i\ri + 2\beta < M\alpha,$
	which implies for all $\varrho>0$
	\begin{align}\label{lip_good_set}
\textstyle 	B_\varrho(p)\cap\big( \Gamma_i \setminus \bigcup\nolimits_{l\in \mathscr {L}}
	\Gamma_l \big)\subset 
B_\varrho(p)\cap E_{\alpha,M}(p).
	\end{align}
	For all $x \in D_{p}^{\varrho/ 2}= B_{ 
	\frac {1-(\beta+1)\beta}{2\sqrt{1+\beta^2}}\varrho}(p) \cap \li p + H(p)\ri$
 there is by virtue of
 \eqref{eq:nice-proj}
   and \eqref{eq:graphs-equal}  a point $\eta_x \in \Gamma_i \cap B_{\varrho/2}(p)\subset \Sigma \cap C_p(\alpha,H(p))$, such that $\Pi_{p+H(p)}(\eta_x)=x$,
 and $\vert \eta_x-p\vert <\sqrt{1+\alpha^2}\varrho /2 <\varrho$, since $\alpha <1$. This implies
	$
	B_r(\eta_x)\subset B_{2\varrho}(p) \Foa r \in (0,\varrho],
$
which combined with \eqref{lip_good_set} and \eqref{lip_dimension} yields 
	\begin{align*}
	\HM^m &\textstyle\li E_{\alpha,M}(p)\cap B_r(\eta_x)\ri 
	\geq \HM^m \big(\big[ \Gamma_i \setminus \bigcup\nolimits_{l\in\mathscr{L}}
	\Gamma_l \big] \cap B_r(\eta_x)\big)
	=\HM^m\big( \Gamma_i \cap B_r(\eta_x)\big)\\
	&\geq \HM^m\li \Pi_{F_i}\li \Gamma_i \cap B_r(\eta_x)\ri\ri
	\geq \HM^m\big( B_{ r/ {\sqrt{1+\beta^2}}}(0)\cap F_i\big) = \omega_m \big( r/ {\sqrt{1+\beta^2}} \big)^m
	\end{align*}
	for all $ r\in (0,\varrho],\,\varrho>0.$
		 Since there was no restriction on the radius $\rho>0$ throughout
	the proof so far, we can simply
	set $R_K=1$ for an arbitrary given compact subset $K\subset\Sigma$.
	\end{proof}


\section{Finite energy sets are manifolds} \label{sec:manifolds}
\subsection{Good approximating planes for admissible sets}\label{sec:flatness}
In addition to the $\beta$-numbers \eqref{eq:beta-number}
measuring the local flatness of a set we introduce here  the 
corresponding $\beta$-number with respect to a fixed plane,
as well as the corresponding \emph{bilateral} flatness
parameter  $\theta$.
\begin{defin} \label{def:reifenberg-flatness}
	For $p \in \Sigma \subset \R^m$, an $m$-plane $F\in \mathscr{G}(n,m)$, and a radius $r>0$, the \emph{$\beta$- and $\theta$-numbers of $\Sigma$ with respect
	to $F$} are 
given by
	\begin{align}
	\beta_\Sigma (p,F,r)&:= \sup_{y\in \Sigma\cap B_r(p)}\dist\big(
	y,(p+F)\cap B_r(p)\big)/r,\label{eq:beta-plane}\\
	\theta_\Sigma(p,F,r)&:= \dist_{\HM} \big( \Sigma \cap B_r(p), (p+F) \cap B_r(p) \big)/r.\label{eq:theta-plane}
	\end{align}
	We say $\Sigma$ is an \emph{$(m,\delta)$-Reifenberg-flat set} if for all compact subsets $K\subset \Sigma$ there is a radius $r_0=r_0(K)>0$, such that for all $p \in K$ and $r\in(0,r_0(K)]$, there exists a plane $F_p(r,\delta)\in \mathscr{G}(n,m)$ with
	$\theta_\Sigma(p,F_p(r,\delta),r)\leq \delta.$ Minimizing over all $m$-planes one obtains analogously to \eqref{eq:beta-number} the \emph{$\theta$-number}
		\begin{equation}\label{eq:theta-number}
	\theta_\Sigma(p,r):= \inf_{F\in\mathscr{G}(n,m)}\theta_\Sigma(p,F,r).
	\end{equation}
\end{defin}

Recall that  the Hausdorff-distance of two sets
$A,B\subset \R^n$ is given by
 $\dist_{\HM}(A,B):=\max \{ \sup_{a\in A}\dist (a,B), \sup_{b\in B} \dist(b,A) \}$ so that
\begin{align}\label{theta_def}
\theta_\Sigma(p,F,r)
&= \max \big\{\,\beta_\Sigma(p,F,r) , \sup_{\xi \in (p+F)\cap B_r(p)}\dist \big( \xi,\Sigma \cap B_r(p)\big)/r\big\}.
\end{align}
The second term of the right hand side of \eqref{theta_def} can be bounded
uniformly for \emph{ any} admissible set in $\mathscr{A}^m(\alpha,M)$.
\begin{lem} \label{cor:dist_plane_to_sigma}
	Let   $\Sigma \in \Aam$ for $1\leq m\leq n$, $M>0$, and $\alpha> 0$. For all compact subsets $K\subset \Sigma$, there exists a radius $\varrho_K \in (0,R_K]$, such that for all $p \in K\cap \Sigmaa$, we have
	\[   \sup_{\xi \in (p+H(p))\cap B_r(p)} \dist\big( \xi,\Sigma\cap B_r(p)\big)/r < {2\alpha}/{\sqrt{1+\alpha^2}} \quad\Foa r \in(0,\varrho_K]. \]
	Moreover, for $q \in K\setminus \Sigmaa$ and any
	 sequence $(p_i)_{i \in \N} \subset \Sigmaa$ with $\lim_{i\to\infty}p_i=q$, such that $\lim_{i\to\infty} H(p_i)=:
	 F \in \mathscr{G}(n,m)$ exists, one also finds
	\[   \sup_{\xi \in (q+F)\cap B_r(q)} \dist\big( \xi,\Sigma\cap B_r(q)\big)/r
	 \leq  {2\alpha}/{\sqrt{1+\alpha^2}} \quad\Foa r \in(0,\varrho_K]. \]
\end{lem}
Notice that for every $q\in K\setminus \Sigmaa$ there is a sequence $(p_i)_i$
as stated in the last part of Lemma \ref{cor:dist_plane_to_sigma} by density of $\Sigmaa$ in $\Sigma$ combined with the compactness of the Grassmannian $\mathscr{G}(n,m)$.

\begin{proof}
	First, assume $p \in \Sigmaa\cap B_{ {R_K}/{10}}(K)$ and define $\tilde K:=\Sigma \cap \overline{ B_{ {R_K}/{10}}(K)}$. Then $\tilde K$ is compact, since $\Sigma$ is closed, and the constants $R_{\tilde K}\in (0,R_K]$ and $c_{\tilde K}\in(0,c_K]$ of
	Definition \ref{def:admissible_sets} applied to $\tilde{K}$
	are solely determined by $K$ itself. So, there exists a dense subset $D_p\subset \li p + H(p)\ri \cap B_{R_{\tilde K}}(p)$ such that for all $x\in D_p$ one finds a point
	 $\eta_x \in \Sigma \cap C_p(\alpha,H(p))$ with
	\begin{align} \label{plane_cor_etax}
	x=\Pi_{p+H(p)}(\eta_x) = p + \Pi_{H(p)}(\eta_x-p),
	\end{align}
	so that $
		\eta_x-p =\Pi_{H(p)}(\eta_x-p)+\Pi_{H(p)^\perp}(\eta_x-p) = x-p + \Pi_{H(p)^\perp}(\eta_x-p),$
		or $\eta_x-x=\Pi_{H(p)^\perp}(\eta_x-p)$, which implies by the fact that $\eta_x\in C_p(\alpha,H(p))$ and by \eqref{plane_cor_etax}
	\begin{align} \label{plane_cor_etax_x}
	\vert \eta_x-x \vert \leq \alpha \left \vert \Pi_{H(p)}(\eta_x-p)\right \vert = \alpha \vert x-p\vert.
	\end{align}
	In particular, setting $
		\varrho_K:=R_{\tilde K}\leq R_K,$
	we obtain for $r \in (0,\varrho_K]$ and $x \in D_p\cap B_{ r/ {\sqrt{1+\alpha^2}}}(p)$
	\begin{align} \label{plane_cor_etax_x_small}
	\vert \eta_x-x\vert <  {\alpha r}/{\sqrt {1+\alpha^2}}
	\end{align}
	by means of \eqref{plane_cor_etax_x}, and also $\vert \eta_x-p\vert^2=\vert x-p\vert^2 + \vert \eta_x-x\vert^2 \leq (1+\alpha^2)\vert x-p\vert^2$, so that $
	\vert \eta_x-p\vert <
	r.$
	Combined with \eqref{plane_cor_etax_x_small} one finds
	\begin{align*}\label{plane_cor_sigmaa_dist_1}
	\dist\big( x,\Sigma \cap B_r(p)\big) \leq \vert \eta_x-x\vert &<  {\alpha r}/ {\sqrt{1+\alpha^2}} 
	\Foa x \in D_p \cap B_{\frac r {\sqrt{1+\alpha^2}}}(p), \,  r \in (0,\varrho_K].
	\end{align*}
	By density of $D_p$ in $\li p +H(p)\ri\cap B_{R_{\tilde K}}(p)$, we find for each $\xi \in \li p+H(p)\ri \cap \overline{B_{ r/ {\sqrt {1+\alpha^2}}}(p)}$ and each $l\in \N$ some point $x_l=x_l(\xi) \in D_p \cap 
	B_{ r/ {\sqrt{1+\alpha^2}}}(p)$, such that $\vert \xi-x_l\vert < 1/ l$, and such that there is a point $\eta_{x_l}\in \Sigma \cap C_p(\alpha,H(p))\cap B_r(p)$ with $\Pi_{p+H(p)}(\eta_{x_l})=x_l$. Consequently, by means of \eqref{plane_cor_etax_x_small} with $x$ replaced by $x_l$, one 
	obtains 
	\[\dist\big( \xi , \Sigma \cap B_r(p)\big) \leq \vert \xi-\eta_{x_l}\vert <  1/ l +
	 \vert x_l-\eta_{x_l}\vert <  1 / l + 
	 { \alpha r}/{\sqrt{1+\alpha^2}} \Foa l \in \N . \]
	Taking the limit $l\to \infty$, we find
	\begin{align}\label{plane_cor_sigma_dist_1}
	\dist \big( \xi , \Sigma \cap B_r(p)\big) \leq    {\alpha r} /{\sqrt{1+\alpha^2}} 
	\quad\Foa \xi \in \li p+H(p)\ri \cap \overline{B_{ r/ {\sqrt{1+\alpha^2}}}(p)}.
	\end{align}
	For $\mu \in p+H(p)$ with $ r/ {\sqrt{1+\alpha^2}}<\vert \mu-p\vert <r,$ we
	set $\xi_\mu:= \frac r{\sqrt{1+\alpha^2}} \cdot \frac {\mu-p}{\vert \mu-p\vert}\in \li p + H(p)\ri \cap \overline{B_{\frac r {\sqrt{1+\alpha^2}}}(p)},$ and estimate
	\begin{align}\label{plane_cor_sigma_dist_2}
	\dist\big( \mu,\Sigma\cap B_r(p)\big) &\leq \dist \big( \xi_\mu,\Sigma \cap B_r(p)
	\big) + \vert \mu-\xi_\mu\vert <\textstyle\frac  {\alpha r}{\sqrt{1+\alpha^2}} + \Big( 1-\frac 1{\sqrt{1+\alpha^2}}\Big) r\notag\\
	&=\textstyle\Big( \frac {\alpha + \sqrt{1+\alpha^2} -1}{\sqrt{1+\alpha^2}}\Big) r < \Big( \frac {\alpha + (1+\alpha) -1}{\sqrt{1+\alpha^2}}\Big) r 	
	=\frac {2\alpha r}{\sqrt{1+\alpha^2}},	
	\end{align}
	so that the combination of \eqref{plane_cor_sigma_dist_1} and \eqref{plane_cor_sigma_dist_2} gives
	\begin{align} \label{plane_cor_case_1}
	\hspace{-0.4cm}\dist\big( \mu,\Sigma \cap B_r(p)\big) &<  {2\alpha r}/{\sqrt{1+\alpha^2}}
	 \Foa \mu \in \li p+H(p)\ri \cap B_r(p),\, r \in (0,\varrho_K] 
	\end{align}
if $	p \in 
	\Sigmaa \cap B_{ {R_K}/{10}}(K). $
	Now assume $q \in K\setminus\Sigmaa$ and take a sequence
	$(p_i)_{i\in N}\subset \Sigmaa$ with $\lim_{i\to \infty} p_i=q$ and $\lim_{i\to\infty}H(p_i)=:F\in \mathscr{G}(n,m)$.  Thus,
	 $p_i\in \Sigmaa \cap B_{ {R_K}/{10}}(K)$ for all $i\gg 1$.
	We claim that
	\begin{align}\label{plane_cor_aux_assumption}
\textstyle\sup_{x \in \li q + F\ri \cap \overline{B_r(q)}} 
\dist \big( x , \Sigma \cap \overline{ B_r(q)}\big) 
\leq  {2\alpha r}/{\sqrt{1+\alpha^2}} \Foa r\in(0,\varrho_K].
	\end{align}
	Indeed, by virtue of the Hausdorff convergence of the $m$-planar disks
	$(p_i+H(p_i))\cap B_r(p_i)$ to the closed disk
	$(q+F)\cap\overline{B_r(q)}$ as $i\to\infty$ we can find
	 for any $x \in \li q + F\ri \cap \overline {B_r(q)}$ a sequence $x_l\to x$
	 as $l\to\infty$  with $x_l \in \li p_l+H(p_l)\ri \cap B_r(p_l)$, and by \eqref{plane_cor_case_1} applied to $p\equiv p_l$ and $\mu:=x_l$
	 one finds points $\eta_l
	  \in \Sigma \cap B_r(p_l)$, such that $\vert x_l-\eta_l\vert < {2\alpha r}/{\sqrt{1+\alpha^2}}$ for all $l \in \N$. Since $\Sigma$ is closed and by the Hausdorff-convergence  we may assume $\eta_l\to\eta \in \Sigma \cap \overline{B_r(q)}$ as $l\to\infty$. Hence, $\vert x-\eta\vert \leq  {2\alpha r}{\sqrt{1+\alpha^2}}$, which implies \eqref{plane_cor_aux_assumption}.
	
	Finally, let $r \in(0,\varrho_K]$. For arbitrary $\xi\in \li q+F\ri \cap B_r(q)$ use \eqref{plane_cor_aux_assumption} to estimate
	\begin{align*}
&	\dist \big( \xi,\Sigma \cap B_r(q)\big) \leq \textstyle\sup_{x \in \li q+F\ri\cap \overline{B_{\vert \xi-q\vert}(q)}} \dist\big( x , \Sigma \cap B_r(q)\big) \\
	&\leq \textstyle\sup_{x \in \li q+F\ri\cap \overline{B_{\vert \xi-q\vert}(q)}} \dist\big( x , \Sigma \cap \overline{B_{\vert \xi-q\vert}(q)}\big) 
	\leq \frac {2\alpha}{\sqrt{1+\alpha^2}}\vert \xi-q\vert < \frac {2\alpha r}{\sqrt{1+\alpha^2}}.
	\end{align*}
\end{proof}

\subsection{Finite energy yields Reifenberg flatness}\label{sec:reifenberg}
We have seen in Lemma \ref{cor:dist_plane_to_sigma} that the second term in bilateral $\theta$-number 
\eqref{theta_def} is automatically controlled for admissible
sets in $\Aam$, but in order to control the first term, the $\beta$-number, we need
locally finite energy. To begin with, notice that the
 numerator $L_\tau(x,y,H(x),H(y)):=\ang \big(H(x),\mathcal{R}_{xy}(H(y))\big)^{(1+\tau)m}$ of our Lagrangian
 in \eqref{eq:energies}
involving the angle metric defined in \eqref{eq:angle-metric} can be rewritten as
$L_\tau(x,y,H(x),H(y))=\sup_{e\in H(x)\cap\S^{n-1}}F_\tau(x,y,e)$ with
\begin{equation}\label{eq:new-Ftau}
F_\tau(x,y,e):= \big\vert \Pi_{H(y)^\perp}(e)- {2} \langle e,x-y\rangle \Pi_{H(y)^\perp}(x-y)/|x-y|^2 
\big\vert^{(1+\tau)m},
\end{equation}
by means of the explicit formula \eqref{eq:reflection} 
for the reflection $\mathcal{R}_{xy}$. 
\begin{thm}[$\beta$-number estimate] \label{thm:dist_sigma_to_plane}
Let $\delta \in (0,1)$, $\Sigma \in \Aam$ for $2\leq m\leq n$, $M>0$, and 
$\alpha> 0$ satisfy
	\begin{align}\label{alpha_M_connection}
	\alpha(M+1) <  \delta/ {50}
	\end{align}
	and assume that $\Sigma $ has locally finite energy $E^\tau$ as  in 
	\eqref{eq:locally-bdd-energy} for some $\tau >-1$.
	Then, for all compact subsets $K\subset \Sigma$ there exists a radius $r_K=r_K(\delta,\tau,m) \in (0,R_K]$, such that for all $p \in K$ there is an $m$-plane $G_p\in \mathscr{G}(n,m)$ with
	\begin{align}
	\Sigma \cap B_r(p) \setminus \overline{B_{\delta r}(p+G_p)} = \emptyset \Foa r \in (0,r_K].	
	\end{align}
	In particular, $
	\sup_{p \in K} \beta_\Sigma(p,G_p,r) \leq \delta \Foa r \in(0,r_K].
	$
	Moreover, $
	 G_p=H(p)$ for all $ p \in \Sigmaa \cap K$,
	and for $q \in K\setminus \Sigmaa$, 
	the $m$-plane $G_q$
 can be taken to be any accumulation point of a
	sequence 
	 $(H(p_k))_{k\in\N}\subset\mathscr{G}(n,m)$ where
 $\Sigma^*\ni p_k\to q$ as $k\to\infty$.
\end{thm}
\begin{proof}
{\it Step 1.}\,
	By density it suffices to find for a given compact set $K\subset \Sigma$ a radius $r_K\in(0,R_K]$ such that
	\begin{align} \label{tubu_neighb_sigmaa}
	\Sigma \cap B_r(p) \setminus \overline{B_{\delta r}(p+H(p))} = \emptyset \Foa p \in 
\Sigmaa \cap B_{1/{10}}(K),\, r \in (0,r_K],
	\end{align}
	since an arbitrary point $p\in K$ may be approximated by points $p_l\in\Sigmaa 
\cap B_{1/{10}}(K)$, for which we may assume w.l.o.g. that there exist  
$H(p_l)\in \mathscr{G}(n,m)$, such that $H(p_l)\to G_p \in \mathscr{G}(n,m)$ 
	as $l\to\infty$ by compactness of the Grassmannian. The Hausdorff-convergence
	of the balls  $B_r(p_l)$ to $ B_r(p)$ and of the $m$-planar disks
	$B_{1}(p_l) \cap (p_l+H(p_l))$ to $ B_{1}(p)\cap (p+G_p)$ as $l\to\infty$ implies that any 
point $q \in \Sigma \cap B_r(p)\setminus \overline{B_{\delta r}(p+G_p)}$ satisfies $q \in \Sigma 
\cap B_r(p_l)$ for $l\gg1$, as well as
	$
	\dist\li q,p_l+H(p_l)\ri \geq \dist \li q,p+G_p\ri - \dist_{\HM}\li (p+G_p)\cap B_{1}(p), 
(p_l + H(p_l))\cap B_{1}(p_l)\ri
	> \delta r \Foa l \gg 1,$
		thus contradicting \eqref{tubu_neighb_sigmaa} for $p\equiv p_l$ and $l\gg1.$
		Before moving on to the the second step notice that all $p_l$ are contained
		in the compact set $\tilde{K}:=\Sigma\cap\overline{B_1(K)}$, so that 
		the constants $R_{\tilde{K}}$ and $c_{\tilde{K}}$ of Definition
		\ref{def:admissible_sets} applied to the compact set $\tilde{K}$ satisfy
		$0<R_{\tilde{K}}\le R_K$ and $0<c_{\tilde{K}}\le c_K$. However, since
		the original compact set $K$ completely determines $\tilde{K}$ we think
		of $R_{\tilde{K}}$ and $c_{\tilde{K}}$ as depending on $K$ only.
		
{\it Step 2.}\,
Assuming for contradiction that there is a point $p\in\Sigma^*\cap B_{1/10}(K)$ such that 
\eqref{tubu_neighb_sigmaa} does \emph{not} hold for some
$r\in (0,r_K]$ for a constant $r_K\in (0,R_K]$ to be determined later, then there is some point 
$\tilde{q}\in\Sigma\cap B_r(p)$
with $\dist\big(\tilde{q},p+H(p)\big) > \delta r. $ By density of $\Sigma^*$ in 
$\Sigma$ we find a point $q\in\Sigma^*\cap B_r(p)$ such that
\begin{equation}\label{q-dist-contra}
\dist\big( q, p+H(p)\big) > \delta r.
\end{equation}
We quantify the energy contribution of a geometric situation like in \eqref{q-dist-contra} in the 
following 
two lemmas, depending on the size of the angle
$\ang (H(p),H(q)).$ The proofs of these auxiliary lemmas are postponed to the
end of this subsection.
\begin{lem}[Almost parallel strands] \label{lem:dist_sigma_to_plane_1}
	Let $\tau>-1$, $\delta \in (0,1)$, $\eps\in
	(0,\delta/500]$, and
	$\Sigma \in \Aam$ for $1\leq m\leq n$, where the constants
	$M, \alpha > 0$ satisfy
 \eqref{alpha_M_connection}. Assume that there is 
 a compact subset $K\subset \Sigma$ and points $p \in \Sigmaa \cap B_{ {1}/{10}}(K)$ 
and $q \in \Sigmaa \cap B_{\eps R}(p)$ for some radius $R \in (0,R_{\tilde K}]$
 for $\tilde K=\Sigma \cap \overline{B_1(K)}$, such that
	\begin{align} 
	\dist \li q,p+H(p) \ri &> \delta \eps R, \label{tubu_neighb_sigmaa_contra}\\
	\va \li H(p),H(q) \ri &< \omega(\delta) + 2 M\alpha\quad\Fo
 \omega(\delta):=  {153\cdot \delta}/{50^3}. \label{parallel_angle}
	\end{align}
		Then,
	\begin{align} \label{parallel_energy}
	\int \nolimits_{\Sigma \cap B_{\eps^2 R}(q)} \int \nolimits_{\Sigma \cap B_{2\eps R}(p)} 
\frac{L_\tau\big(\mu,\eta,H(\mu),H(\eta)\big)}{|\mu-\eta|^{2m}}
 \ d\HM^m(\eta) \ d\HM^m(\mu) > c_1,
	\end{align}
	where $
c_1=	c_1(K,\eps,\delta,\tau,m):= c_{\tilde K}^2 \cdot \eps^{2m} 
\cdot \frac 1 {255^{2m}}\cdot \frac {\delta^{(1+\tau)m}}{10^{2(\tau-1)m}}>0.$ 
\end{lem}

\begin{lem}[Transversal strands] \label{lem:dist_sigma_to_plane_2}
Suppose in addition to $m\ge 2$ that all assumptions of Lemma 
\ref{lem:dist_sigma_to_plane_1} hold true 
except \eqref{parallel_angle}, then
	\begin{align} \label{transversal_energy}
	\int \nolimits_{\Sigma \cap B_{\eps^2 R}(q)} \int \nolimits_{\Sigma \cap B_{ R}(p)}
\frac{L_\tau\big(\mu,\eta,H(\mu),H(\eta)\big)}{|\mu-\eta|^{2m}} 
 \ d\HM^m(\eta) \ d\HM^m(\mu) > c_2,
	\end{align}
	where $
c_2=	c_2(K,\eps,\delta,\tau,m):=c_{\tilde K}^2 \cdot \eps^{4m} \cdot 
\frac {(1.9)^{(3+\tau)m}}{10^{5(1+\tau)m}} \cdot \delta^{(1+\tau)m}>0.$
\end{lem}
	To apply these lemmas, we fix $\eps_0=\eps_0(\delta):=\frac \delta{500}$ and 
	obtain 
	\begin{equation}\label{eq:c_1c_2}
c_2<c_1\quad\Foa\tau > -1.
	\end{equation}

{\it Step 3.}\,
In order to deduce from \eqref{q-dist-contra} a contradiction notice first that there is
some $N=N(K)\in\N$ such that $B_2(K)\cap\Sigma\subset B_N(0)$, so that
$E^\tau(\Sigma\cap B_2(K))\le E^\tau(\Sigma\cap B_N(0)) <\infty$ since
$\Sigma$ has locally finite energy; see \eqref{eq:locally-bdd-energy}. Therefore,
we can use the absolute continuity of the double integral
 defining $E^\tau$ to find: For all $c>0$, there exists a radius 
$R^\textnormal{ac}_K(c)\in (0,1] $, such that
	\begin{align} \label{abs_continuity}
	\int \nolimits_{\Sigma \cap B_r(x)} \int \nolimits_{\Sigma \cap B_r(y) }
   \frac{L_\tau\big(\mu,\eta,H(\mu),H(\eta)\big)}{|\mu-\eta|^{2m}}  
 \ d\HM^m(\mu)\ d\HM^m(\eta) <c 
	\end{align}
	for all $ x,y \in \tilde{K}$ and all $r\in(0,R^\textnormal{ac}_K(c)]$.
	In particular, $R^\textnormal{ac}_K(c_2)\le R^\textnormal{ac}_K(c_1)$ 
	by means of \eqref{eq:c_1c_2}.
	 Choose $r_K:=\varepsilon_0\cdot
	\min\{ R^\textnormal{ac}_K 
	(c_2) ,R_{\tilde{K}}\} < R_{\tilde{K}} /500 \le R_K /500$
	 and distinguish two cases.

{\it {Case I: $\va \li H(p),H(q) \ri < \omega(\delta)+2M\alpha$}.}\,
	Then apply Lemma \ref{lem:dist_sigma_to_plane_1} for $R:=\frac r {\eps_0} \leq 
R_{\tilde K}$ and $\eps:=\eps_0$ to find by \eqref{abs_continuity}, 
\eqref{parallel_energy}, and \eqref{eq:c_1c_2}
for $x=q$, $y=p$
	\begin{align*}
	{c_2}&> \int \nolimits_{\Sigma \cap B_{R^\textnormal{ac}_K(c_2)}(q)} \int 
\nolimits_{\Sigma \cap B_{R^\textnormal{ac}_K(c_2)}(p)}
\frac{L_\tau\big(\mu,\eta,H(\mu),H(\eta)\big)}{|\mu-\eta|^{2m}}   
 \ d\HM^m(\eta) \ d\HM(\mu)\\
	&\geq \int \nolimits_{\Sigma \cap B_{\eps_0^2R}(q)} \int \nolimits_{\Sigma 
\cap B_{2\eps_0R}(p)}\frac{L_\tau\big(\mu,\eta,H(\mu),H(\eta)\big)}{|\mu-\eta|^{2m}}  
\ d\HM^m(\eta) \ d\HM(\mu)
	>c_1>{c_2},
	\end{align*}
	which provides a contradiction.\\
	
{\it {Case II: $\va \li H(p),H(q) \ri \geq \omega(\delta)+2M\alpha$}.}\,
	Then apply Lemma \ref{lem:dist_sigma_to_plane_2} for $R:= r/ {\eps_0}$ and $\eps:=\eps_0$ to find by \eqref{transversal_energy} and \eqref{abs_continuity}
	\begin{align*}
{c_2} & > \int \nolimits_{\Sigma \cap B_{R^\textnormal{ac}_K(c_2)}(q)} 
\int \nolimits_{\Sigma \cap B_{R^\textnormal{ac}_K({c_2})}(p)}
 \frac{L_\tau\big(\mu,\eta,H(\mu),H(\eta)\big)}{|\mu-\eta|^{2m}}   \ d\HM^m(\eta) \ d\HM(\mu)\\
	&\geq \int \nolimits_{\Sigma \cap B_{\eps_0^2R}(q)} \int \nolimits_{\Sigma \cap B_{R}(p)}
 \frac{L_\tau\big(\mu,\eta,H(\mu),H(\eta)\big)}{|\mu-\eta|^{2m}}   \ d\HM^m(\eta) \ d\HM(\mu)
	 > c_2,
	\end{align*}
	which is a contradiction as well.
\end{proof}
It remains to prove the
 two auxiliary lemmas. 
\begin{proof}[Proof of Lemma \ref{lem:dist_sigma_to_plane_1}]
Fix $\delta\in (0,1)$ and $\varepsilon\in (0,\delta/500].$
	Notice that both points $p,q$ are contained in $\tilde K$. Let $r\in (0,R/10^5]$.
Definition \ref{def:admissible_sets} applied to $\tilde{K}$ guarantees the existence of 
a dense subset 
$D_q\subset ( q + H(q)) \cap B_R(q)$, so that for any given $\rho <  r/ {\sqrt {1+\alpha^2}}$
we can find a point
$y(\rho) \in D_q \cap B_{ \rho}(q)$  and a corresponding point
 $\eta_{y}(\rho) \in \Sigma \cap C_q(\alpha,H(q))$ 
with $\Pi_{q +H(q)}(\eta_{y}(\rho))=y(\rho)$, which implies $\vert \eta_y(\rho) -q \vert 
\leq \sqrt{1+\alpha^2} \rho <r.$ 
	Consequently, for any such $\rho \in (0,r/ {\sqrt {1+\alpha^2}})$,
	\[ \HM^m \li E_{\alpha,M}(q) \cap B_{r}(q)\ri \geq \HM^m \big( E_{\alpha,M}(q) 
\cap B_{r-\sqrt{1+\alpha^2}\rho}(\eta_y(\rho)) \big) \geq c_{\tilde K} 
\big( r-\sqrt{1+\alpha^2}\rho\big) ^{m}.  \]
	Taking the limit $\rho \to 0$ guarantees $\HM^m\li E_{\alpha,M}(q) \cap B_r(q)\ri 
\geq c_{\tilde K}r^m$. In particular, for the set $M_q(\eps):=E_{\alpha,M}(q) \cap B_{\eps^2R}(q)$ we obtain
	\begin{align}\label{Mq_measure_1}
	\HM^m \li M_q(\eps) \ri \geq c_{\tilde K} \li \eps^2R\ri^m. 
	\end{align}
	For any point $\mu \in M_q(\eps)$ and an arbitrary vector $e\in H(\mu)\cap \S^{n-1}$ set
	\begin{align} \label{vector_def_1}
	w\equiv w(\mu,e):=  {\Pi_{H(p)}(e)}/{\vert \Pi_{H(p)}(e)  \vert}
 \in H(p)\cap \S^{n-1},			
	\end{align}
	which is well-defined since by \eqref{parallel_angle}
	\begin{align}\label{angle_mu_p_1}
	\va\li H(\mu),H(p)\ri \leq \va \li H(\mu),H(q)\ri + \va \li H(q),H(p) \ri < 
\omega(\delta) + 3M\alpha ,
	\end{align}
	so that by Lemma \ref{lem:angle_formula}, we obtain by definition of $\omega(\delta)$ in
\eqref{parallel_angle} and assumption \eqref{alpha_M_connection}
	\begin{align} \label{parallel_proj_e_1}
	 \vert \Pi_{H(p)}(e) \vert & =  \vert e-\Pi_{H(p)^\perp}(e)
 \vert \geq 1 - \va\li H(p),H(\mu)\ri >1- \li \omega(\delta) + 3M\alpha\ri >0.
	\end{align}
	 Furthermore, the definition of the set $M_q(\eps)$ implies
	\begin{align}\label{eta_exists_condition_1}
	 \vert \Pi_{p+H(p)}( \mu +  \eps  R w/2) - p \vert &= 
 \vert\Pi_{H(p)}( \mu + \eps R w/2 -p) \vert 
 \leq \vert \mu-q\vert +\vert
 q-p\vert +  \eps  R/2 \notag\\
	&< ( 3/ 2 + \eps ) \eps R <R.
	\end{align}
	By means of Definition \ref{def:admissible_sets} we can find a point $x\equiv x(\mu,e)
 \in D_p$ such that we obtain $ \vert x -\Pi_{p+H(p)}\li\mu-\eps  R w/2\ri
\vert<\eps^2R$, and  there exists a corresponding point $\eta_x=\eta_x(\mu,e)\in \Sigma \cap C_p(\alpha,H(p))$ 
with $x=\Pi_{p+H(p)}(\eta_x)=p+\Pi_{H(p)}(\eta_x-p)$. Consequently,
	\begin{align} \label{proj_etax_1}
	\Pi_{H(p)}(\eta_x-p)&=x-p=f_x + \Pi_{H(p)}\li \mu+ \eps  Rw/2-p\ri	
	\end{align}
	for $\vert f_x\vert = \vert x-\Pi_{p+H(p)}\li \mu +  \eps  Rw/2\ri
 \vert <\eps^2R$, which implies, on the one hand, the estimate
	\begin{align}\label{proj_etax_norm_1}
	\left \vert \Pi_{H(p)}(\eta_x-p)\right \vert \leq \vert f_x\vert + \left \vert 
\Pi_{H(p)}\li \mu +  \eps  Rw/2-p\ri\right \vert < \li  3/ 2 + 2 \eps\ri \eps R 
	\end{align}
	due to \eqref{eta_exists_condition_1}, and, on the other hand, the identity
	\begin{align}\label{etax_mu_projection_1}
	\Pi_{H(p)}(\eta_x-\mu)= \Pi_{H(p)} \li \eps  Rw/2\ri + f_x.
	\end{align}
	Since $\eta_x \in C_p(\alpha,H(p))$ we find by \eqref{proj_etax_norm_1}
	\begin{align}\label{proj_etax_orhto_norm_1}
	 \vert \Pi_{H(p)^\perp}(\eta_x-p) \vert \leq \alpha  \vert
 \Pi_{H(p)}(\eta_x-p) \vert < \alpha \li  3/ 2 +2\eps\ri \eps R;
	\end{align}
	hence $\vert \eta_x-p\vert ^2 = \vert \Pi_{H(p)}(\eta_x-p)
\vert^2+ \vert \Pi_{H(p)^\perp}(\eta_x-p) \vert^2< (1+\alpha^2)\li 
 3/ 2 +2\eps\ri^2\eps^2 R^2,$ i.e.,
	\begin{align}\label{etax_norm_1}
	\vert \eta_x-p\vert < \sqrt{1+\alpha^2} \li  3/ 2 +2\eps\ri \eps R <R,
	\end{align}
	since $\delta<1$ and by virtue of \eqref{alpha_M_connection}.
	The set $M_p(\mu,e,\eps):=E_{\alpha,M}(p)\cap B_{\eps^2R}\li\eta_x(\mu,e)\ri$ 
satisfies
	\begin{align}\label{Mp_inclusion_1}
	M_p(\mu,e,\eps)\subset \Sigma \cap B_{2\eps R}(p),
	\end{align}
	since for all $\eta \in M_p(\mu,e,\eps)$ one has
	\begin{align} \label{eta_p_dist}
	\vert \eta-p\vert \leq \vert \eta-\eta_x\vert &+ \vert \eta_x-p \vert < \eps^2 R +
 \sqrt{1+\alpha^2}\big(  3/ 2 +2\eps\big)\eps R <2\eps R,
		\end{align}
	where we also used \eqref{etax_norm_1} and $\sqrt{1+\alpha^2} \leq 1+\alpha <\ {51}/{50}$ by
  \eqref{alpha_M_connection}. Moreover,
	\begin{align}\label{Mp_measure_1}
	\HM^m\li M_p(\mu,e,\eps)\ri \geq c_{\tilde K} (\eps^2R)^m
	\end{align}
	by virtue of Definition \ref{def:admissible_sets} applied to the compact
 set $\tilde K \subset \Sigma$.
		For the fixed point $\eta_x\in \Sigma \cap C_p(\alpha,H(p))\cap B_R(p)$, 
we estimate by means of the identity $\Id=\Pi_{H(p)}+\Pi_{H(p)^\perp}$
	\begin{align*}
	 \vert \langle e , \mu-\eta_x\rangle  \vert   
	&\geq  \vert  \langle \Pi_{H(p)}(e), \Pi_{H(p)}(\mu-\eta_x) \rangle 
\vert-  \vert  \langle \Pi_{H(p)^\perp}(e),  \Pi_{H(p)^\perp}(\mu-\eta_x) \rangle 
 \vert\\
	&\geq  \vert  \langle \Pi_{H(p)}(e),  \eps  R \cdot \Pi_{H(p)}(w)/2
 \rangle  \vert- \vert f_x\vert- \vert  \Pi_{H(p)^\perp}(e) \vert \cdot 
 \vert   \Pi_{H(p)^\perp}(\mu-\eta_x)  \vert\\
	&> \eps  R  \vert \Pi_{H(p)}(e) \vert /2 - \eps^2R -\li \omega(\delta)+ 
3M\alpha\ri \cdot \vert   \Pi_{H(p)^\perp}(\mu-\eta_x)  \vert,
	\end{align*}
	due to \eqref{etax_mu_projection_1}, \eqref{angle_mu_p_1}, $\vert e \vert =1$, and 
$\vert f_x\vert < \eps^2R$. Consequently, with \eqref{parallel_proj_e_1}, we obtain
	\[ \left \vert \langle e,\mu-\eta_x\rangle \right \vert > \li 1 - \li \omega(\delta) + 
3M\alpha \ri \ri  \eps  R/2 - \eps^2 R - \li \omega(\delta) + 3 M\alpha\ri  \vert 
\Pi_{H(p)^\perp}(\mu-\eta_x) \vert. \]
The estimate \eqref{proj_etax_orhto_norm_1} implies
	\begin{align*}
	\vert \Pi_{H(p)^\perp}(\mu-\eta_x) \vert &\leq \vert \mu-p \vert +
 \vert \Pi_{H(p)^\perp}(p-\eta_x) \vert <\vert \mu - q \vert + \vert q - p 
\vert + \alpha \li  3/ 2 +2\eps\ri \eps R \\
	&< \li \eps + 1 + \alpha\li  3/ 2 +2\eps\ri\ri \eps R;
	\end{align*}
	hence $
	\left \vert \langle e, \mu-\eta_x\rangle \right \vert > \li  1/ 2-
 \li \omega(\delta) + 3M\alpha\ri \li  3/ 2 +2\eps\ri\li1+\alpha\ri-\eps\ri\eps R. $
	For arbitrary $\eta \in M_p(\mu,e,\eps)$ one therefore has
	\begin{align} \label{scalar_prod_1}
	 \vert \langle e,\mu-\eta\rangle  \vert &\geq  \vert \langle e, 
\mu-\eta_x\rangle \vert - \vert \eta_x-\eta\vert \notag\\
	&>\li  1/ 2 - \li \omega(\delta) + 3M\alpha\ri \li  3/ 2 +2\eps\ri\li1+
\alpha\ri-2\eps\ri\eps R.
	\end{align}
	Moreover, for all $\mu \in M_q(\eps)$, $e \in H(\mu)\cap \S^{n-1}$ and
	 $\eta \in 
M_p(\mu,e,\eps)$ we estimate by means of \eqref{eta_p_dist}
	\begin{align}\label{norm_bound_1}
	\vert \mu-\eta\vert &\leq \vert \mu-q\vert + \vert q-p \vert + \vert p-\eta\vert 
	<2 \eps^2R+\eps R + \sqrt{1+\alpha^2} \li  3/ 2 +2\eps\ri\eps R \notag\\ &< \li
 2 \eps+1\ri ( 1+  3 \sqrt{1+\alpha^2}/2) \eps R.
	\end{align}
		For the remaining term in the energy density we simply write $
		 \vert \Pi_{H(\eta)^\perp}(\mu-\eta)\vert =  \vert \Pi_{H(p)^\perp}(\mu-p) - 
\Pi_{H(p)^\perp}(\eta_x-p) + \Pi_{H(p)^\perp}(\eta_x-\eta) 
	 -\Pi_{H(p)^\perp}(\mu-\eta) + \Pi_{H(\eta)^\perp}(\mu-\eta)\vert $, which -- using 
\eqref{proj_etax_orhto_norm_1} and \eqref{norm_bound_1} -- can be bounded from below
by 
\begin{align*}
	& \vert \Pi_{H(p)^\perp}(\mu-p) \vert  - \alpha (  3/ 2 +2\eps )
\eps R - \eps^2R
 - \va ( H(p),H(\eta)) ( 2 \eps +1 ) ( 1 +  3 
\sqrt {1+ \alpha^2}/2) \eps R\\
	&\hspace{-0.2cm}\geq \vert \Pi_{H(p)^\perp}(q-p)  \vert -\vert q-\mu\vert - \alpha 
( 3/ 2 +2\eps) \eps R - \eps^2R
	 - M\alpha ( 2 \eps +1 ) ( 1 +  3 \sqrt {1+ \alpha^2}/2) \eps R.
	\end{align*}
	Since $\vert q - \mu\vert < \eps^2 R$ and $ \vert \Pi_{H(P)^\perp}(q-p)
 \vert >\delta\eps R$ by assumption, we obtain
	\begin{align} \label{proj_bound_1}
	 \vert \Pi_{H(\eta)^\perp}(\mu-\eta)\vert &>( \delta - 2 \eps -
 \alpha ( 3/ 2 + 2\eps)- M\alpha ( 2\eps + 1)( 1 + 3 \sqrt{1+\alpha^2}/2))
\eps R	
\end{align}
for all $\mu \in M_q(\eps), $  $ e \in H(\mu)\cap \S^{n-1},$ and $
 \eta \in M_p(\mu,e,\eps).$
	Finally,
	\begin{align} \label{angle_bound_1}
	\va \li H(\eta),H(\mu) \ri &\leq \va \li H(\eta),H(p) \ri + \va\li H(p),H(q)\ri + 
\va \li H(q),H(\mu)\ri\notag\\
	&< M\alpha + \omega(\delta) + 2 M\alpha + M\alpha = \omega(\delta) + 4 M \alpha,
	\end{align}
	since we have \eqref{parallel_angle} and $\mu \in M_q(\eps)$ as well as
 $\eta \in M_p(\mu,e,\eps)$. 
	
	For $\mu \in M_q(\eps)$ and $e \in H(\mu)\cap \S^{n-1}$ 
the numerator $L_\tau$ of the energy density of $E^\tau$  satisfies $L_\tau
	(\mu,\eta,H(\mu),H(\eta))\ge  F^\tau ( \mu,\eta,e)$ for all $\eta\in B_R(p),$
where $F_\tau$ is given by \eqref{eq:new-Ftau}.
	In particular, for $\eta \in M_p(\mu,e,\eps)$, we may use \eqref{scalar_prod_1},
 \eqref{norm_bound_1}, \eqref{proj_bound_1}, and \eqref{angle_bound_1} to conclude
that the energy density $L_\tau(\mu,\eta,H(\mu),H(\eta))/|\mu-\eta|^{2m}$ is bounded
from below by 
	\begin{align*}
	 & {\vert \mu-\eta\vert^{-2m}} 
\big( 2 {\vert \langle e , \mu-\eta\rangle \vert }{\vert \mu-\eta\vert^{-2}} \cdot 
 \vert \Pi_{H(\eta)^\perp}(\mu-\eta) \vert -  \vert \Pi_{H(\eta)^\perp}(e) 
\vert \big) ^{(1+\tau)m} \notag\\
	&\hspace{-0.15cm}\geq  { \big(( 2 \eps +1) ( 1 +  3 
\sqrt {1+\alpha^2}/2) \eps R\big)^{-2m} }  
	\cdot \Big[  \frac { \li  1  - \li \omega(\delta) + 3M\alpha\ri \li 
 3  +4\eps\ri \li 1+\alpha\ri-4\eps\ri \eps^2 R^2 } {( \li 2 \eps +1 \ri 
( 1+ 3  \sqrt{1+\alpha^2}/2) \eps R ) ^{2}} \notag \\
	&\hspace{-0.15cm}\cdot\big\{ \delta - 2 \eps - \alpha \li  3/2  +2\eps \ri -
 M\alpha \li 2\eps+1\ri ( 1 +  3  \sqrt{1+\alpha^2}/2)\big\}- \li \omega(\delta)+ 
4M\alpha \ri \Big]^{(1+\tau)m}.
	\end{align*}
	We define $ {f(\eps, \delta)}/{\li\eps R\ri^{2m}}$ to be the right hand side of 
the last estimate. Since $\eps \leq  \delta /{500}$ one obtains by means of
\eqref{alpha_M_connection} $
	f(\eps,\delta) > \li \frac {100}{255}\ri^{2m} \cdot \li \frac
 \delta {100}\ri^{(1+\tau)m}.$  
	Integrating the energy density  over the Cartesian product of
$M_q(\eps)\subset\Sigma \cap B_{\eps^2R}(q)$ and 
$M_p(\mu,e,\eps)\subset \Sigma \cap B_{2\eps R}(p)$ (see \eqref{Mp_inclusion_1}),
 we find by virtue of 
\eqref{Mq_measure_1} and \eqref{Mp_measure_1} the strict inequality 
\eqref{parallel_energy}
with the lower bound
	$c_1(K,\eps,\delta,\tau,m)$ as stated in  
Lemma \ref{lem:dist_sigma_to_plane_1}.
 	\end{proof}

\begin{proof}[Proof of Lemma \ref{lem:dist_sigma_to_plane_2}]
Again fix $\delta\in (0,1)$ and $\varepsilon \in (0,\delta/500].$
Analogously to the setting of Lemma \ref{lem:dist_sigma_to_plane_1}, 
we obtain $p,q \in \tilde K$, and for $M_q(\eps):=E_{\alpha,M}(q) 
\cap B_{\eps^2R}(q)$ we have
	\begin{align} \label{Mq_measure_2}
	\HM^m\li M_q(\eps)\ri \geq c_{\tilde K}\li \eps^2 R\ri^m.
	\end{align}
	Since $\va \li H(p),H(q)\ri \geq \omega(\delta) + 2 M\alpha$, we have
for arbitrary $\mu \in M_q(\eps)$
	\[ \va \li H(\mu),H(p)\ri \geq \va\li H(p),H(q)\ri - \va \li H(q),H(\mu) \ri >\omega(\delta)+ M\alpha, \]
	so that there exists by Lemma \ref{lem:angle_formula} a vector
 $e^\ast = e^\ast(\mu) \in H(\mu)\cap \S^{n-1}$, such that
	\begin{align} \label{e_proj_2}
	 \vert \Pi_{H(p)^\perp}(e^\ast) \vert > \omega(\delta)+M\alpha.
	\end{align}
	We claim that there is a vector $v=v(\mu,e^\ast)\in H(p)\cap \S^{n-1}$, with
	\begin{align} \label{v_def_2}
	\left \langle \Pi_{H(p)}(e^\ast),v \right \rangle =0.
	\end{align}
	Indeed, if $\Pi_{H(p)}(e^\ast)=0$ then any $v \in H(p)\cap\S^{n-1}$ satisfies
identity \eqref{v_def_2}, 
and if $\Pi_{H(p)}(e^\ast)\neq 0$ then 
we have $\li \R e^\ast\ri^\perp \cap H(p)\neq \{0\}$ 
because $\dim(\R e^\ast)^\perp=n-1$ and $\dim H(p)=m\geq2$, so that one can 
choose $v \in \li \R e^\ast\ri^\perp \cap H(p)\cap S^{n-1}$ to satisfy 
\eqref{v_def_2}.
	
	Now compute
	\begin{align} \label{eta_exists_condition_2}
	 \vert \Pi_{p+H(p)} ( \mu +  R  v/2) - p  \vert &=  \vert 
\Pi_{H(p)}( \mu +  R  v/2 - p)  \vert 
	\leq \vert \mu - q \vert + \vert q -p \vert +  R  
 \vert \Pi_{H(p)}(v) \vert/2\notag\\
	& <\li  1/ 2 + \eps + \eps^2\ri R <R .
	\end{align}
	By Definition \ref{def:admissible_sets} there is a point
$y \in D_p$, such that $\left \vert y - \Pi_{p+H(p)}\li \mu +  R  v/2\ri \right 
\vert < \eps^2 R$, together with a corresponding point
 $\eta_y=\eta_y(\mu,e^\ast) \in \Sigma \cap C_p(\alpha,H(p))$ 
with $\Pi_{p+H(p)}(\eta_y)=y=p+\Pi_{H(p)}(\eta_y-p)$. Therefore,
	\begin{align}\label{projection_etay_p}
	\Pi_{H(p)}\li \eta_y - p \ri = y-p = f_y +\Pi_{H(p)}\li \mu+ R  v/2 -p\ri,
	\end{align}
	with $f_y:= y-\Pi_{H(p)}\li \mu+ R  v/2\ri$. In other words,
	\begin{align}\label{projection_etay_mu}
	\Pi_{H(p)}\li \eta_y-\mu\ri = \Pi_{H(p)}\li  R  v/2 \ri + f_y = 
 R  v/2 + f_y.
	\end{align}
	In particular, by \eqref{eta_exists_condition_2} and 
\eqref{projection_etay_p} one finds
	\begin{align} \label{proj_etay_norm}
	 \vert \Pi_{H(p)}\li \eta_y-p\ri  \vert 
\leq \vert f_y\vert +  \vert \Pi_{H(p)} \li \mu + 
 R  v/2 -p\ri  \vert  < \li  1/ 2 + \eps + 2\eps^2\ri R, 
	\end{align}
	and
	\begin{align}\label{proj_etay_p_norm_unten}
	 \vert \Pi_{H(p)}\li \eta_y-p\ri  \vert &\geq  \vert \Pi_{H(p)}\li \mu +
 R  v/2 - p \ri  \vert - \vert f_y\vert \\
	&\geq  R  \vert \Pi_{H(p)}(v)  \vert /2- \vert \mu-q\vert -
 \vert q-p\vert -\vert f_y\vert
	 >\li  1/ 2 - \eps -2\eps^2\ri R.\notag
	\end{align}
Inequality \eqref{proj_etay_norm} together with the fact that 
$\eta_y \in C_p(\alpha,H(p))$ implies
	\begin{align}\label{proj_etay_ortho_norm}
	 \vert \Pi_{H(p)^\perp}\li \eta_y-p\ri \vert \leq \alpha 
 \vert \Pi_{H(p)}( \eta_y-p)  \vert < \alpha (  1/ 2 + 
\eps + 2 \eps^2) R.
	\end{align}
With $\eta_y-p=\Pi_{H(p)}(\eta_y-p)+\Pi_{H(p)^\perp}(\eta_y-p)$ one arrives at
	\begin{align}\label{etay_p_dist}
	(  1/ 2 - \eps - 2\eps^2) R < \vert \eta_y-p\vert < \sqrt {1+\alpha^2} 
(  1/ 2 + \eps+2\eps^2) R.
	\end{align}
	Define $M_p(\mu,e^\ast,\eps):=E_{\alpha,M}(p)\cap 
B_{\eps^2R}(\eta_y(\mu,e^\ast))$ which satisfies
	\begin{align} \label{Mp_inclusion_2}
	M_p(\mu,e^\ast,\eps) \subset \Sigma \cap B_R(p),
	\end{align}
	because $\vert \eta-p\vert \leq \vert \eta-\eta_y\vert +
\vert \eta_y-p\vert < \sqrt{1+\alpha^2}\li 1/ 2 + \eps + 3\eps^2\ri R <R$ 
for $\eta\in M_p(\mu,e^*,\varepsilon)$
since $\sqrt{1+\alpha^2}\leq 1+\alpha < {51}/{50}$
by \eqref{alpha_M_connection}.
 By virtue of Definition \ref{def:admissible_sets} one has also
	\begin{align}\label{Mp_measure_2}
	\HM^m\li M_p \li \mu,e^\ast,\eps \ri \ri \geq c_{\tilde K} 
\li \eps^2 R\ri^m.
	\end{align}
	Now,  by $\vert e^\ast\vert=1$, \eqref{v_def_2}, and 
\eqref{projection_etay_mu} one uses $\Id=\Pi_{H(p)}+\Pi_{H(p)^\perp}$ to estimate
	\begin{align*}
	 \vert \langle e^\ast, \mu-\eta_y\rangle  \vert 
	&\leq  \vert  \langle \Pi_{H(p)}(e^\ast), \Pi_{H(p)}(\mu-\eta_y)
 \rangle  \vert + \vert  \langle  \Pi_{H(p)^\perp}(e^\ast), 
\Pi_{H(p)^\perp}(\mu-\eta_y) \rangle  \vert\\
	&\leq \vert f_y\vert + \vert \mu-q\vert + \vert q-p\vert + 
\vert \Pi_{H(p)^\perp}(p-\eta_y) \vert,
	\end{align*}
	so that we obtain by means of \eqref{proj_etay_ortho_norm}
	$ \left \vert \langle e^\ast,\mu-\eta_y\rangle \right \vert < 
\li \li 2\eps^2+\eps\ri \li 1+\alpha\ri +  \alpha/ 2 \ri R.$
	Consequently, for an arbitrary $\eta \in M_p(\mu,e^\ast,\eps)$ one finds
	\begin{align}\label{scalar_prod_bound_2}
	\left \vert \langle e^\ast,\mu-\eta\rangle\right\vert \leq \left 
\vert \langle e^\ast,\mu-\eta_y\rangle \right \vert + \vert \eta_y-\eta\vert 
&< \li \li 3\eps^2+\eps\ri\li1+\alpha\ri+ \alpha / 2 \ri R.
	\end{align}
	On the one hand, we obtain
	\begin{align} \label{norm_bound_oben_2}
	\vert \mu-\eta\vert &\leq \vert \mu-q\vert + \vert q - p \vert +\vert p 
-\eta\vert
	< \eps^2 R + \eps R + \sqrt{1+\alpha^2}\li  1/ 2 + \eps + 3\eps^2\ri R 
\notag\\
	&< \sqrt{1+\alpha^2}\li  1/ 2 + 2 \eps +4\eps^2\ri R,	
	\end{align}
	and, on the other hand,
	\begin{align} \label{norm_bound_unten_2}
	\vert \mu-\eta\vert &\geq \vert \eta_y - p \vert - \vert \eta_y-\eta\vert - \vert p-q\vert - \vert \mu-q\vert\notag
	\\ &>\li  1/ 2 -\eps-2\eps^2\ri R - \eps^2R-\eps R-\eps^2 R
	=\li  1/ 2 - 2\eps-4\eps^2\ri R,
	\end{align}
	due to \eqref{etay_p_dist}. Moreover, according to 
\eqref{proj_etay_ortho_norm} and \eqref{norm_bound_oben_2},
	\begin{align}\label{proj_bound_2}
	 \vert \Pi_{H(\eta)^\perp}&(\mu -\eta) \vert \leq 
 \vert ( \Pi_{H(\eta)^\perp}-\Pi_{H(p)^\perp})(\mu-\eta) \vert + 
 \vert \Pi_{H(p)^\perp}(\mu-p) \vert +  \vert \Pi_{H(p)^\perp}(p-\eta_y)
 \vert \notag\\
	&\hspace{2cm}+  \vert \Pi_{H(p)^\perp}(\eta_y-\eta) \vert\notag\\
	&\leq \va \li H(\eta),H(p)\ri \vert \mu-\eta\vert + \vert \mu-q\vert + 
\vert q-p\vert +  \vert \Pi_{H(p)^\perp}(p-\eta_y)
 \vert + \vert \eta_y-\eta \vert\notag\\
	&<M\alpha \sqrt{1+\alpha^2}(  1/ 2 +2\eps+4\eps^2) R + 
\eps^2 R+ \eps R + \alpha (  1/ 2 +\eps + 2\eps^2 ) R+ \eps^2 R\notag\\
	&<\alpha ( M\sqrt{1+\alpha^2}+1)( 1/ 2 + 2\eps+ 4\eps^2) R +
 ( 2\eps^2+\eps) R.
	\end{align}
	Finally, since $\eta \in M_p(\mu,e^\ast,\eps)$ and by \eqref{e_proj_2} we obtain
	\begin{align}\label{angle_bound_2}
	 \vert \Pi_{H(\eta)^\perp}(e^\ast) \vert \geq 
 \vert \Pi_{H(p)^\perp}(e^\ast) \vert - \va \li H(\eta),H(p)\ri >\omega(\delta)+M\alpha-M\alpha=\omega(\delta).
	\end{align}
	For fixed $\mu \in M_q(\eps)\subset B_{\eps^2R}(q)$ 
the numerator $L_\tau$ of the energy density satisfies 
	$ L_\tau\li \mu,\eta,H(\mu),H(\eta)\ri\ge  F^\tau\li \mu,\eta,e\ri$ for 
all $e\in H(\mu)\cap\S^{n-1}$ and all $\eta \in B_R(p).$
	In particular, for 
$\eta \in M_p(\mu,e^\ast,\eps)\subset B_R(p)$, 
we can use \eqref{scalar_prod_bound_2}, \eqref{norm_bound_oben_2}, 
\eqref{norm_bound_unten_2}, \eqref{proj_bound_2}, and \eqref{angle_bound_2}
 in order to conclude that the energy density $L_\tau(\mu,\eta,H(\mu),H(\eta))/|\mu-\eta|^{2m}$
is bounded from below by
	\begin{align*}
	 &{\vert \mu-\eta\vert^{-2m}}\li  \vert \Pi_{H(\eta)^\perp}(e^\ast)
\vert - 2  {\vert \langle e^\ast,\mu-\eta\rangle\vert}{\vert \mu-\eta\vert^{-2}} 
\cdot\vert\Pi_{H(\eta)^\perp}(\mu-\eta)\vert \ri^{(1+\tau)m}\\
	&\hspace{-0.17cm}\geq  {( \sqrt{1+\alpha^2} ( 1/ 2 + 2\eps+4\eps^2 )
 R)^{-2m}}
 \cdot \Big[ \omega(\delta) -{\li \li 1/ 2 -2\eps-4\eps^2\ri R\ri^{-2}}\cdot\\
	&\hspace{-0.15cm} \Big\{{2((3\eps^2+\eps)( 1+\alpha ) +  \alpha / 2 )
R^2
( \alpha ( M\sqrt{1+\alpha^2}+1)( 1/ 2 + 2\eps+4\eps^2) + 
 2 \eps^2+\eps )}  \Big\}
\Big]^{(1+\tau)m}.
	\end{align*}
	We define $ {g(\eps,\delta)}/{R^{2m}}$ to be the right hand side of 
	this 
inequality. Then, as $\eps \leq  {\delta}/{500}$ one obtains by 
 means of \eqref{alpha_M_connection} the estimate $
	g(\eps,\delta)>   \frac {(1.9)^{(3+\tau)m}}{10^{5(1+\tau)m}} 
\cdot \delta^{(1+\tau)m}$. 
	Now we can integrate the energy density over 
the Cartesian product of $M_q(\eps) \subset \Sigma \cap B_{\eps^2R}(q)$
 and $M_p(\mu,e^\ast,\eps)\subset \Sigma \cap B_R(p)$ (see \eqref{Mp_inclusion_2}), 
in order to establish with the help of
\eqref{Mq_measure_2} and \eqref{Mp_measure_2} the desired inequality 
\ref{transversal_energy} with the constant $c_2
	(K,\eps,\delta,\tau,m) $ as stated in
 Lemma \ref{lem:dist_sigma_to_plane_2}.
\end{proof}

\begin{cor} \label{cor:theta_bound}
	Let $\delta \in (0,1)$, $\alpha,M>0$ satisfy 
\eqref{alpha_M_connection}, and suppose $\Sigma \in \Aam$, $2\leq m\leq n$, has 
locally finite energy $E^\tau$ for some $\tau >-1$.
Then, for all compact sets $K\subset \Sigma$ and all $p\in K$, we have
\begin{align}\label{eq:theta_bound}
\theta_\Sigma\li p,G_p,r\ri \leq \delta 
\Foa r \in \li 0,\min\{r_K,\rho_K\}\right],
	\end{align}
	where $\rho_K$ denotes the radius of Corollary 
\ref{cor:dist_plane_to_sigma} and $r_K$ as well as $G_p$ are as stated in 
Theorem \ref{thm:dist_sigma_to_plane}. 
In particular, 
$\Sigma$ is $(m,\delta)$-Reifenberg-flat.
\end{cor}
\begin{proof}
	According to Theorem \ref{thm:dist_sigma_to_plane}, one finds
	$ \sup_{p \in K} \beta_\Sigma(p,G_p,r) \leq \delta \Foa r \in (0,r_K].$
	Moreover, we have $G_p=H(p)$ for all $p\in\Sigmaa$ and 
$G_q=\lim_{i\to\infty}H(p_i)$ for a sequence $(p_i)_{i\in\N}\subset \Sigmaa$ 
approximating $q\in K\setminus \Sigmaa$. Consequently, Lemma 
\ref{cor:dist_plane_to_sigma} guarantees
	\[  \sup_{\xi \in (p+G_p)\cap B_r(p)}
 \dist\li \xi,\Sigma\cap B_r(p)\ri \leq  {2\alpha r}/{\sqrt{1+\alpha^2}}<\delta r\quad
 \Foa r \in(0, \rho_{K}], \]
 where we used  \eqref{alpha_M_connection} for the last inequality. 
In view of \eqref{theta_def} this finishes 
the proof.
\end{proof}
Now Reifenberg's famous topological disk lemma 
\cite{reifenberg_1960,simon_1996a,hong-wang_2010} implies that finite energy sets
are topological manifolds, a result that we do not rely on in the following
sections.
\begin{cor} \label{cor:top_mfd} 
	Let $m,n\in\N$ with $2\leq m\leq n$. For
any 
$\kappa \in(0,1)$ there is a constant $\delta_0(m,\kappa)\in (0,1)$ 
such that any set $\Sigma\in \Aam$ with locally finite energy $E^\tau$
for some $\tau\in (-1,\infty)$,
where $\alpha,M>0$ satisfy \eqref{alpha_M_connection} 
for $\delta=\delta_0$, is an $m$-dimensional
 topological manifold which is locally bi-H\"older homeomorphic to the 
$m$-dimensional flat unit disk $B_1(0)\subset \R^m$ with H\"older constant $\kappa$. 
\end{cor}

\subsection{Lipschitz and $C^k$-submanifolds}\label{sec:lip-c1-mfds}
The  local structure of the admissibility class $\Aam$ allows us to directly derive
 the stronger Lipschitz manifold statement of  Theorem \ref{thm:self-avoidance} \emph{without}
 using Reifenberg's topological disk theorem at all. In order to do so, we take a closer look at the projection of $\Sigma \cap B_r(p)$ onto the affine plane $p+G_p$. The following arguments are based solely on  the Reifenberg-flatness of $\Sigma$.

First, we notice injectivity of the orthogonal projection onto approximating planes for $(m,\delta)$-Reifenberg-flat sets.

\begin{lem}[Injectivity of  projection] \label{lem:injective_projection}
	Suppose $\delta\in ( 0 ,  1/ 2),$ and $\Sigma\subset \R^n$ is an $(m,\delta)$-Reifenberg-flat set. Let $p \in K\subset \Sigma$, where $K$ is compact, and suppose that $\eta \in K \cap B_r(p)$ for some $r \in (0,r_0(K)]$ satisfies 
	\begin{align} \label{inj_angle_condition}
	\va \li F_\eta(\rho,\delta), F_p(r,\delta)\ri+\delta <1 \Foa \rho \in (0,r),
	\end{align}
	for the radius $r_0(K)$ and the $m$-planes $F_\eta(\rho,\delta) $ and
	$F_p(r,\delta)$ as in Definition \ref{def:reifenberg-flatness}.
	Then for every $\xi \in \Sigma \cap B_r(p)\setminus \{\eta\}$ one has
	$
		\Pi_{p+F_p(r,\delta)}(\xi) \neq \Pi_{p+F_p(r,\delta)}(\eta).$
	\end{lem}
Notice that if for each $x\in (p+F_p(r,\delta))$ there is a point
$\eta_x\in K\cap B_r(p)$ with $\Pi_{p+F_p(r,\delta)}(\eta_x)=x$ satisfying
\eqref{inj_angle_condition}, then the projection $\Pi_{p+F_p(r,\delta)}$
restricted to $\Sigma\cap B_r(p)$ is injective.
\begin{proof}
	Assuming the contrary for some $\xi \in \Sigma \cap B_r(p)\setminus \{\eta\}$, we can estimate for $F:=F_p(r,\delta)\in \mathscr{G}(n,m)$
	\begin{align*}
	\vert \eta-\xi\vert &= \left \vert \Pi_F(\eta-\xi) + \Pi_{F^\perp}(\eta-\xi)\right \vert = \left \vert \Pi_{F^\perp}(\eta-\xi)\right \vert 
	\leq \left \vert \Pi_{F^\perp}(\eta-p) \right\vert + \left\vert \Pi_{F^\perp}(\xi-p)\right \vert \\
	&=\dist\li \eta,p+F\ri + \dist \li \xi,p+F\ri \leq 2\delta r <r,
	\end{align*}
	since $\delta < 1/ 2$. Therefore, there exists an integer $N\in \N$, such that
	\begin{align}\label{inj_lem_radius}
	r_k:=\li 1 +  1 /k\ri \cdot \vert \eta-\xi \vert <r \leq r_0(K) \quad\Foa k \geq N.
	\end{align}
	Obviously, $\xi \in \Sigma \cap B_{r_k}(\eta)$, and therefore, by the Reifenberg-flatness of $\Sigma$, used in the point $\eta$ for the radius $r_k$ with the approximating plane $Q:=F_\eta(r_k,\delta)$ 
	\begin{align} \label{inj_lem_dist}
	\dist\li \xi,(\eta+Q)\cap B_{r_k}(\eta)\ri \leq \delta r_k,
	\end{align}
	so that we can estimate using \eqref{inj_lem_radius} and \eqref{inj_lem_dist}
	\begin{align*}
	\vert \eta-\xi\vert &=\left \vert \Pi_{F^\perp}(\eta-\xi)\right \vert = \left \vert \Pi_{F^\perp}(\eta-\xi) - \Pi_{Q^\perp}(\eta-\xi) + \Pi_{Q^\perp}(\eta-\xi)\right \vert\\
	&\leq \va\li F,Q\ri \cdot \vert \eta-\xi \vert + \dist \li \xi,\li \eta+Q\ri \cap B_{r_k}(\eta)\ri\\
	& \le \big[ \va \li F,Q\ri + \li 1 + 1/ k\ri \delta \,\big] \cdot \vert \eta-\xi\vert,
	\end{align*}
	and the right-hand side is by assumption \eqref{inj_angle_condition} 
	(for $\rho:= r_k$) strictly less than $\vert \eta-\xi\vert$ for $k$ sufficiently large, 
	which yields a contradiction.
\end{proof}

For sufficiently small $\delta>0$ and radius $r>0$, the orthogonal projection of $\Sigma \cap B_r(p)$ onto good approximating affine planes contains a whole $m$-dimensional disk.
\begin{prop}[Surjectivity of projection] \label{lem:surjective_projection}
	There exists a $\delta_1=\delta_1(m)\in (0,1/2)$, such that for all closed $(m,\delta)$-Reifenberg-flat sets $\Sigma\subset \R^n$ with $\delta\leq \delta_1$ and all compact
	subsets $K\subset\Sigma$ there exists a radius $\rho_1(K)\in (0,r_0(K)]$ such that
	\[ \li p + F_p(r,\delta)\ri \cap B_{ r /4}(p) 
\subset \Pi_{p+F_p(r,\delta)} \li \Sigma \cap 
B_{ r/ 2}(p)\ri \Foa p \in K \AND r\in(0,\rho_1(K) ]. \]
\end{prop}
\begin{proof}
Up to isometry the composition $
\Psi:=\Pi_{p+F_p(r,\delta)}\circ\tau:
p+F_p(r,\delta)\to p+F_p(r,\delta)$ satisfies the assumptions of
\cite[Prop. 2.5]{kolasinski-etal_2018} to yield the claim. Here, $\tau$ is
the mapping constructed in the following technical lemma  proven 
 in \cite{kaefer_2019}. 
\end{proof}
\begin{lem}[{\cite[Lemma 3.7]{kaefer_2019}}] \label{lem:reifenbergfunction}
	There exists a $\delta_*=\delta_*(m)>0$ such that for every closed
	 $(\delta,m)$-Reifenberg-flat set $\Sigma \subset \R^n$ with 
	 $\delta\leq \delta_*$ and $x \in \Sigma$ there 
	is a radius $R_0=R_0(x,\delta, \Sigma)>0$ and a constant $C_*=C_*(m)$
		such that for all $F \in \mathscr{G}(n,m)$ with
	$\theta_\Sigma(x,F,r)\le \delta$ for $r\le R_0$ there exists
	 a continuous function  $
	 \tau \colon (x+F) \cap \overline{B_{ {15r}/{16}}(x)} \to \Sigma \cap \overline{B_r(x)}  $
	with $
	\vert \tau(y)-y \vert \leq C_*r\delta \leq  5 r/{144}  \Foa y \in (x+F) \cap \overline{B_{15r/16}(x)}.$
\end{lem}

We can now use the local bijectivity of the projection established in
Lemma \ref{lem:injective_projection}  
and Proposition \ref{lem:surjective_projection} to prove that every
admissible set with locally finite energy possesses a local graph 
representation\footnote{Independent of the admissibility class, such a statement holds true for
every $(m,\delta)$-Reifenberg-flat set $\Sigma \subset \R^n$ with
$\delta\leq \delta_1$ such that for all $x \in (p+F_p(r,\delta))\cap B_{r/4}(p)$ there is an $\eta \in \Sigma \cap B_{r/2}(p)$ with $\Pi_{p+F_p(r,\delta)}(\eta)=x$ and the approximating planes $F_\eta(\rho,\delta)$ allow an estimate like \eqref{inj_angle_condition} with a uniform upper bound $C<1$ for all such $\eta$ and $\rho\in(0,r)$; see \cite[Ch. 3.1]{kaefer_2020}.
},
which in particular implies Theorem \ref{thm:self-avoidance}. 
\begin{thm} \label{thm:lipschitz_graph} 
	Let $2\le m\le n$. There exist  constants $C=C(m)$ and $\delta_2=\delta_2(m)\in(0,\min\{\delta_1,1/C\} )$ 
	such that for every  $\Sigma \in \Aam$ with
	$\alpha,M>0$ satisfying \eqref{alpha_M_connection} for some  $\delta \in \li 0,\delta_2\right]$, with locally finite energy $E^\tau$ for some $\tau>-1$ the following
	holds.
	For  all compact sets $K\subset \Sigma$  there is a radius 
	$\rho_2=\rho_2(K)\in \li 0, \min\{r_K,\rho_K,\rho_1(K)\}\right]$ 
	such that for all $p \in K$ there is a function $u_p\in C^{0,1}(G_p,G_p^\perp)$ with
$u_p(0)=0$ and $ \lip u_p \leq  {C \delta}/{(1- C \delta)},$  such that
	\begin{align}\label{lip_thm_graph}
	\Sigma \cap B_{ {\rho_2}/{4}}(p)= \li p + \graph u_p\ri 
	\cap B_{ {\rho_2}/{4}}(p),
	\end{align}
		where $\rho_K$ denotes the radius of Corollary 
\ref{cor:dist_plane_to_sigma}, and $r_K$ as well as $G_p$ are as stated in 
Theorem \ref{thm:dist_sigma_to_plane}. 
	In particular, $\Sigma$ is a $C^{0,1}$-submanifold of $\R^n$.
\end{thm}
\begin{proof}
	We define $\tilde K := \Sigma \cap \overline{B_{ {R_K}/2}(K)}$ and\footnote{Notice
	as before 
	that the compact set $\tilde{K}$ is solely determined by $K$ itself, so that all the
	constants for $\tilde{K}$ can be considered to actually depend on $K$.}  $\rho_2=
	\rho_2(K) := \min\{r_{\tilde K},\rho_{\tilde K},\rho_1(\tilde K)
	\}\leq \min\{r_K,\rho_K,\rho_1(K) \}$. 
	Applying Corollary \ref{cor:theta_bound}, one finds
	\begin{align}\label{lip_thm_theta}
	 \theta_{\Sigma}\li {p},G_{{p}},r\ri 
	 \leq \delta \Foa r \in \li 0,\rho_2\right ],\,{p}\in\tilde{K}.
	\end{align}
	For $\delta\leq \delta_1$, $p\in \tilde K$ and $r\leq \rho_2 \leq 
	\rho_1(\tilde K)$, by Proposition \ref{lem:surjective_projection} one has
	\begin{align} \label{lip_thm_surj}
	\li p+G_p \ri \cap B_{ r /4}(p) \subset \Pi_{p+G_p}\li \Sigma \cap B_{ r /2}(p)\ri,
	\end{align}
	which guarantees for each $x \in \li p+G_p\ri \cap B_{ r/ 4}(p)$
	 the existence of a point $q_x\in\Sigma \cap B_{ r/ 2}(p)$ with $\Pi_{p+G_p}(q_x)=x$. In particular, for $p \in K$, one finds $\Sigma \cap B_{ r/ 2}(p)\subset \tilde K$, since $r\leq \rho_2\leq r_K\leq R_K$, so that by \eqref{lip_thm_theta} and Lemma \ref{lem:reifenberg_angles} for $r_1=r_2:=r$, $d_1=d_2:=\delta$,
	  $F_1:=G_p$, $F_2:=G_q$, $p_1:=p$, and $p_2:=q$, we can estimate
	  $
	\va \li G_p,G_q\ri \leq   {6\delta\tilde  C}/{(1-2\delta)}
	$ for all $ q\in\Sigma \cap B_{ r/ 2}(p),$
	with $\tilde C=\tilde C(m)>\sqrt 2$. Choosing $\delta$ sufficiently small, one finds a constant $ C= C(m)>0$ such that
	\begin{align} \label{lip_thm_angle_2}
	 \va \li G_p,G_q\ri + \delta \leq {6\delta \tilde C}/{(1-2\delta)} +\delta \leq  C \delta <1
	\quad\Foa q\in\Sigma\cap B_{r/2}(p).
	\end{align}
	 Lemma \ref{lem:injective_projection} for  $\eta:= q_x$, $F_\eta(\rho,\delta):= G_{q_x}$ for all $\rho\leq r$, and $F_p(r,\delta):= G_p$ implies together with \eqref{lip_thm_surj} that $\Pi_{p+G_p \mid_{\Sigma \cap B_{ r/ 2}(p)}}$ is injective and $\Pi_{p+G_p\mid_{\Sigma_p(r)}}:\Sigma_p(r)\to (p+G_p)\cap B_{r/4}(p)$
	  is surjective,
	 where 
	 \begin{equation}\label{eq:Sigma_p}
	 \Sigma_p(r):=  
	 \{ \xi \in \Sigma \cap B_{ r /2}(p) \colon  \Pi_{p+G_p}(\xi) \in\li p+G_p\ri\cap B_{ r/ 4}(p) \}.
	 \end{equation}
	  Therefore $\Pi_{p+G_p\mid_{\Sigma_p(r)}}\colon \Sigma_p(r)
	  \to \li p+G_p\ri \cap 
	  B_{ r/ 4}(p)$ is bijective and
	\begin{align}\label{lip_thm_bijective}
	\big( \Pi_{p+G_p \mid_{\Sigma_p(r)}}\big)^{-1} \colon \li p + 
	G_p\ri \cap B_{ r/ 4}(p) \to \Sigma_p(r)
	\end{align}
	is well-defined. Obviously,
		\begin{align} \label{lip_thm_Sigma_p}
	\Sigma \cap B_{ {r}/{4}}(p)\subset \Sigma_p(r) \Foa r \leq \rho_2.
	\end{align} 
	With \eqref{lip_thm_bijective}, we can define the function $
		u_p \colon G_p\cap B_{ r/ 4}(0) \to G_p^\perp $ by means of
		$u_p(x):= \Pi_{G_p^\perp} 
	\big( \big( \Pi_{p+G_p \mid_{\Sigma_p(r)}}
	\big)^{-1} \li x+p\ri-p\big)$
	satisfying $u_p(0)=0$.
		Then for all $q \in \Sigma_p(r)$ and 
		$x_q:=\Pi_{G_p}(q-p)\in G_p\cap B_{ r/ 4}(0)$, one finds
	$
	q= p + \Pi_{G_p}(q-p) + \Pi_{G_p^\perp}(q-p) = p + x_q +  \Pi_{G_p^\perp}
	\big(\big( \Pi_{p+G_p \mid_{\Sigma_p(r)}}\big)^{-1} 
	\li x_q+p\ri-p\big)= p+x_q+u_p(x_q).
	$
	Consequently,
	\begin{align}
	\Sigma_p(r)\subset (p+\graph u_p)\cap B_{r/2}(p),\label{eq:new-eq}\\
	\Sigma_p(r)\cap B_{r/4}(p) =(p+\graph u_p)\cap B_{r/4}(p),
	\label{lip_thm_Sigma_p_graph}
	\end{align}
	which by \eqref{lip_thm_Sigma_p} for $r=\rho_2$ implies \eqref{lip_thm_graph}.
	 Moreover, for $\eta,\mu\in\Sigma_p(\rho_2)$ and corresponding points
	 $x_\eta,x_\mu\in G_p \cap B_{ r/ 4}(0)$, one has $\eta \in B_{\li 1 +  1/k
	  \ri \vert \eta-\mu\vert}(\mu)$, so that \eqref{lip_thm_theta} implies
	\begin{align*}
	 \vert u_p(x_\eta)\!-\!u_p(x_\mu) \vert &=  \vert\Pi_{G_p^\perp}
	 (\eta\!-\!p)\!-\!\Pi_{G_p^\perp}(\mu\!-\!p) \vert =  
	 \vert \Pi_{G_p^\perp}(\eta\!-\!\mu)\!+\!\Pi_{G_\mu^\perp}
	 (\eta\!-\!\mu)\! -\! \Pi_{G_\mu^\perp}(\eta\!-\!\mu) \vert\\
	&\leq \va \li G_p,G_\mu\ri \vert \eta-\mu \vert + \dist \li \eta,\li \mu+G_\mu\ri \cap B_{\li 1 + 1/ k\ri \vert \eta-\mu\vert}(\mu)\ri\\
	&\leq \va \li G_p,G_\mu\ri \vert \eta-\mu\vert + \delta \li 1+ 1/ k \ri \vert \eta-\mu\vert
	\end{align*}
	for $k$ sufficiently large.
	By taking the limit $k \to \infty$ and \eqref{lip_thm_angle_2}
	for $r=\rho_2$, we obtain
	\begin{align*}
	\vert u_p(x_\eta)-u_p(x_\mu) \vert &\leq  C \delta \cdot \vert \eta-\mu\vert \leq  C\delta \cdot \big(  \vert \Pi_{G_p}(\eta-\mu) \vert +  \vert \Pi_{G_p^\perp}(\eta-\mu) \vert \big)\\
	&=  C \delta \cdot \li \vert x_\eta-x_\mu\vert + 	 \vert u_p(x_\eta)-u_p(x_\mu) \vert\ri.
	\end{align*}
	Absorbing the last term yields for sufficiently small $\delta$ 
	the Lipschitz continuity of $u_p$ on the disk $G_p\cap 
	B_{\rho_2/4}(0)$ with Lipschitz constant 
	 $\lip u_p \leq  { C \delta}/{(1- C\delta)} $. Finally, extending $u_p$ to the whole plane $G_p$ by Kirszbraun's theorem \cite[2.10.43]{federer_1969} concludes the proof. 
\end{proof}
Applying Theorem \ref{thm:lipschitz_graph} to the examples mentioned in the introduction and discussed in Sections \ref{sec:admissible_sets} one finds that all such sets with locally finite energy are embedded Lipschitz submanifolds of $\R^n$ as long as the constants $\alpha,M
>0$ defining their admissibility class $\Aam$ are sufficiently small. In particular, countable collections of sufficiently flat Lipschitz graphs, and all $C^1$-immersions of compact $m$-dimensional $C^1$-manifolds with locally finite 
M\"obius 
energy, are $C^{0,1}$-submanifolds. For both classes, however, we already have a  local graph structure to improve this statement. 
\begin{cor}[Lipschitz graphs with small Lipschitz constant] 
\label{cor:lip_collection}
	Suppose  that the set
	$
	 \Sigma = \overline{ \bigcup_{i \in \N} \li p_i + \graph u_i\ri } $
	 has locally finite energy $E^\tau$ for some $\tau > -1$,
where $u_i \in C^{0,1}(F_i,F_i^\perp)$, $F_i \in \mathscr{G}(n,m)$ 
for $2\leq m\leq n$, $u_i(0)=0$, and $\lip u_i \leq\beta$ for all 
$i \in \N$. If $\beta < 1 /{6400}$     then for all compact subsets 
$K\subset \Sigma$ there exists a radius $\rho_3=\rho_3(K)$ such that 
for every $p \in K$, there exists an index $i=i(p)\in \N$, such that 
	\begin{equation}\label{eq:local-rep}
	\Sigma \cap B_{ {\rho_3}/{5}}(p) = 
	\li p_i+\graph u_i \ri \cap B_{ {\rho_3}/{5}}(p).
	\end{equation}
	\end{cor}
Notice that the Lipschitz constants $\lip u_i$ inherited from the initially
given Lipschitz functions $u_i$ are, in general,
much smaller than the Lipschitz constants of the functions
$u_p$ produced by Theorem \ref{thm:lipschitz_graph}. Let us also note that
a direct application of  that theorem
  would produce local graph representations via  $u_p$, but in order to
  show that these graphs  locally coincide with the graphs of the given functions
  $u_i$ requires their Lipschitz bound $\beta$ to be much smaller than $\delta_2(m)$
  of Theorem \ref{thm:lipschitz_graph} which itself is already very small. In Corollary
 \ref{cor:lip_collection}, however, $\beta$ needs only to be smaller than 
 a universal constant independent of the dimension $m$. 
\begin{proof}[Proof of Corollary \ref{cor:lip_collection}]
Abbreviate $\Gamma_i:=p_i+\graph u_i$ for $i\in\N$.
Since $\beta<1/6400$, we can find\footnote{Choosing $\delta$ smaller than
but close to $1/2$
and then $M=1$ and $\alpha<\delta/100$ is a possible option.}
$\delta\in (0,1/2)$ and then
$\alpha,M>0$ such that
\begin{align} \label{lip_cor_beta}
\beta\leq M\alpha/(16(M+1))\quad\AND\quad  (M+1)\alpha  <  {\delta}/{50}.
\end{align}
Therefore, $\Sigma \in \Aam$, by Proposition \ref{lem:lip_admis}, with
$m$-planes $H(p)=T_p(\Gamma_{i(p)})=T_{p-p_{i(p)}}\graph u_{i(p)}$ for all 
$p\in \Gamma_{i(p)}\cap\Sigma^*$, where $i(p)$ is the 
unique smallest index $j$ such that $p$ is contained in $\Gamma_j$. 
For the proof of the corollary we need three auxiliary lemmas.
\begin{lem}\label{lem:1}
For $l\in\N$ and a compact subset $K'\subset\Sigma$, where $\Sigma$ satisfies
all assumptions of Corollary \ref{cor:lip_collection}, one finds for all $p\in\Gamma_l
\cap K'$ with 
\begin{equation}\label{eq:plus-empty}
B_1(p)\cap\Gamma_l\cap\Gamma_k=\emptyset\quad\Foa k=1,\ldots,l-1
\end{equation}
the local graph representation
\begin{equation}\label{G}
\Sigma\cap B_{\rho'/4}(p)=\Gamma_l\cap B_{\rho'/4}(p)\quad\Fo
\rho':=\min\{r_{K''},\rho_{K''}\}\le 1,
\end{equation}
where $r_{K''},\rho_{K''}$ are the constants of Theorem \ref{thm:dist_sigma_to_plane} 
and Lemma \ref{cor:dist_plane_to_sigma}, respectively, for the larger compact subset $K'':=\overline{B_1(K')}
\cap\Sigma.$
\end{lem}
\begin{proof}
For fixed $p\in\Gamma_l\cap K'$ set $G_\xi:=H(\xi)=T_\xi\Gamma_l$ for all 
$\xi\in \Gamma_l\cap B_1(p)\cap\Sigma^*$ in accordance with  
Proposition \ref{lem:lip_admis}. For $\xi\in\Gamma_l\cap B_1(p)\setminus\Sigma^*$
we choose a sequence $\Sigma^*\cap\Gamma_l\ni\xi_j\to\xi$ as $j\to\infty$,
which is possible because of \eqref{eq:plus-empty}, to define 
 $G_\xi:=\lim_{j\to\infty}H(\xi_j)=\lim_{j\to\infty}T_{\xi_j}\Gamma_l\in\mathscr{G}(n,m)$.
Lemma \ref{lem:angle_lip} implies
\begin{equation}\label{W}
\ang(G_\xi,F_l)\le\beta\AND\ang (G_\xi,G_\mu)\le 2\beta
\Foa \xi,\mu\in\Gamma_l\cap B_1(p).
\end{equation}
The first inequality for $\xi:=p$ in combination with
the
Tilting Lemma \ref{lem:tilting} for $F:=F_l$,
$G:=G_p$, $u:=u_l$, and $\chi:=\beta $ with $\sigma=\chi(1+
\lip u)\le\beta(1+\beta)<1-\sqrt{1+\beta^2}/2$ (since $\beta<1/6400$)
yields
\begin{equation}\label{eq:bigprojetions}
B_{\frac{\rho}2}(p)\cap(p+G_p)\subset
\Pi_{p+G_p}\big(B_\rho(p)\cap \Gamma_l\big)\Foa\rho >0.
\end{equation}
By virtue of the second inequality in
\eqref{lip_cor_beta} we may apply
Corollary \ref{cor:theta_bound}
to conclude that
$\Sigma$ is an $(m,\delta)$-Reifenberg flat set. In particular,
 exploit \eqref{eq:theta_bound} for the compact set
  $K''\subset\Sigma$ to find
 \begin{equation}\label{eq:good-points-theta}
     \theta_\Sigma(\eta,G_\eta,r)\le\delta\quad\Foa\eta\in K'',\, 
     r\in (0,\rho'],
     \end{equation}  
 for $m$-planes $G_\eta$ as stated in Theorem  \ref{thm:dist_sigma_to_plane}.
 Since $B_{\rho'}(p)\cap\Gamma_l\subset K''$ we can use the smallness of
 $\beta$ and \eqref{eq:good-points-theta} 
for 
$\eta:=p$ and any other graph point $\xi\in\Gamma_l\cap B_1(p)$ 
 to 
define approximating planes $F_p(\rho',\delta):=G_p$ and
$F_q(\rho,\delta):=G_q$ for all $\rho\in (0,\rho')$ to verify
assumption \eqref{inj_angle_condition} of Lemma \ref{lem:injective_projection}
for $K''\subset\Sigma$ by means of
\eqref{W}.
Thus, this lemma together with \eqref{eq:bigprojetions} for $\rho:=\rho'/2$
yields the bijectivity of the projection
$\Pi_{p+G_p}|_{\Sigma_p(\rho')}:\Sigma_p(\rho')\to B_{\rho'/4}(p)\cap
(p+G_p)$,
where we set as in the proof of Theorem \ref{thm:lipschitz_graph}
$\Sigma_p(r):=\{\xi\in B_{r/2}(p):\Pi_{p+G_p}(\xi)\in 
B_{r/4}(p)\}$. The bijectivity of $\Pi_{p+G_p}|_{\Sigma_p(\rho')}$
combined with \eqref{eq:bigprojetions} for $\rho:=\rho'/2$
implies in addition that
$
\Sigma_p(\rho')\subset \Gamma_l\cap B_{\rho'/2}(p),
$
which combined with the obvious inclusions $\Sigma\cap B_{\rho'/4}(p)
\subset \Sigma_p(\rho')$ and $\Gamma_{l}\subset
\Sigma$ yields the local graph representation \eqref{G} as desired.
 \end{proof}
Now fix the compact set $K\subset\Sigma$ and $N \in \N$ with $N\ge 3$ 
such that 
$\overline{B_1(K)}\subset B_N(0)$. Consider an index 
$i\in\N$ such that $\Gamma_i\cap \overline{B_1(K)}\not=\emptyset,$
and define  the nested
compact neighbourhoods $K_j\equiv K^i_j:=\Sigma\cap
\overline{B_{N^{i+1-j}}(0)}$ for $j=0,1,\ldots,i$.  
Notice that $K_{j+1}
\subset K_j$ for $j=0,\ldots i-1$. This leads to the second auxiliary lemma.
\begin{lem}\label{lem:2}
Let $j\in\{1,\ldots,i\}.$ If $\Gamma_j\cap\Gamma_l\cap \overline{B_1(K_{j})}
\not=\emptyset$
for some $l\in\{1,\ldots,j-1\}$ such that
$$
B_1(q)\cap\Gamma_l\cap\Gamma_k=\emptyset\Foa q\in\Gamma_l\cap K_{l},
\,k=1,\ldots,l-1,
$$
then $\Gamma_j\cap \overline{B_1(K_{j})}\subset 
\Gamma_l\cap \overline{B_1(K_{j})}.$
\end{lem}
\begin{proof}
Apply Lemma \ref{lem:1} to $\Gamma_l$ for the compact set $K':=K_l$ so that $
K''=\Sigma \cap \overline{B_1(K_{l})}$ and 
$\rho'=\rho'(l)=\min\{r_{K''},\rho_{K''}\}\le 1$, 
to obtain the local graph representation
\begin{equation}\label{G1}
\Sigma\cap B_{\rho'/4}(q)=\Gamma_l\cap B_{\rho'/4}(q)\quad\Foa
q\in\Gamma_l\cap K_{l}.
\end{equation}
Now connect an arbitrary point $ \eta=p_j+y+u_j(y)\in \Gamma_j\cap\Gamma_l\cap
\overline{B_1(K_{j})}$ with any other graph point
$\mu=p_j+x+u_j(x)\in\Gamma_j\cap \overline{B_1(K_{j})}$, 
$x,y\in F_j$, by the Lipschitz
continuous curve $\gamma:[0,1]\to\Gamma_j$ defined as
$\gamma(t):=p_j+tx+(1-t)y+u_j(tx+(1-t)y)$ with $\lip \gamma
\le |x-y|(1+\lip u_j).$
Since $\eta,\mu \in \Gamma_j\cap \overline{B_1(K_{j})}\subset K_l$, one 
finds $\vert \gamma(t)\vert \leq  (\lip\gamma)/2 +1+N^{i+1-j}\leq N^{i+1-l}$ and, consequently, $\gamma(t)\in K_l$ for all $t\in[0,1]$.
Then the set $J:=\{t\in [0,1]:\gamma(t)\in\Gamma_l\}$ contains $t=0$,
is closed since $\gamma$ is bounded and continuous, and relatively open in
$[0,1]$ by virtue of \eqref{G1} for $q:=\gamma(t),$ $t\in J$. Hence $J=[0,1]$ 
and therefore $\mu=\gamma(1)\in\Gamma_l$.
\end{proof}
To take care of possibly overlapping graphs we provide  two mutually excluding
alternatives in the third auxiliary result.
\begin{lem}\label{lem:3}
For any $j\in\{1,\ldots,i\}$ {\bf either} 
$B_1(p)\cap\Gamma_j\cap \Gamma_k=\emptyset$
for all $p\in\Gamma_{j}\cap K_j$, $k=1,\ldots,j-1$, {\bf or} 
there exists $l\in\{1,\ldots,j-1\}$
such that $\Gamma_j\cap \overline{B_1(K_{j})}\subset\Gamma_l\cap 
\overline{B_1(K_{j})}$ and
$B_1(q)\cap\Gamma_l\cap\Gamma_k=\emptyset$ for all $q\in\Gamma_l\cap
K_l$, $k=1\ldots,l-1.$ 
\end{lem}
\begin{proof}
We prove this by induction. For $j=1$ the first 
option is obviously true, because
there simply is no smaller index.
Now assume that the statement is true up to index $j$. If $B_1(p)\cap
\Gamma_{j+1}\cap\Gamma_k=\emptyset$ for all $k=1,\ldots,j$ then we are done.
If $B_1(p)\cap\Gamma_{j+1}\cap\Gamma_\lambda\not=\emptyset$ for some
$p\in\Gamma_{j+1}\cap K_{j+1}$ and some $\lambda\in\{1,\ldots,j\}$ 
then 
\begin{equation}\label{plusplus}
\Gamma_{j+1}\cap\Gamma_\lambda\cap \overline{B_1(K_{j+1})}\not=\emptyset.
\end{equation}
The induction hypothesis  yields either $B_1(q)\cap\Gamma_\lambda\cap\Gamma_k
=\emptyset$ for all $q\in\Gamma_\lambda\cap K_\lambda$, $k=1,\ldots,\lambda-1$,
or there is some index $l\in\{1,\ldots,\lambda-1\}$ such that 
\begin{equation}\label{ppp}
\Gamma_\lambda
\cap \overline{B_1(K_{\lambda})}
\subset\Gamma_l\cap \overline{B_1(K_{\lambda})}
\end{equation}
and $
B_1(q)\cap\Gamma_l\cap\Gamma_k=\emptyset $ for all $ q\in\Gamma_l\cap K_l,$
$k=1,\ldots,l-1.$
In the first case we can apply Lemma \ref{lem:2} to 
$\Gamma_{j+1}$ and $\Gamma_\lambda$ by means of \eqref{plusplus} to obtain
$\Gamma_{j+1}\cap \overline{B_1(K_{j+1})}
\subset \Gamma_\lambda\cap \overline{B_1(K_{j+1})}$. 
In the second case
we may combine \eqref{plusplus} with \eqref{ppp} to arrive at
$\emptyset\not=\Gamma_{j+1}\cap\Gamma_\lambda\cap \overline{B_1(K_{j+1})}
\subset\Gamma_{j+1}\cap\Gamma_l\cap \overline{B_1(K_{j+1})}$, 
so that Lemma \ref{lem:2}
applied to $\Gamma_{j+1}$ and $\Gamma_l$ yields 
$\Gamma_{j+1}\cap \overline{B_1(K_{j+1})}
\subset \Gamma_l\cap \overline{B_1(K_{j+1})}$.
\end{proof}
Now we can finish the proof of Corollary \ref{cor:lip_collection}. 
If $p\in \overline{B_1(K)}\cap
\Gamma_i\subset K_i\cap \Gamma_i$ for some 
$i\in\N$ exploit Lemma \ref{lem:3} for $j=i$ to obtain
two possible cases.

{\it Case 1.}\, $B_1(q)\cap\Gamma_i\cap\Gamma_k=\emptyset$ for all 
$q\in \Gamma_i\cap K_i$, $k=1\ldots,
i-1$.  Then use Lemma \ref{lem:1} for $l:=i$, 
$K':=\overline{B_1(K)}\cap
\Sigma$, $K'':=\overline{B_2(K)}\cap \Sigma$ 
and $\rho':=\min\{r_{K''},\rho_{K''}\}$ 
to obtain the  local graph representation
\begin{equation}\label{G11}
\Sigma\cap B_{\rho'/4}(q)=\Gamma_i\cap B_{\rho'/4}(q)\quad\Foa
q\in\Gamma_i\cap \overline{B_1(K)},
\end{equation}
in particular for the point $p$.

{\it Case 2.}\, $\Gamma_i\cap \overline{B_1(K_i)}
\subset\Gamma_l\cap  \overline{B_1(K_i)}$
and therefore also $\Gamma_i\cap \overline{B_1(K)}
\subset\Gamma_l\cap \overline{B_1(K)}$
for some $l\in\{1,\ldots,i-1\}$ with $B_1(q)\cap \Gamma_l\cap\Gamma_k
=\emptyset$ for all $q\in\Gamma_l\cap K_l$, $k=1,\ldots,l-1$.
Then apply Lemma \ref{lem:1} to $\Gamma_l$ and again to
the compact set
$K':=\overline{B_1(K)}\cap
\Sigma$, $K'':=\overline{B_2(K)}\cap \Sigma$ 
and $\rho':=\min\{r_{K''},\rho_{K''}\}$ to obtain \eqref{G11} with index
$i$ replaced by $l$.

If, finally $p\in K\setminus\bigcup_{\lambda\in\N} \Gamma_\lambda$ then we find $i\in\N$
and $p'\in\Gamma_i\cap\overline{B_1(K)}\cap\Sigma$ such that
$B_{\rho'/5}(p)\subset B_{\rho'/4}(p')$, and we can
apply  \eqref{G11} to the respective balls centered at $p'$ to find
$
\Sigma\cap B_{\rho'/5}(p)=\Gamma_\lambda\cap B_{\rho'/5}(p)\quad
\textnormal{for $\lambda=i$ or $\lambda=l$.}
$
Thus, we have established \eqref{eq:local-rep}  if we set
$\rho_3:=\rho'.$
Notice that  $\rho'$ merely depends on $K$, not on the index $i$ or $l$. 
\end{proof}

\begin{proof}[Proof of Corollary \ref{cor:c1mfd}.]
	As in Section \ref{subsec:immersion} we can find for any given
	$\beta>0$
	the representation
	$ \Sigma = \bigcup \nolimits_{i=1}^N \li p_i + \graph u_i\ri \cap B_{r_i}(p_i), $
	where $u_i \in C^k\li F_i,F_i^\perp\ri$, 
	$F_i\in \mathscr{G}(n,m)$, $u_i(0)=0$, $Du_i(0)=0$, and $\Vert Du_i \Vert_{C^0}\leq \beta$. In particular, $\lip u_i \leq \beta$ for all $i \in \{1, \dots, N\}$. Since $\beta$ can be 
	chosen arbitrarily, we can locally argue analogously to Corollary 
	\ref{cor:lip_collection} and find for all $p \in \Sigma$ a radius 
	$r=r(\Sigma)>0$ and an $i \in \{1,\dots,N\}$ such that
	$\Sigma \cap B_r(p) = \li p_i + \graph u_i\ri \cap B_r(p).$
	As the $u_i$ are of class $ C^k$, the set
	$\Sigma$ is a $C^k$-submanifold of $\R^n$.
	\end{proof}

\begin{rem}\label{rem:modified-admissible}
If we require in  condition (ii) of Definition \ref{def:admissible_sets}
only that the set $D_p$ is a dense subset of an affine $m$-dimensional
\emph{half space}
$(p+H_*(p))\cap B_{R_K}(p)$, where $H_*(p):=\{x\in H(p)\colon
\langle x,\nu_p\rangle\ge 0\}$ for some vector $\nu_p\in
H(p)\cap\S^{n-1}$ then 
all results of this section also hold true if we add
the following extra condition\footnote{Such a condition was also used
in \cite{kolasinski_2011},\cite[Definition 1.1]{kolasinski-etal_2013a} 
to define the so-called \emph{$m$-fine sets.}}
on the $\theta$- and $\beta$-numbers
defined in \eqref{eq:theta-number} and \eqref{eq:beta-number}: There is
a constant $M_\Sigma$ such that  for all compact subsets $K\subset\Sigma$ 
one has
 $\theta_\Sigma(p,r)\le M_\Sigma\beta_\Sigma(p,r)$
for all $p\in K$ and $r\in (0,R_K]$. We denote this modified class by
$\mathscr{A}^m_*(\alpha,M)$, and briefly
indicate the necessary modifications in the proofs of the
$\beta$-number estimate in Theorem \ref{thm:dist_sigma_to_plane}
and of the Reifenberg flatness in Corollary \ref{cor:theta_bound},
so that then Theorem \ref{thm:lipschitz_graph}  is still applicable.
In order to establish Theorem \ref{thm:dist_sigma_to_plane} the main problem 
is to guarantee the existence of $\eta_x$ and $\eta_y$, respectively, 
satisfying (\ref{proj_etax_1}) and (\ref{projection_etay_p}). For almost 
parallel strands one can choose $e\in H(\mu)$ such that 
$w=\Pi_{H(p)}(e)/\vert \Pi_{H(p)}(e)\vert=\nu_p$ and replace $\mu+\eps R w/2$ 
by $\mu+2\eps R\nu_p$ in (\ref{eta_exists_condition_1}) guaranteeing 
$\Pi_{p+H(p)}(\mu+2\eps R\nu_p) \in (p+H_*(p))\cap B_r(p)$ and hence the 
existence of the desired $\eta_x$. For transversal strands $e^* \in H(\mu)$ 
is fixed. However, we can choose the sign of $v$ such that 
$\langle v,\nu_p\rangle \geq 0$ and consider $\mu+Rv/2+2\eps R \nu_p$ 
instead of $\mu + R v/2$ in (\ref{eta_exists_condition_2}) so that one finds 
$\eta_y$ analogously to the first case. Then the remaining parts of Lemma 
\ref{lem:dist_sigma_to_plane_1}, Lemma \ref{lem:dist_sigma_to_plane_2}, and 
therefore Theorem \ref{thm:dist_sigma_to_plane} can be adopted if we set 
$\eps_0$ sufficiently small. For Corollary \ref{cor:theta_bound}, we first 
find $F_p(r,\delta)\in \mathscr{G}(n,m)$ satisfying 
$\theta_\Sigma(p,F_p(r,\delta),r)\leq M_\Sigma \delta$ by virtue of the 
additional condition defining $\mathscr{A}^m_*(\alpha,M)$ and the already 
proven $\beta$-number estimate. Bounding the angle $\va(F_p(r,\delta),G_p)$ 
by means of Lemma \ref{lem:reifenberg_angles}, one finds a 
constant $C=C(m,M_\Sigma)$ such that $\theta_\Sigma(p,G_p,r)\leq C\delta$ 
for 
all $p \in K $ and all $r\leq r_K$.
To establish the condition of 
$(m,\tilde \delta)$-Reifenberg-flatness on a fixed $K \subset \Sigma$ one 
finally has to choose $\delta=\tilde\delta/C$. Then we obtain the local graph 
representation of Theorem \ref{thm:lipschitz_graph} for all 
$K \subset \Sigma \in \mathscr{A}^m_*(\alpha,M)$ with sufficiently small 
$\alpha$ and $M$ and locally finite energy $E^\tau$. 
\end{rem}

 
\section{Sufficient regularity for finite energy} \label{sec:sufficient_regularity}
First we assume $C^2$-regularity and prove that this implies  finite energy.
\begin{lem}[$C^2$-regularity] \label{lem:C2_sufficiency}
Let $1 \le m\le n$, $F \in \mathscr{G}(n,m)$, $u \in C^2(F,F^\perp)$, $\Sigma:=\graph u$, and $p,q \in \Sigma \cap B_N(0)$ for some $N \in \N$. Then for each $\tau\ge (1/m)-1$ there is a constant $C=C(\tau,m)>0$ such that
\begin{equation}\label{c2-estimate}
	L_\tau \li p,q,T_p\Sigma,T_q\Sigma\ri |p-q|^{-2m} 
	\leq C \Vert D^2 u \Vert_{C^0(F\cap B_N(0))}^{(1+\tau)m} \cdot \vert p-q\vert^{(\tau-1)m}.
	\end{equation}
	In particular, $E^\tau(\Sigma\cap B_N(0))< \infty$ for $\tau >0$.
\end{lem}
\begin{proof}
Recall that $L_\tau(p,q,T_p\Sigma,T_q\Sigma)=\sup_{e\in T_p\Sigma\cap\S^{n-1}}
F_\tau (p,q,e)$, which according to the explicit formula \eqref{eq:new-Ftau} 
and by means of $\tau \ge (1/m)-1$ can
be bounded from above by
	\begin{align} \label{c2_zerlegung}
	& {2^{(1+\tau)m-1}}  \sup_{e \in T_p\Sigma\cap\S^{n-1}}\big[ \vert \Pi_{T_q\Sigma^\perp}(e)\vert^{(1+\tau)m}
	+ \big( {2}{\vert p-q \vert^{-2}}  \vert \Pi_{T_q\Sigma^\perp}(p-q)  \vert   \vert\langle e,p-q\rangle \vert\big)^{(1+\tau)m} \big]\notag\\
	&\leq {2^{(1+\tau)m-1}} \big( \va\li T_p\Sigma,T_q\Sigma\ri^{(1+\tau)m} +\big(   {2}{\vert p-q \vert}^{-1}  \vert \Pi_{T_q\Sigma^\perp}(p-q)  \vert\big)^{(1+\tau)m} \big),
	\end{align}
	where we also used Lemma \ref{lem:angle_formula}.
	By  Lemma \ref{lem:angle_lip} the angle can be estimated as
	\begin{align}\label{c2_angle_bound}
	\va\li T_p\Sigma,T_q\Sigma\ri &\leq \Vert Du(x)-Du(y) \Vert  
	\leq \Vert D^2 u\Vert_{C^0(F\cap B_N(0))} \cdot \vert x-y \vert \notag\\
	&\leq  \Vert D^2 u\Vert_{C^0(F\cap B_N(0))} \cdot \vert p-q \vert,
	\end{align}
	for $x,y \in F\cap B_N(0)$ with $p=x+u(x)$ and $q=y+u(y)$.
	For the projection in the second term in \eqref{c2_zerlegung} we use the notation  $g(\xi):=\xi+u(\xi)$ for $\xi\in F$ to write $p-q=g(x)-g(y)=
	\int_0^1Dg(tx+(1-t)y)(x-y)dt,
	$ so that by virtue of $Dg(y)(x-y)\in T_q\Sigma$ and $Dg(\xi)-Dg(\eta)=
	Du(\xi)-Du(\eta)$ for all $\xi,\eta\in F$ we derive the identity
	\begin{align*}
	  \Pi_{T_q\Sigma^\perp}(p-q) 
	&= 
	\textstyle\int \nolimits_0^1 \Pi_{T_q\Sigma^\perp} \li \li Du(tx+(1-t)y) - Du(y) \ri \li x-y\ri \ri \ dt \notag\\
	&= \textstyle\int \nolimits_0^1 \int \nolimits_0^1 \Pi_{T_q\Sigma^\perp} \big( \frac d{d\sigma} Du\li \sigma \li tx+(1-t)y\ri + \li 1-\sigma\ri y\ri_{\mid \sigma=s}(x-y)\big) \ ds\ dt,
	\end{align*}
	which yields the estimate $
	|\Pi_{T_q\Sigma^\perp}(p-q)|\leq \Vert D^2 u\Vert_{C^0(F\cap B_N(0))} \cdot \vert q-p\vert^2.$
	 Combining this with \eqref{c2_zerlegung} and \eqref{c2_angle_bound} leads to the proof of \eqref{c2-estimate}.
	Finally, integrating \eqref{c2-estimate} for $\tau >0$ over $\Sigma\cap B_N(0)
	\times \Sigma\cap B_N(0)$ one obtains $E^\tau(\Sigma\cap B_N(0))<\infty$.
	\end{proof}
With a  covering argument (as in the proof of Theorem \ref{thm:sufficient-sobolev} below) one can show the following corollary.
\begin{cor}\label{cor:embedded-c2}
Embedded locally compact $C^2$-submanifolds have locally finite energy $E^\tau$ for all 
$\tau >0$. 
\end{cor}

A closer look on the proof of Lemma \ref{lem:C2_sufficiency} reveals
that a substantially lower regularity already guarantees finite 
M\"obius energy. 
\begin{lem}[Fractional Sobolev regularity] \label{lem:sobolev_sufficiency}
	Let  $F\in \mathscr{G}(n,m)$, $1 \le m\le n$,
	$\tau>0$, $u\in C^{0,1} 
	\cap W^{1+\frac 1{1+\tau},(1+\tau)m}(F,F^\perp)$, $\Sigma_u:=
	\graph u$, and fix a point $z=\zeta+u(\zeta)\in\Sigma_u$.
	Then for all $r>0$ there is a constant $C=C(\tau,m,\lip u)>0$ 
	such that
	$E^\tau(\Sigma_u\cap B_r(z))\le
	C  [Du ]^{(1+\tau)m}_{{(1+\tau)^{-1}},\,(1+\tau)m}.
	$ 
\end{lem}

\begin{proof}
	Applying the Area Formula 
	\cite[Theorem 3.9 \& Section 3.3.4B]{evans-gariepy_2015} we find 
	a constant $C>0$ depending on $m$ and on $\lip u$
	such that 
	\begin{align*} E^\tau(\Sigma_u\!\cap\! B_r(z))
&\leq \textstyle C\!  \int \nolimits_{F \cap B_r(\zeta)} \int \nolimits_{F\cap B_r(\zeta)} 
\!\!\!\frac{L^\tau\li g(x),g(y),T_{g(x)} \Sigma_u, T_{g(y)}\Sigma_u \ri}{ 
|g(x)-g(y)|^{2m}}
 d\HM^m(x)  d\HM^m(y),
\end{align*}
where we used the notation $g(\xi):=\xi+u(\xi)$ for $\xi\in F$ as
in the proof of Lemma \ref{lem:C2_sufficiency}. As in that proof
we can 
	adjust the constant $C$ (now depending also on $\tau$)
	to bound the numerator of the integrand 
	from above by
	\begin{align*}
 C \cdot\big[ \va \li T_{g(x)}\Sigma_u, T_{g(y)}\Sigma_u\ri^{(1+\tau)m}
	+ \big(  2 {\vert g(x)-g(y)\vert}^{-1}  
	\vert \Pi_{T_{g(y)}\Sigma_u^\perp}(g(x)-g(y))\vert 
	\big)^{(1+\tau)m}\big].
	\end{align*}
For the second term one estimates $
 \vert \Pi_{T_{g(y)}\Sigma_u^\perp}(g(x)-g(y))\vert 
	\leq \int \nolimits_0^1  \vert Du(tx+(1-t)y)-Du(y) \vert 
	\cdot \vert x-y\vert \ dt$,
	which together with Lemma \ref{lem:angle_lip} and 
	$\vert x-y\vert\leq \vert g(x)-g(y)\vert$ implies
	\begin{align*}
E^\tau(\Sigma_u\cap B_r(z))  &\leq 
\textstyle C \int \nolimits_{F \cap B_r(\zeta)} \int \nolimits_{F\cap B_r(\zeta)} \frac { \left \Vert Du(x)-Du(y)\right \Vert^{(1+\tau)m} }{\vert x-y\vert^{2m}} \ d\HM^m(x) \ d\HM^m(y) \\
	& \hspace{-2.5cm}+ \textstyle 
	C \int \nolimits_0^1 \int \nolimits_{F \cap B_r(\zeta)} \int \nolimits_{F\cap B_r(\zeta)} \frac { \left \Vert Du(tx+(1-t)y)-Du(y)\right \Vert^{(1+\tau)m} }{\vert x-y\vert^{2m}} \ d\HM^m(x) \ d\HM^m(y)\ dt.
	\end{align*}
	Changing variables to $\tilde x=tx+(1-t)y$ yields $\vert x-y\vert
	= {\vert \tilde x -y\vert}/{t}$, so that
	\begin{align} \label{sobolev_single_graph}
	E^\tau(\Sigma_u\cap B_r(z))
\leq C\li 1+ 1/{(m+1)}\ri  
[Du ]^{(1+\tau)m}_{{{(1+\tau)^{-1}},(1+\tau)m}}.
	\end{align}
	\end{proof}

	\begin{proof}[Proof of Theorem \ref{thm:sufficient-sobolev}.]
Consider an embedded
	compact submanifold of
	        $\R^n$ with local graph representations of class
		$C^{0,1}\cap  W^{1+{(1+\tau)^{-1}},(1+\tau)m}$.
 Then, we can find $z_1,\dots,z_N\in\Sigma$ and radii 
	$r_i>0$, such that
$\Sigma \subset \bigcup_{i=1}^N \Sigma \cap B_{ {r_i}/2}(z_i)$ and 
$\Sigma \cap B_{r_i}(z_i)=\li z_i+\graph u_i \ri \cap B_{r_i}(z_i)$
	with $u_i \in C^{0,1}\cap  
	W^{1+{(1+\tau)^{-1}},(1+\tau)m}(F_i,F_i^\perp)$ where 
	$F_i\in \mathscr{G}(n,m)$ for $i=1,\dots,  N.$ 
	Set $r:=\min\{ {r_1}/2,\dots ,  {r_N}/2 \}$ to obtain 
	$\Sigma \cap B_r(q)\subset \Sigma \cap B_{r_i}(z_i)$ for every
	$q\in\Sigma \cap B_{ {r_i}/2}(z_i)$, which 
	together with \eqref{sobolev_single_graph}  implies
	\begin{align*}
\textstyle\int \nolimits_\Sigma& \textstyle\int \nolimits_{\Sigma \cap B_r(q)} 
	L_\tau\li p,q,T_p \Sigma,T_q\Sigma\ri |p-q|^{-2m} 
	\ d\HM^m(p)\ d\HM^m(q) \\
	&\textstyle\leq \sum \nolimits_{i=1}^N 	\int \nolimits_{\Sigma 
	\cap B_{ {r_i}/2}(z_i)} \int 
	\nolimits_{\Sigma \cap B_{r_i}(z_i)} 
 L_\tau\li p,q,T_p \Sigma,T_q\Sigma\ri |p-q|^{-2m} 
	 \ d\HM^m(p)\ d\HM^m(q) \\
&\leq \textstyle\sum\nolimits_{i=1}^N E^\tau\big(\Sigma_{u_i}\cap B_{r_i}(z_i)\big)
{\le}  \sum \nolimits_{i=1}^N C(\tau,m,\lip u_i)
	[Du_i ]^{(1+\tau)m}_{{{(1+\tau)^{-1}},(1+\tau)m}}.
	\end{align*}
	On the other hand, the integrand is bounded
	from above  by $r^{-2m}$ for all $q\in\Sigma$
	and $p\in\Sigma\setminus B_r(q)$, so that we finally arrive at
\[\textstyle E^\tau(\Sigma) \le \sum \nolimits_{i=1}^N  C(\tau,m,\lip u_i)
 [Du_i ]^{(1+\tau)m}_{{{(1+\tau)^{-1}},(1+\tau)m}} +  
 \HM^m(\Sigma)^2 {r^{-2m}} 
	 < \infty. \]
\end{proof}


\appendix

\section{Angles} \label{app:angles}
We first recall in Lemma \ref{lem:angle_formula} some general identities for the angle metric \eqref{eq:angle-metric} on the 
Grassmannian $\mathscr{G}(n,m)$, 
demonstrate a simple inclusion for two cones around different $m$-planes and with different opening angles (Lemma \ref{lem:cone}), and
estimate  in Lemma \ref{lem:reifenberg_angles} the angle between two  planes approximating a set in two different ways  on  different scales. 
Then we introduce principal angles and relate these to the angle metric (see Definition \ref{def:princ_angles}--Lemma \ref{lem:angle_rel}), which
finally allows us to compare in Corollary \ref{cor:energy_comparison} our M\"obius invariant energies 
$E^\tau$ to the Kusner-Sullivan energy $E_\textnormal{KS}$. 

	\begin{lem} [8.9 (3) in \cite{allard_1972}] \label{lem:angle_formula}
		Let $F,G \in \mathscr{G}(n,m)$, then
		\[ \Vert \Pi_{F}-\Pi_{G} \Vert = \Vert \Pi_{F^\perp} - \Pi_{G^\perp} \Vert
		=  \Vert \Pi_{F^\perp} \circ \Pi_{G} \Vert = \Vert \Pi_{F} \circ \Pi_{G^\perp} \Vert= \Vert \Pi_{G^\perp}\circ \Pi_{F}\Vert 
		=  \Vert \Pi_{G} \circ \Pi_{F^\perp} \Vert.\]
	\end{lem}
Recall from \eqref{eq:cone} our notation $C_x(\beta,F)$ for a cone around $F\in\mathscr{G}(n,m)$, centered at $x$ with opening angle
$2\arctan\beta$.
\begin{lem}[Cone Lemma] \label{lem:cone}
		Let $p \in \R^n$, $F,G \in \mathscr{G}(n,m)$ with $\va(F,G) \leq \chi$ for some $\chi \geq 0$, and assume that $\sigma,\kappa\geq 0$ satisfy
		\begin{align} \label{cone_lem_condition}
		0 \leq \frac {\sigma + (1+\sigma)\chi}{1-(1+\sigma)\chi} \leq \kappa.
		\end{align}
		Then we have $
		 C_p(\sigma,F)\subset C_p(\kappa,G).$
	\end{lem}
	\begin{proof}
		Let $z \in C_p(\sigma,F)$ and estimate
		\begin{align*}
		 \vert \Pi_{G^\perp} (z-p)  \vert 
	&
	\leq  \vert \Pi_{F^\perp}(z-p) \vert + \chi \vert z-p\vert
	\\
		&
	\leq \sigma  \vert \Pi_F(z-p) \vert + 
	\chi  \vert \Pi_G(z-p) \vert + 
	\chi  \vert \Pi_{G^\perp}(z-p) \vert
	\\
		&
	\leq \sigma  \vert \Pi_G(z-p) \vert + (\sigma+1)\chi 
	\big( \vert \Pi_G(z-p) \vert + 
	\vert \Pi_{G^\perp}(z-p) \vert\big),
		\end{align*}
		which implies after absorbing the last summand 
	$ \big(1- \chi(1+\sigma)\big)  \vert \Pi_{G^\perp}(z-p)  \vert 
	\leq \big(\sigma + (1+\sigma)\chi\big)  \vert \Pi_G(z-p) \vert,$
		proving the claim by means of assumption \eqref{cone_lem_condition}.
	\end{proof}
For the following estimate between two approximating planes recall 
Definition \eqref{eq:theta-plane} of the theta-number with
respect to a fixed plane.
\begin{lem} \label{lem:reifenberg_angles}
		Let $p_1,p_2 \in \Sigma \subset \R^n$ and $0< r_1 \leq r_2$,
		\ $d_1,d_2 \in (0,1/2)$, and $F_1,F_2 \in \mathscr{G}(n,m)$ 
		such that $\vert p_1- p_2\vert < r_1/2$,
		and
	\begin{align}
	 \dist\big(\xi,\Sigma\cap B_{r_1}(p_1)\big)&\le d_1r_1\Foa
	  \xi\in (p_1+F_1)\cap B_{r_1}(p_1) \label{eq:theta_1}, \\
	 \beta_\Sigma(p_2,F_2,r_2)&\leq d_2.\label{eq:dist_2}
		\end{align}
				Then there is a constant $\tilde C=\tilde C(m)>\sqrt 2$
		with $
		\va(F_1,F_2) \leq 2\tilde C     ( d_1 + 
		2 d_2 {r_2}/{r_1}  )/(1-2d_1). $
	\end{lem} 
	\begin{proof}
	For an orthonormal basis $\{e_1,\dots,e_m\}$ of $F_1$ and an arbitrary 
	$\varepsilon \in (0, (1-2d_1)/2)$ we set $x_0:=p_1$ and
		\begin{align} \label{angle_x_i}
		x_i := p_1 +  {(1-2d_1-2\varepsilon)}  r_1 \cdot e_i/2 \in \li p_1+F_1 \ri \cap B_{r_1}(p_1) \Fo i=1, \dots,m.\end{align}
		By 
		means of \eqref{eq:theta_1}
		we can find $q_i \in \Sigma \cap B_{r_1}(p_1)$ such that 
		\begin{align}\label{angle_lemma_dist_1} \vert q_i-x_i \vert \leq (d_1+\varepsilon)r_1 \Foa i=1,\dots,m. \end{align}
		Additionally, we define $q_0:=p_1$ so that
		$ \vert q_i-p_1 \vert \leq \vert q_i-x_i \vert + 
		\vert x_i - p_1 \vert \leq  (d_1 + \varepsilon) r_1 + 
		 {(1-2d_1-2\varepsilon)}  r_1/2 =  {r_1}/ 2 \Foa i=1,\dots ,m.$ 
		Consequently, $\vert q_i - p_2 \vert \leq \vert q_i - p_1 \vert + \vert p_1 - p_2 \vert < r_1 \leq r_2$ for all $i \in \{1, \dots , m\}$ and $\vert q_0 - p_2 \vert = \vert p_1 - p_2 \vert < \frac {r_1} 2 < r_2$.
		Hence, $q_i \in \Sigma \cap B_{r_2}(p_2)$ for all $i \in \{0,\dots,m\}$.
		Now we can use \eqref{eq:dist_2} to obtain 
		points $y_i \in (p_2+F_2) \cap B_{r_2}(p_2)$ with
		\begin{align}\label{angle_lemma_dist_2} \vert y_i - q_i\vert \leq d_2 r_2 \Foa i=0,\dots, m. \end{align}
		For $i \in \{1,\dots,m\}$ we can define
		$\tilde x_i := \frac{x_i-x_0}{\vert x_i -x_0\vert } = e_i \in F_1 \AND \tilde y_i:= \frac {y_i-y_0}{\vert x_i-x_0\vert } \in F_2.$
		Then \eqref{angle_x_i}--\eqref{angle_lemma_dist_2} imply
		for all $i=1,\ldots,m$
		\begin{align*}
		\vert \tilde x_i - \tilde y_i\vert & = \textstyle
		\frac{ 2}{(1- 2 d_1-2\varepsilon)r_1} 
 \left\vert x_i-x_0-y_i+y_0 \right \vert 
		\leq 
		 \frac 2 {1- 2 d_1-2\varepsilon} \big( 
		 d_1 + \varepsilon + 2 {r_2}d_2/{r_1}  \big). 
		\end{align*}
		Consequently, if
		\begin{align} \label{angle_lemma_smallnes}
		 2  \left( d_1 + 2  {r_2}d_2 /r_1\right)/(1-2d_1) < 1/ {\sqrt2}, \end{align}
		then we can choose $0<\eps\ll1$ such that $\dist\li e_i,F_2\ri \leq \vert \tilde x_i - \tilde y_i\vert <1/\sqrt 2$ for all $i=1,\dots,m$. 
                Applying \cite[Prop. 2.5]{kolasinski-etal_2013a} then yields a constant $\tilde C=\tilde C(m)>\sqrt 2$ with
	$ \va(F_1,F_2) \leq 2\tilde C
	\cdot  \left( d_1 + \varepsilon + 2 {r_2}d_2/{r_1} \right)/(1-2d_1-2\varepsilon).$
		We conclude this case by taking the limit $\eps \to 0$.
		If \eqref{angle_lemma_smallnes} does not hold, then 
		the desired inequality for $
		  \va(F_1,F_2)$ trivially holds true for the same constant $\tilde C$.
	\end{proof}
	Now we recall the definition of principal angles.  
\begin{defin} [Ch. 12.4.3 in \cite{golub-loan_1996}]\label{def:princ_angles}
		For two $m$-planes $F,G \in \mathscr{G}(n,m)$, the \emph{principal angles} $\vartheta_1,\dots, \vartheta_m\in [0,\pi/2] $ are given by $
			\cos\vartheta_1:= \sup_{ x\in F \cap \S^{n-1}} \sup_{ y \in G \cap \S^{n-1} }\vert \langle x,y \rangle \vert $
			and
\[  \cos\vartheta_k:= \sup_{x\in F \cap \S^{n-1}\atop 
x \perp \operatorname{span}(x_1,\dots,x_{k-1})}
 \sup_{ y \in G \cap \S^{n-1}\atop y \perp 
 \operatorname{span}(y_1,\dots,y_{k-1})}
			\vert\langle x,y \rangle\vert \quad\Fo k \in \{2,\dots,m\}.\]
			Here, for each $k\in\{1,\ldots,m\}$ the \emph{principal vectors }
			$x_k\in \li F \setminus \operatorname{span}(x_1,\dots,x_{k-1})\ri  \cap \S^{n-1}$ and $y_k \in \li G \setminus \operatorname{span}(y_1,\dots,y_{k-1})\ri \cap \S^{n-1}$ are chosen to satisfy $\langle x_k,y_k\rangle = \cos\vartheta_k.$
				\end{defin}
It is mentioned in \cite[Ch. 12.4.3]{golub-loan_1996} that one can find principal vectors such that
		\begin{align} \label{princ_vectors}
			\langle x_i,y_j\rangle = \begin{cases}
				\cos\vartheta_i &\If i=j,\\
				0 &\If i\neq j,
			\end{cases}
		\end{align}
		by virtue of the singular value decomposition of a matrix. Moreover, each set of principal vectors forms an orthonormal basis of the corresponding plane.
		\begin{lem}[Angle metric vs. principal angle] \label{lem:princ_angles}
		For $F,G \in \mathscr{G}(n,m)$ one has $
			 \sin\vartheta_m = \va (F,G),  $
		where $\vartheta_m \in [0, \pi/ 2]$ denotes the largest 
		principal angle for $F$ and $G$.
	\end{lem}
	\begin{proof}
	For a collection of principal vectors of $F$ and $G$ satisfying \eqref{princ_vectors}, we define
	$ X_1:=\begin{pmatrix} x_1 \vert \dots \vert x_m \end{pmatrix}\in \R^{n\times m}
	$ and $ Y_1:=\begin{pmatrix} y_1 \vert \dots \vert y_m \end{pmatrix}\in \R^{n\times m}.$
	With an orthonormal basis $(x_{m+1},\dots,x_n)$ of $F^\perp$ and $(y_{m+1},\dots,y_n)$ of $G^\perp$, we set $
	X_2:=\begin{pmatrix} x_{m+1} \vert \dots \vert x_n \end{pmatrix}\in \R^{n\times (n-m)}$ and $ Y_2:=\begin{pmatrix} y_{m+1} \vert \dots \vert y_n \end{pmatrix}\in \R^{n\times (n-m)}.  $
	Then, $X:=\begin{pmatrix} X_1\vert X_2 \end{pmatrix}$ and 
	$Y:=\begin{pmatrix} Y_1\vert Y_2\end{pmatrix}$ are orthogonal matrices,
	and so is
	 $X^T\cdot Y$. In particular, for $z=(\zeta, 0) \in (\R^m \times \{0\}^{n-m})\cap \S^{n-1}$, one finds
	$1= | X^T \cdot Y \cdot z|^2 = | X_1^T\cdot Y_1\cdot \zeta |^2 + 
	| X_2^T\cdot Y_1\cdot \zeta |^2$.
	Therefore,
	\begin{align}\label{princ_lem_identity}
	\textstyle\left \Vert X_2^T\cdot Y_1\right \Vert =	\sup_{\zeta 
	\in \R^m\cap \S^{m-1}}\left \vert X_2^T\cdot Y_1\cdot \zeta 
	\right \vert = \sqrt{1-\inf_{\zeta\in\R^m\cap \S^{m-1}}
	\left \vert X_1^T\cdot Y_1\cdot \zeta \right \vert^2}.
	\end{align}
	By choice of $x_1,\dots,x_m$ and $y_1,\dots,y_m$ one finds $X_1^T\cdot Y_1=\operatorname{diag}(\cos\vartheta_1,\dots,\cos\vartheta_m)$. Consequently, the right hand side of \eqref{princ_lem_identity} is equal to $\sqrt{1-\cos^2\vartheta_m}=\sin\vartheta_m$.
	On the other hand, we have $\Pi_{F^\perp}(z)= X_2\cdot X_2^T\cdot z$ and $\Pi_G(z)=Y_1\cdot Y_1^T\cdot z$ for all $z\in \R^n$. Together with Lemma \ref{lem:angle_formula} and the orthogonality of $X$ and $Y$, one finds

$		\va(F,G)
= \Vert \Pi_{F^\perp} \circ \Pi_G \Vert = \Vert X_2\cdot X_2^T 
\cdot Y_1\cdot Y_1^T \Vert
	= \Vert X^T \cdot X_2\cdot X_2^T \cdot Y_1\cdot Y_1^T \cdot Y\Vert
=\Vert \begin{pmatrix} 0 \\ X_2^T\end{pmatrix} 
\cdot \begin{pmatrix} Y_1 \ \vert\ 0 \end{pmatrix}\Vert =  
\Vert X_2^T\cdot Y_1  \Vert,
	$
	which coincides with the left hand side of \eqref{princ_lem_identity}.
	\end{proof}
		\begin{lem}[Combined angle] \label{lem:angle_rel}
		For $F,G \in \mathscr{G}(n,m)$ and $\cos\vartheta:=\Pi_{i=1}^m \cos\vartheta_i$, where $\vartheta_1,\dots,\vartheta_m$ are the principal angles for $F$ and $G$, one has 
	$\sin\vartheta_m\leq \sin\vartheta \leq \sqrt m \sin\vartheta_m $,
		and
		\begin{align}
 0&\leq \left(1-\cos\vartheta\right)^m \leq (\sin\vartheta)^{(1+\tau)m}
			 \quad\forall \tau \in[0,1),\label{1}\\
 0 &\leq 2^{-m}(\sin\vartheta)^{2m}  \leq \left(1-\cos\vartheta\right)^m \leq (\sin\vartheta)^{2m},\label{2}\\
0 &\leq c\cdot
 (\sin\vartheta)^{(1+\tau)m}\leq \left(1-\cos\vartheta\right)^m \quad\forall \tau \in(1,\infty),\label{3}
		\end{align}
		where
$c=c(\tau,m):=\min\big\{1, \big[\left( 1-\frac 1 \tau\right) 
\left(1-\frac 1 {\tau^2}\right)^{-(1+\tau)/2}\big]^m\big\}>0.$ 
	\end{lem}
	\begin{proof}
The ordering $0\leq \vartheta_1\leq \dots  \leq\vartheta_m\leq  \pi/ 2$ implies
$
			(\cos\vartheta_m)^m\leq \cos\vartheta\leq \cos\vartheta_m,$
			so that
		\begin{align}\label{angle_rel_sinm_1}
			\sin\vartheta_m=\sqrt{1-\cos^2\vartheta_m}\leq 
			\sqrt{1-\cos^2\vartheta} =\sin\vartheta.
		\end{align}
		On the other hand, by the formula for the
		geometric series, one has
	$
	{1-(\cos\vartheta_m)^{2m}}={(1-\cos^2\vartheta_m)}\cdot
	\sum_{i=0}^{m-1} (\cos\vartheta_m)^{2i}\leq  {(1-\cos^2\vartheta_m)}\cdot
	m.
	$
		Hence, we obtain
		\begin{align}\label{angle_rel_sinm_2}
\sin\vartheta=\sqrt{1-\cos^2\vartheta}\leq \sqrt{1-(\cos\vartheta_m)^{2m}}
\leq \sqrt m \cdot \sqrt{1-\cos^2\vartheta_m}=\sqrt m \sin\vartheta_m.
		\end{align}
		Combining \eqref{angle_rel_sinm_1} and \eqref{angle_rel_sinm_2} yields
	$\sin\vartheta_m\leq \sin\vartheta\leq \sqrt m \sin\vartheta_m.$
For the remaining inequalities, we define $h_\tau(\vartheta)
:={(1-\cos\vartheta)}\cdot
{(\sin\vartheta)^{-(1+\tau)}}$ and compute
		\begin{align} \label{angle_rel_limit}
			\lim_{\vartheta \to  \pi/ 2} h_\tau(\vartheta)=1 
			\quad\AND\quad 
			\lim_{\vartheta \to 0} h_\tau(\vartheta)=\begin{cases}
				0 &\Fo \tau \in [0,1),\\
				 1/ 2 &\Fo \tau=1,\\
				\infty &\Fo \tau \in(1,\infty).
			\end{cases}
		\end{align}
For $\vartheta \in (0, \pi/ 2)$ we can differentiate $h_\tau$ to obtain
$		
h_\tau'(\vartheta)=
 (\sin\vartheta)^{-\tau}\cdot[1-(1+\tau) {\cos\vartheta}\cdot{(1+\cos\vartheta)^{-1}}
].$
In the case $\tau \leq 1$, one finds $(1+\tau){\cos\vartheta}/{(1+\cos\vartheta)}
\leq  {(1+\tau)}/{2}\leq 1$ guaranteeing that $h_\tau$ is non-decreasing on 
$(0, \pi/ 2)$. Together with \eqref{angle_rel_limit} this implies
\eqref{1} and \eqref{2}.
		In the case $\tau>1$, one finds that $
h_\tau'(\vartheta) = 0$ if and only if $ \cos\vartheta= 1/\tau,$
		and therefore \eqref{3} holds true. 
	\end{proof}
Since the numerator \eqref{eq:kusner-sullivan-integrand}   
in the energy density of the Kusner-Sullivan 
energy $E_\textnormal{KS}$ uses the combined angle defined in Lemma
\ref{lem:angle_rel} it is easy to conclude from the inequalities
in that lemma the following
comparison result for $E^\tau$ and $E_\textnormal{KS}$.
\begin{cor} \label{cor:energy_comparison}
	For $\Sigma \in \Aam$ one finds 
	\begin{align}
0&\leq E_\textnormal{KS}(\Sigma) \leq \sqrt{m}^{{(1+\tau)m} } 
E^{\tau}(\Sigma)\Foa
\tau \in [0,1),\label{tauless1}\\
0 &\leq 2^{-m} E^\tau(\Sigma)\leq 
		E_\textnormal{KS}(\Sigma) \leq m^m E^{\tau}(\Sigma) \Fo \tau =1,
		\label{tauequal1}\\
0& \leq  c
E^{\tau}(\Sigma)\leq E_\textnormal{KS}(\Sigma) \Foa \tau \in (1,\infty),
\label{taugreater1}
\end{align}
where $c=c(\tau,m)>0$ is the constant defined in Lemma \ref{lem:angle_rel}.
\end{cor}
\begin{rem}\label{rem:wedge}
We observe that a wedge-shaped singularity leads to infinite
$E^\tau$-energy for any $\tau>-1$. 
Indeed, consider for simplicity the set
$\Sigma_\beta=\{x e_1+\beta \vert x\vert e_2 
+ y e_3\colon x,y\in\R \} \subset \R^3$ for $\beta\in (0,\infty)$ 
and an orthonormal 
basis $\{e_1,e_2,e_3\}$, and set $H(p)$ to be tangential to the set for all 
$p=p(x,y)\in \Sigma_\beta$ with $x\neq 0$. For $i\in \N$ let 
$p_i:=(e_1+\beta e_2)/2^i$ and $q_i=(-e_1+\beta e_2+e_3)/2^i$. 
Then we can compute $\vert p_i-q_i\vert = \sqrt 5/2^i$, 
$\vert \langle e_3,p_i-q_i\rangle \vert =2^{-i}$ and 
$\vert \Pi_{H(p_i)^\perp}(p_i-q_i)\vert =2\beta/(2^i\sqrt{1+\beta^2})$. 
Therefore, we find for $\kappa(\beta)
:=\beta/(4\sqrt{1+\beta^2})$ and 
$\eps_i:=\kappa/2^i$ a constant $C=C(\beta,m,\tau)>0$ such that 
$F_\tau(\mu,\eta,e_3)/\vert \mu-\eta\vert^{2m}\geq 2^{2mi} C $ 
for all $\eta \in B_{\eps_i}(p_i)\cap \Sigma_\beta$ and 
$\mu \in B_{\eps_i}(q_i)\cap \Sigma_\beta$, where we used 
that $e_3\in H(\eta)=H(p_i)$ for 
all such $\eta$. Moreover, for $\eta_1 \in B_{\eps_{i_1}}(p_{i_1})$ and 
$\eta_2 \in B_{\eps_{i_2}}(p_{i_2})$ with $i_1<i_2$, one finds 
$\vert \eta_1-\eta_2\vert >(1-\kappa)/2^{i_1}-(1+\kappa)/2^{i_2}\geq 
(1-3\kappa)/2^{i_1+1}>0$ since $\kappa < 1/4$. Hence, the $\eps_i$ 
neighbourhoods of the $p_i$ are disjoint. Analogously, the same holds true 
for $q_i$. Finally, for any $N\in\N$ there is an index $i_0\in\N$ such
that
\begin{align*}
 E^\tau(\Sigma_\beta\cap B_N(0))&\geq 
 \textstyle
 \sum \nolimits_{i=i_0}^\infty \int_{\Sigma_\beta\cap B_{\eps_i}(q_i)}  \int_{\Sigma_\beta\cap B_{\eps_i}(p_i)} \frac {F_\tau(\mu,\eta,e_3)}{\vert \mu-\eta\vert^{2m}} d\HM^m(\eta) \ d\HM^m(\mu)\\
 &\geq \textstyle\sum \nolimits_{i=i_0}^\infty \omega_m^2 \cdot (\kappa/2^i)^{2m} 
 \cdot 2^{2mi} C =\infty.
\end{align*}
\end{rem}

\section{Lipschitz graphs}\label{app:graphs}
First, in Lemma \ref{lem:angle_lip},  we estimate the deviation of a 
Lipschitz graph's tangent plane from its domain plane. Then 
in Lemma \ref{lem:shift-graphs} we shift Lipschitz
functions without changing the trace of the graph, and
in the Tilting Lemma \ref{lem:tilting}, we provide a lower bound on  the  size of the projection of a Lipschitz graph onto
a plane slightly tilted from its domain plane. 
Finally, we prove that the intersection
of two Lipschitz graphs that are sufficiently flat in comparison to the angle between their domain planes is contained
in a lower-dimensional graph; see Lemma \ref{lem:intersect_lip_graphs}. This quantitative result  generalizes the well-known fact that the transversal intersection
of two $C^1$-submanifolds constitutes a lower-dimensional $C^1$-submanifold as, e.g., proven in 
\cite[p. 30]{guillemin-pollack_1974}.
	\begin{lem}  \label{lem:angle_lip}
		Let $\beta\in [0,1) $, $F\in \mathscr{G}(n,m)$, and assume $u \in C^{0,1}(F,F^\perp)$ satisfies $\lip u \leq \beta$. For $x,y \in F$, with $p=x+u(x)$ and $q =y+u(y)$, such that $Du(x)$ and $Du(y)$ exist, one finds
$ \va \li T_p\li\graph u\ri,F \ri \leq \Vert Du(x)\Vert \leq \beta, $ 
		and 
		\[ \va \li T_p\li \graph u\ri,T_q\li\graph u\ri \ri \leq \Vert Du(x)-Du(y)\Vert \leq \textstyle\sqrt \frac {1+\beta^2}{1-\beta^2}  \va \li T_p \graph u,T_q\graph u\ri .  \]
	\end{lem}	
	\begin{proof} 
          This is an immediate implication of 8.9 (5) in \cite{allard_1972}.
        \end{proof}

\begin{lem}[Shifting Lemma]\label{lem:shift-graphs}
For any $x\in p+\graph u$, where $u\in C^{0,1}(F,F^\perp), $
$F\in\mathscr{G}(n,m)$, there exists a function $\tilde{u}\in
C^{0,1}(F,F^\perp)$ with $\lip\tilde{u}=\lip u$ and $\tilde{u}(0)=0$, such that
$
x+\graph\tilde{u}=p+\graph u.
$
\end{lem}
\begin{proof}
For $x=p+\xi +u(\xi)$, $\xi\in F$, set $\tilde{u}(y):=u(y+\xi)-u(\xi)$. Then
the claim follows from the identity 
$q=p+\eta+u(\eta)=x+\eta-\xi+u(\eta)-u(\xi)=x+\eta-\xi+\tilde{u}(\eta-\xi)$
for an arbitrary point
$q\in p+\graph u$.
\end{proof}

	\begin{lem}[Tilting Lemma] \label{lem:tilting}
		Let $F,G\in \mathscr{G}(n,m)$ satisfy $\va(F,G)\leq \chi$ for some $\chi \in [0,1)$, and suppose $u\in C^{0,1}(F,F^\perp)$ with $u(0)=0$ and whose Lipschitz constant $\lip u $ satisfies
$		
		\sigma:= \chi(1+\lip u) <1.
$	
		Then we have
		\[ B_{\frac{(1-\sigma)\rho}{\sqrt{1+ (\lip u)^2}}}(0) \cap G \subset \Pi_G \li \graph u \cap {B_\rho(0)} \ri \Foa \rho>0.\]
	\end{lem}
	\begin{proof}
		For $p \in \graph u$ set $x:=\Pi_F(p)$, $z:=\Pi_G(p)$ and estimate
		\begin{align} \label{tilting_dist}
		\vert x - z\vert&= \left \vert \li \Pi_F -\Pi_G\ri (p)\right \vert \leq \chi \vert p \vert =\chi \vert x+ u(x) \vert \leq \chi \li 1 + \lip u\ri \vert x \vert = \sigma \vert x  \vert,
		\end{align}
		since $u(0)=0$. Define $\phi_1 \colon F \to \graph u, \ x \mapsto x + u(x)$ and $ \phi_2\colon \graph u \to G$ by $p \mapsto \Pi_G(p)$, and look at the composition $\phi:=\phi_2 \circ \phi_1\colon F \to G$, then \eqref{tilting_dist} implies
		\begin{align}\label{tilting_phi}
		\vert x -\phi(x)\vert \leq \sigma \vert x \vert \quad\Foa x \in F.
		\end{align} 
	Taking the linear isometry $I_F \colon F \to \R^m$ and defining
		$\Psi \in C^0(\R^m,\R^m)$ by $\Psi:=I_F \circ \Pi_{F\mid_G} \circ \phi \circ I_F^{-1}\colon \R^m\to \R^m$ we infer from \eqref{tilting_phi} the inequality
		\begin{align}\label{tilting_psi}
		\vert \xi - \Psi(\xi)\vert &= \vert I_F(x) - \Psi(I_F(x))\vert 
		=\vert I_F(x) - I_F \circ \Pi_{F \mid_G} \circ \phi (x)\vert 
		\notag \\
		&= \vert x - \Pi_{F \mid_G} \circ \phi (x)\vert
		=  \vert \Pi_F(x-\phi(x)) \vert \leq \vert x - \phi(x)\vert 
		\notag\\
		& \leq \sigma \vert x \vert = \sigma  \vert I_F^{-1}(\xi)\vert
		 = \sigma \vert \xi\vert \quad\Foa \xi=I_F(x) \in \R^m. 
		\end{align}
		Applying \cite[Prop. 2.5]{kolasinski-etal_2018} to $F:=\Psi$, we find for any given $\rho>0$ that for all $\eta \in B_{(1-\sigma)\rho}(0)\subset \R^m$, there exists a $\xi \in \overline{B_\rho(0)}$, such that $\Psi(\xi)=\eta$. Moreover, for each $y \in F\cap B_{(1-\sigma)\rho}(0)$, there is a unique $\eta \in B_{(1-\sigma)\rho}(0)\subset \R^m$, such that $y =I_F^{-1}(\eta)$, so that with $x:=I_F^{-1}(\xi) \in F \cap \overline{B_\rho(0)}$ for $\xi$ as above one finds
		$ \Pi_{F\mid_G} \circ \phi(x)= I_F^{-1}\circ \Psi(\xi)= I_F^{-1}(\eta)=y, $
		which implies that $\phi\colon F\to G$ is surjective since $\Pi_{F\mid_G}$ is bijective, 
                due to $\va(F,G)\leq \chi <1$, see \cite[Lem. 2.2]{kolasinski-etal_2018}. Notice that $\phi(0)=\phi_2\circ \phi_1(0)=\phi_2(0+u(0))=\Pi_G(0)=0$, and \eqref{tilting_phi} implies that $
		 (1+\sigma)\vert x \vert \geq \vert \phi(x)\vert \geq (1-\sigma)\vert x \vert \Foa x \in F,$
		so that
		\begin{align*}
		G\! \cap\! B_r(0) &\subset \phi\big( 
		B_{\frac r {1-\sigma}}(0)\cap F \big) 
		=\Pi_G\! \circ\! \phi_1 \big( 
		B_{\frac r {1-\sigma}}(0)\cap F \big) 
		\subset \Pi_G\big( 
		\graph u \cap 
		B_{\frac {\sqrt {1 + (\lip u)^2}r}{1-\sigma}}(0)\big),
		\end{align*}
		because $
		 \vert x +u(x)\vert ^2 = \vert x \vert^2+\vert u(x)\vert^2 \leq \li 1 + (\lip u)^2\ri \vert x \vert^2 \Foa x \in F. $
			\end{proof}
\begin{lem}[Intersecting Lipschitz Graphs] \label{lem:intersect_lip_graphs}
	Assume
	\begin{align} \label{lip_intersection_cond}
	0\leq \sigma < \chi/ 8 <  1/ 8.
	\end{align}
	Then for any two $m$-planes $F,G \in \mathscr{G}(n,m)$ with $\va\li F,G\ri\geq \chi$ and for functions $f \in C^{0,1}(F,F^\perp)$, $g\in C^{0,1}(G,G^\perp)$ satisfying $f(0)=0=g(0)$, $\lip f\leq \sigma $, $\lip g\leq \sigma$, the intersection of their graphs is contained in the graph of a 
	Lipschitz function with an at most $(m-1)$-dimensional domain. More precisely, there exists a $j$-plane $ X \in \mathscr{G}(n,j)$ for some $0\le
	j\leq m-1$, and a Lipschitz function $S\in C^{0,1}( X, X^\perp)$, such that the intersection $
	\graph f \cap \graph g$ is contained in $ \graph S.$
	In particular, 
	$\dim_{\HM}\li \graph f \cap \graph g \ri \leq j \leq m-1.$
\end{lem}

\begin{proof}
We may assume that $\sigma>0$, otherwise both graphs coincide with 
their $m$-planes of definition whose intersection is  lower-dimensional.
	Let $\{x_1,\dots, x_m\}\subset F$ and 
	$\{y_1,\dots,y_m\}\subset G$ be two sets
	of  
	orthonormal 
 principal vectors satisfying \eqref{princ_vectors}. Then
	by Lemma \ref{lem:princ_angles} we find a 
	minimal $j \in \{0,\dots,m-1\}$ such that $\sin\vartheta_k\ge\chi$ for all 
	$k \ge j+1$.
	Here, $\vartheta_k$ denotes the $k$-th principal angle as defined 
	in 
	Definition \ref{def:princ_angles}. Then, for any 
	$v \in \operatorname{span}\{x_{j+1},\dots,x_m\}$, i.e.,
	$v=\sum\nolimits_{i=j+1}^m a_ix_i$, we obtain
	\begin{align*} 
	\left \vert \Pi_G(v)\right \vert ^2 &=\textstyle\big \vert \sum\nolimits_{k=1}^m \langle v,y_k
	\rangle y_k\big \vert^2 = \sum\nolimits_{k=j+1}^m\vert a_k\vert^2 \cdot \left \vert \langle x_k,y_k\rangle \right \vert^2 =\sum\nolimits_{k=j+1}^m \vert a_k\vert^2\cos^2\vartheta_k \\
	&\leq \textstyle\sum\nolimits_{k=j+1}^m \vert a_k\vert^2\cos^2\vartheta_{j+1}
	= \vert v \vert^2 \cos^2\vartheta_{j+1},
	\end{align*}
where we used the monotonicity of $\vartheta_k$ and \eqref{princ_vectors}.
	Consequently,
	\begin{align} \label{intersect_big_proj}
	\left \vert \Pi_{G^\perp}(v)\right \vert^2 =
	\vert v\vert^2 - \left \vert \Pi_G(v)\right \vert ^2 \geq \vert v\vert^{2} 
	\li 1 - \cos^2\vartheta_{j+1}\ri = \vert v \vert^{2} 
	\sin^2\vartheta_{j+1}\geq \vert v \vert^2 \cdot \chi^2.
	\end{align}
	 First, we investigate the case $j=0$ which can only occur if $F\cap G=\{0\}$ and, therefore, $n\geq 2m$. 
	Since $f(0)=0=g(0)$ the origin is contained in $\graph f\cap
	\graph g$, and we claim that this intersection contains no other point,
	thus proving the lemma in this simple situation. Assume contrariwise that
	there is a point $q\in\graph f\cap\graph g\setminus\{0\}$. Then the 
	the Lipschitz continuity of $f$ implies 
	$ \vert\Pi_{F^\perp}(q)\vert\leq \sigma \vert \Pi_F(q)\vert$,
	so that $\vert q\vert \leq  (1+\sigma)\vert \Pi_F(q)\vert$, from which
	we  infer by means of  \eqref{intersect_big_proj} applied to $v:=\Pi_F(q)$
	\begin{align}\label{intersect_0_case1}
	\vert \Pi_{G^\perp}(q)  \vert &\geq  \vert \Pi_{G^\perp}\li \Pi_F(q)\ri  
	\vert -  \vert \Pi_{G^\perp}\li \Pi_{F^\perp}(q)\ri\vert
	\geq \chi  \vert \Pi_F(q) \vert - \vert \Pi_{F^\perp}(q) \vert\notag\\
	& \geq \li \chi-\sigma \ri  \vert \Pi_F(q) \vert\geq  
	{(\chi-\sigma)}{(1+\sigma)^{-1}}\vert q\vert.
	\end{align}
	The Lipschitz continuity of $g$, on the other hand,  yields
	$ \vert \Pi_{G^\perp}(q) \vert \leq \sigma  \vert \Pi_G(q) \vert 
	\leq \sigma \vert q\vert$, which can be combined with 
	\eqref{intersect_0_case1} to find
	$ {(\chi-\sigma)}{(1+\sigma)^{-1}} \leq \sigma$ contradicting 
	\eqref{lip_intersection_cond}. 
	
	If $j>0$, then we define  $
	Z:=F\cap G=\operatorname{span}\{x_1,\dots,x_i\},$  where we allow $i=0$ if
	$F\cap G=\{0\},$ $ Y:= \operatorname{span}\{x_{j+1},\dots,x_m\}, $  $
	W:=F \cap \operatorname{span}\left\{ Z,Y \right\}^\perp = \operatorname\{x_{i+1},\dots,x_{j}\},$ and $  X:=F\cap Y^\perp =\operatorname{span}\{x_1,\dots,x_j\}
	$.
	We claim that for all $q_1,q_2\in \graph f \cap \graph g$, and 
	$C:= 5 {(\chi-8\sigma)^{-1}}$, we have
	\begin{align} \label{intersect_inequ}
	\vert \Pi_Y(q_2-q_1) \vert \leq C \vert \Pi_{ X}(q_2-q_1)\vert.
	\end{align}
	Assuming the contrary, one finds $q_1,q_2\in \graph f\cap \graph g$ with
	\begin{align}\label{intersect_inequ_contrary}
	\vert \Pi_{ X}(q_2-q_1) \vert < \vert \Pi_Y(q_2-q_1) \vert /C.
	\end{align}
	As in the first case, the Lipschitz continuity of $f$ and $g$ yields
	\begin{align} 
	\left \vert \Pi_{F^\perp}(q_2-q_1)\right \vert &\leq \sigma \left \vert \Pi_F(q_2-q_1)\right \vert \leq \sigma\vert q_2-q_1\vert ,\label{intersect_lip_f}\\
	\left \vert \Pi_{G^\perp}(q_2-q_1)\right \vert &\leq \sigma \left \vert \Pi_G(q_2-q_1)\right \vert \leq \sigma\vert q_2-q_1\vert.\label{intersect_lip_g}
	\end{align} 
	Since $\Id_{\R^n}=\Pi_{ X}+\Pi_Y+\Pi_{F^\perp}$ and 
	$\Pi_F=\Pi_{ X}+\Pi_Y$, \eqref{intersect_lip_g} and 
	\eqref{intersect_inequ_contrary} guarantee
	\begin{align} \label{intersect_upper_bound}
	\vert q_2-q_1\vert &=  \vert \li \Pi_{ X}+\Pi_Y + \Pi_{F^\perp}\ri (q_2-q_1)
	\vert  \leq \li 1+\sigma\ri \li  \vert \Pi_{ X}(q_2-q_1) \vert+ \vert \Pi_{Y}(q_2-q_1) \vert\ri\notag\\
	&<\li 1+\sigma\ri \li 1 + 1/ C\ri  \vert \Pi_{Y}(q_2-q_1) \vert.
	\end{align}
	Setting $p_i:= \Pi_{ X^\perp}(q_i)$ for $i=1,2$, we compute $
	\Pi_{G^\perp}(q_2-q_1)= 
	\Pi_{G^\perp} \li \Pi_{ X}(q_2-q_1)\ri + \Pi_{G^\perp}(p_2-p_1).$
	Consequently, \eqref{intersect_inequ_contrary}  and \eqref{intersect_lip_g} imply 
	\begin{align}  \label{intersect_proj_p_1}
	\vert \Pi_{G^\perp}(p_2-p_1) \vert \leq  \vert \Pi_{G^\perp}(q_2-q_1) \vert +  \vert \Pi_{G^\perp} ( \Pi_{ X}(q_2-q_1) )  \vert <  ( \sigma + 1/ C ) \vert q_2-q_1\vert.
	\end{align}
	On the other hand,  since $\Pi_{ X^\perp}=\Pi_Y+\Pi_{F^\perp}$, \eqref{intersect_big_proj}, \eqref{intersect_lip_f}, and \eqref{intersect_upper_bound} guarantee
	\begin{align} \label{intersect_proj_p_2}
	\left \vert \Pi_{G^\perp} (p_2-p_1)\right \vert &\geq \left \vert \Pi_{G^\perp}\li\Pi_Y(q_2-q_1)\ri \right \vert - \left \vert \Pi_{G^\perp}\li \Pi_{F^\perp}(q_2-q_1)\ri\right \vert  \\
	&\hspace{-3.5cm}\geq \chi \left \vert \Pi_Y(q_2-q_1)\right\vert - \left \vert \Pi_{F^\perp}(q_2-q_1)\right \vert >\li  \chi { \li 1+\sigma \ri^{-1} \li 1+ 1/ C\ri^{-1}}-\sigma\ri \vert q_2-q_1\vert.\notag
	\end{align} 
	Combining \eqref{intersect_proj_p_1} and \eqref{intersect_proj_p_2} yields $
	\chi { \li 1+\sigma \ri^{-1}\li 1+ 1/ C\ri^{-1}}-\sigma < \sigma + 1/ C,$
	contradicting \eqref{lip_intersection_cond}. Hence, \eqref{intersect_inequ} holds true. Therefore, for $q_1,q_2\in\graph f \cap \graph g$ with $\Pi_{ X}(q_1)=\Pi_{ X}(q_2)$, one finds $ \vert \Pi_Y(q_2-q_1) \vert \leq  C\vert \Pi_{ X}(q_2-q_1) 
	\vert =0$. Due to \eqref{intersect_lip_f} and $\Pi_F=\Pi_{ X}+\Pi_Y$, 
	we obtain $ \vert \Pi_{F^\perp}(q_2-q_1) \vert \leq \sigma  \vert 
	\Pi_F(q_2-q_1) \vert =0$. 
	Consequently, $q_1=q_2$, i.e., for all $ x\in  X$, there exists
	at most one $q_{ x}\in\graph f \cap \graph g$ such that 
	$\Pi_{ X}(q_{ x})= x.$
	To define the map $S$ set $M:=\{ x \in X \ \colon \ \graph f \cap \graph g \cap (  x + X^\perp ) \neq \emptyset\}$. Then, 
	\begin{align} \label{intersect_set_M}
	\graph f \cap \graph g \cap (  x+ X^\perp ) = \left \{q_{ x}\right\} \Foa  x \in M,
	\end{align}
	and we obtain the well-defined map $
	S_M\colon M\to X^\perp, \  x \mapsto q_{ x}- x.$
	By \eqref{intersect_set_M} one finds
	\begin{align}\label{intersect_graph_inclusion}
	\graph f \cap \graph g &=\textstyle \bigcup\nolimits_{ x \in  X}\big( \graph f \cap \graph g 
	\cap (  x + X^\perp ) \big) \\
	&\hspace{-2cm}= \textstyle
	 \bigcup\nolimits_{ x\in M}\big( \graph f \cap \graph g \cap (  x +
	X^\perp ) \big)= \bigcup\nolimits_{ x\in M} q_{ x} =\graph S_M.\notag
	\end{align}
	Moreover, for $ \xi_1, \xi_2\in M$, we can use the fact that
	$q_{ \xi_i}- \xi_i\in  X^\perp$ and $\xi_i\in  X$ for 
	$i=1,2,$ to write
	$
	\vert S_M(\xi_1)-S_M(\xi_2) \vert = \vert q_{\xi_1}-\xi_1 - 
	(q_{\xi_2}-\xi_2) \vert , $ which equals $
	\vert \Pi_{ X}(q_{\xi_1}-\xi_1) + \Pi_{ X^\perp}
	(q_{\xi_1}-\xi_1) -\Pi_{ X}(q_{\xi_2}-\xi_2) - 
	\Pi_{ X^\perp}(q_{\xi_2}-\xi_2)  \vert	= 
	\vert \Pi_{X^\perp}(q_{\xi_1}-q_{\xi_2}) \vert.
	$
	This expression can be bounded from above by $
	\vert \Pi_{Y}(q_{\xi_1}-q_{\xi_2}) \vert +  
	\vert \Pi_{F^\perp}(q_{\xi_1}-q_{\xi_2} )\vert$ since 
	$\Pi_{ X^\perp }=
	\Pi_Y+\Pi_{F^\perp}$.
	With \eqref{intersect_lip_f} and $\Pi_{F}=\Pi_{ X}+\Pi_{Y}$ we can 
	compute $\vert \Pi_{F^\perp}(q_{\xi_1}-q_{\xi_2}) \vert \leq 
	\sigma ( \vert \Pi_{ X}(q_{\xi_1}-q_{\xi_2}) \vert + 
	\vert \Pi_Y(q_{\xi_1}-q_{\xi_2}) \vert )$. Consequently, 
	$	 
	\vert S_M(\xi_1)-S_M(\xi_2) \vert $ is bounded from above by $ \li \li 1 + \sigma \ri C+\sigma \ri \vert \Pi_{ X}(q_{\xi_1}-q_{\xi_2}) \vert,
	$
	where we additionally applied \eqref{intersect_inequ}. Finally, using 
	$q_{\xi_i}-\xi_i\in  X^\perp$  to derive the
	identity $
	\vert \Pi_{ X}(q_{\xi_1}-q_{\xi_2}) \vert = 
	\vert \xi_1-\xi_2 + \Pi_{ X}(q_{\xi_1}-\xi_1) -\Pi_{X}(q_{\xi_2}-\xi_2) 
	\vert =\vert \xi_1-\xi_2\vert,   $
	we obtain
	\begin{align}\label{intersect_S_M_lip}
	\left \vert S_M(\xi_1)-S_M(\xi_2)\right \vert \leq \li \li 1 + \sigma \ri C+\sigma \ri \vert \xi_1-\xi_2\vert.
	\end{align}
	Hence, $S_M$ is Lipschitz continuous with 
	$\lip S_M \leq \li 1 + \sigma \ri C+\sigma$. 
	Due to Kirszbraun's theorem \cite[2.10.43]{federer_1969} 
	we can find a Lipschitz continuous extension $S \colon 
	X \to  X^\perp$ with $\lip S = \lip S_M$. Hence, we have
	\begin{align}
	\graph f \cap \graph g =\graph S_M\subset \graph S.
	\end{align}
	In particular, by virtue of \cite[2.4.2 Thm. 2 (ii)]{evans-gariepy_2015} 
	and $\dim  X=j$, we finally conclude $
	\dim_{\HM}\li \graph f \cap \graph g \ri \leq \dim_{\HM}(\graph S)=j.$
\end{proof}

  \section*{Acknowledgments}

 The 
 first author was partially supported by NCN Grant no.\
 2013/10/M/ST1/00416 \emph{Geometric curvature energies for subsets of the
Euclidean space.} The second author's work is partially funded by DFG Grant no.
Mo 966/7-1 \emph{Geometric curvature functionals: energy landscape and discrete
methods} and by the
Excellence Initiative of the German federal and state governments.

\bibliography{../bib-files/refs-bookproj}{}
         \bibliographystyle{acm}

\end{document}